\documentclass[leqno, 12pt]{amsart}
\usepackage{enumerate,tikz, amscd, mathtools, quiver, float, hyperref}
\usepackage{mathrsfs} 
\hypersetup{colorlinks,linkcolor={blue},citecolor={blue},urlcolor={cyan}} 
\usepackage[tableposition=top]{caption}
	
\newtheorem{theorem}{Theorem}[section]
\newtheorem{proposition}[theorem]{Proposition}
\newtheorem{lemma}[theorem]{Lemma}

\newtheorem{corollary}[theorem]{Corollary}

\newtheorem{example}[theorem]{Example}

\newtheorem{conj}[theorem]{Conjecture}

\newtheorem{notation}[theorem]{Notation}

\theoremstyle{definition}

\newtheorem{definition}[theorem]{Definition}

\newtheorem{remark}[theorem]{Remark}

\DeclareMathOperator{\Hom}{Hom}

\DeclareMathOperator{\Mod}{-mod}
\DeclareMathOperator{\End}{End}
\DeclareMathOperator{\Ent}{Ent}
\DeclareMathOperator{\Exit}{Exit}

\DeclareMathOperator{\Path}{\textbf{Path}}
\DeclareMathOperator{\Cell}{\textbf{Cell}}
\DeclareMathOperator{\im}{Im}

\newcommand\floor[1]{\lfloor #1 \rfloor}
\newcommand{\newterm}{\textsf}
\newcommand{\op}[1]{\operatorname{#1}}


\title{Homotopy path algebras}
\author[Favero]{David Favero}
\address{
  \begin{tabular}{l}
   David Favero \\
   \hspace{.1in} University of Minnesota, School of Mathematics \\
      \hspace{.1in} 206 Church Street, Minneapolis MN 55455, USA \\
         \hspace{.1in} Email: {\bf favero@umn.edu} \\
  \end{tabular}
}
\author[Huang]{Jesse Huang}
\address{
  \begin{tabular}{l}
   Jesse Huang \\
   \hspace{.1in} University of Waterloo, Department of Pure Mathematics \\
   \hspace{.1in} 200 University Avenue West, Waterloo ON, Canada  N2L3G1 \\
   \hspace{.1in} Email: {\bf j654huang@uwaterloo.ca} \\
  \end{tabular}
}

\begin{document}

\begin{abstract}
We define a basic class of algebras which we call homotopy path algebras.  We find that such algebras always admit a cellular resolution and detail the intimate relationship between these algebras, stratifications of topological spaces, and entrance/exit paths.  As examples, we prove versions of homological mirror symmetry due to Bondal-Ruan for toric varieties and due to Berglund-H\"ubsch-Krawitz for hypersurfaces with maximal symmetry.  We also demonstrate that a form of shellability implies Koszulity and the existence of a minimal cellular resolution.  In particular, when the algebra determined by the image of the toric Frobenius morphism is directable, then it is Koszul and  admits a minimal cellular resolution. 
\end{abstract}

\maketitle

\tableofcontents

\section{Introduction}
In 2006, Bondal and Ruan introduced a completely novel approach to homological mirror symmetry (HMS) \cite{Bon06}.  Their description connected sections of line bundles on toric varieties to entrance paths on a ``mirror'' torus.  The mantle for this approach was soon taken up by Fang-Liu-Treumann-Zaslow in a series of articles \cite{FLTZ11, FLTZ12, FLTZ14} which included proofs of HMS for smooth, proper toric varieties and maximally equivariant toric stacks. Many great works have followed which we will not attempt to catalogue, however, let us mention that a full proof of HMS for all toric DM stacks was provided by Kuwagaki \cite{Kuw20}. 

The work herein began as an attempt to codify the original ideas of Bondal-Ruan with the hope of extending their approach beyond the toric framework.  The resulting theory can be thought of as nothing more than the study of a simple class of $\mathsf k$-algebras which we would like to call ``homotopy path algebras'' (HPAs) (where $\mathsf k$ is \emph{any} unital commutative Noetherian base ring).  

The name HPA is descriptive.  If we embed a quiver in a topological space, then the corresponding HPA is just the path algebra of the quiver quotiented by the ideal generated by path homotopy.  On the other hand, these HPAs have a completely algebraic description (see Definition~\ref{defn: HPA}) and hence can be studied as a class of algebras in their own right.  Alternatively, bases of HPAs also arise exactly as entrance paths for a topological space $Y$ stratified by a poset $S$ of contractible strata: picture the vertices as chosen base points in $S$ and arrows as a basis of homotopy classes of indecomposable directed paths in $Y$.  

The study of entrance paths on a space stratified by a poset of subspaces really goes back to MacPherson, who famously conjectured an equivalence between constructible (co)sheaves and representations of entrance paths.  The strongest incarnation of this conjecture was handled by Curry-Patel \cite{CP16} for conically stratified spaces.  A higher categorical version is given in Lurie's Derived Algebraic Geometry \cite{Lur09}, Appendix A, (see also \cite{Tr09}). 

Our motivating examples were both the stratification of a torus considered by Bondal-Ruan and the tree stratification (see Definition~\ref{def: tree strata}) on the classifying space of entrance paths. These stratifications are not conical. Instead, they are what we call \emph{block} stratifications (see Definition~\ref{def: block}), which are in particular \emph{simple} stratifications (see Definition~\ref{def: simple}). Our first result is a version of MacPherson's conjecture for simple stratifications. 
% Briefly, this means the strata are a coarsening of a CW stratification and there is a partial ordering induced by the intersection of their closures (see Definition~\ref{} for details).  
\begin{theorem}
Let $S$ be a simple stratification of $Y$.  The following categories are equivalent:
\begin{itemize}
    \item $Sh_S(Y)$, the category of $S$-constructible sheaves on $Y$;
    \item $A\Mod$, the category of modules over the HPA described above;
    \item $Fun(\Ent_S(Y), \mathsf k\Mod)^{\op{op}}$, the opposite category of representations of $S$-entrance paths.
\end{itemize}
\end{theorem}

Entrance paths themselves may be thought of as a type of directed fundamental group or as similar to flow lines in Morse theory. When $S$ is sufficiently nice, entrance paths should recover the homotopy type of $Y$, reminiscent of the seminal work \cite{CJS95} where gradient flow lines of a Morse function are used to reconstruct the homotopy type of a smooth manifold. This direction has recently been explored by Nanda \cite{Nan19} in the context of discrete Morse theory \cite{For98}, who proves the classifying space of the discrete Morse flow category recovers the homotopy type of a regular CW complex.  Following Nanda's approach, we prove:

\begin{theorem}
Let $S$ be a block stratification of $Y$.  Then $Y$ is homotopic to the classifying space of $S$-entrance paths.  That is, the homotopy type of $Y$ is encoded in the category of $S$-entrance paths.  \end{theorem}

The topological take on HPAs also provides some pleasant  algebraic consequences. For example, Table~\ref{tab: section4 summary} relates fundamental $A$-modules to $S$-constructible sheaves described by strata, exit paths, and entrance paths. Furthermore, projective resolutions of $A$-modules have topological interpretations similar to CW and Morse homology  as in the following theorem.
\begin{theorem}
Let $A$ be any HPA.  The geometric realization of $A$ (see Definition~\ref{def: geometric realization}) $X_A$ together with the tree stratification $S^{tr}$ recovers $A$ as the algebra of $S^{tr}$-entrance paths.  In addition, $X_A$ is a projective cellular resolution (see Definition~\ref{def: cellular res}) of the diagonal bimodule.  Furthermore, for any internal acyclic matching (see Definition~\ref{def: internal}), the simplicial collapse $\overline{X_A}$ is a projective cellular subresolution.  Finally, if $X_A$ is shellable in a certain sense (see Proposition~\ref{prop: shellable-koszul}) then $A$ is Koszul, and there exists an internal acyclic matching such that the projective cellular resolution $\overline{X_A}$ is minimal.
\end{theorem}

For quiver algebras with relations, the existence and uniqueness of a minimal resolution of the diagonal bimodule is well-known \cite{Ei56, BK99} (though not always explicit). Explicit minimal resolutions of the diagonal bimodule are of both independent interest and structural importance.  They give interesting numerical invariants called Betti-numbers, allow for efficient calculations of Hochschild homology and cohomology, and provide functorial resolutions of all modules. The close similarity to CW homology, called cellular resolutions, is more recent (see \cite{BS98}).  

Our approach to constructing minimal cellular resolutions is similar to  e.g.\ \cite{JW09}.  We use discrete Morse theory to construct minimal cellular resolutions analogous to the discrete Morse complex of Forman \cite{For98}. We comment that the resulting resolutions are closely related to the minimal cellular bimodule resolutions of abelian skew group algebras provided by Craw-Quintero V\'{e}lez \cite{CQV12}. While their construction uses a different approach, the algebras they consider are nothing more than HPAs which admit cycles from our perspective.

Stringing our results together, we now return to the inspiration for this work: Bondal-Ruan mirror symmetry.  As a consequence of the theorems above we recover, generalize, and extend Bondal-Ruan's original result:
\begin{theorem}[Bondal-Ruan Homological Mirror Symmetry]
Let $\mathcal X$ be a toric DM stack of Bondal-Ruan type (see Definition~\ref{def: BR type}).  There is a stratification $S$ of the torus $\mathbb T^n$ and an associated HPA, $A$, such that the following derived categories are equivalent
\[
D(\mathcal X) \cong D(A) \cong DSh_{S^{tr}}(X_A) \cong DSh_S(\mathbb T^n).
\]
Furthermore, $X_A$ is homotopic to a torus and when $A$ is directable, it is Koszul and there is an explicit minimal resolution of the diagonal bimodule $A$ which corresponds to a resolution of the diagonal $\Delta_{\mathcal X}$ by line bundles.
\end{theorem}

We conclude by noting that our approach may, in some cases, be applicable to homological mirror symmetry beyond the toric context. For example, our methods can be deployed on Berglund-H\"ubsch-Krawitz hypersurfaces with maximal symmetry (see Example~\ref{ex: BHK}). 

% In the smooth geometric setting, the correspondence above is also closely tied to the symplectic topology of cotangent bundles of appropriately stratified smooth manifolds studied in Nadler-Zaslow and Nadler \cite{N09, NZ09, N09} and Ganatra-Pardon-Shende \cite{GPS18}, paving the way for a homotopical interpretation of wrapped Fukaya categories.

\section*{Acknowledgement}
We are grateful to Alexey Bondal for suggesting we explore the inspirational Oberwolfach Report \cite{Bon06}.  This is the foundation for many of the results found herein.  We also thank Matt Ballard and Alex Duncan for kindly answering our questions on their talk on \cite{BDM} and useful additional discussions at the Banff International Research Station. 

For the revision of our preprint, we thank the referee for providing useful comments; Peter Haine for their interest in our exodromy result and suggestions on using poset stratifications; Pouya Layeghi and Mykola Sapronov for carefully reading our paper and providing feedback.

J.\ Huang is supported by a Pacific Institute for the Mathematical Sciences Postdoctoral Fellowship.  D.\ Favero is supported by NSERC through the Discovery Grant and Canada Research Chair programs.  

\section{Notation and conventions}

\subsection{The base ring}
We work over an arbitrary unital commutative Noetherian base ring $\mathsf k$ e.g. take $\mathsf k = \mathbb Z$ or $\mathbb C$. 

\subsection{Posets}
For a finite nonempty poset $P$, we denote the unbounded order complex of $P$ by $K(P)$, whose $k$-cells are strictly increasing chains of $k+1$ elements in $P$.

\subsection{Stratification}
In this paper, stratification on a topological space always means \newterm{stratification by a poset} (Definition \ref{defn: poset stratification}).

\subsection{Quivers}
A quiver is denoted by $Q$ and its path algebra by $\mathsf kQ$.  The vertex set of $Q$ is denoted by $Q_0$ and the set of arrows is denoted by $Q_1$.
We also want to emphasize our convention, which may not be standard: the direction of the arrows and their multiplication in the quiver agrees with the entrance path direction and the concatenation of entrance paths. We use \emph{increasing} order of the strata to define entrance paths. For a regular CW complex with cell strata, it means our indexing poset is the usual face poset with opposite ordering.

\subsection{CW complexes}
In this paper, a \newterm{CW complex} is a Hausdorff topological space $X$ filtered by a sequence of closed subspaces
\[
X=X_n\supset X_{n-1}\supset \cdots \supset X_0 \supset X_{-1}=\emptyset,
\]
where $X_k=\coprod_{\alpha\in I, i\leq k} B_\alpha^{i}$ with $B_i^n$ homeomorphic to $\mathbb{R}^n$ and $\overline{B}_i^n\setminus B_i^n\subset X_{n-1}$, satisfying the CW axioms
\begin{itemize}
    \item $\{B_j^k : B_j^k\cap \overline{B}_i^n$ \text{ is finite for all $i\in I$ and $n\geq 0 \}$}.
    \item $A\subset X$ is closed iff $A\cap \overline{B}_i^n$ is closed for all $i
    \in I$ and $n\geq 0$.
\end{itemize}
\emph{Importantly}, all CW complexes we consider satisfy the \newterm{axiom of frontier}:
\begin{itemize}
    \item If $B_i^k\cap \overline{B}_j^n\neq\emptyset$, then $B_i^k\subset \overline{B}_j^n$.
\end{itemize}
A CW complex is \newterm{finite} if the index set $I$ is finite . A CW complex is \newterm{regular} if in addition $\overline{B}_i^n$ is homeomorphic to the closed unit ball in $\mathbb{R}^n$. 

A \newterm{semi-simplicial complex} (or $\Delta$-complex \cite{Hat02}) is the quotient space
\[
X=\coprod_{k,\alpha}\Delta_{k,\alpha}/\sim
\]
where $\Delta_{k,\alpha}$ are closed simplices, and $\sim$ identifies in each dimension a set of faces in that dimension from distinct simplices. By construction, a semi-simplicial complex $X$ is a CW complex where each closed $k$-cell $\overline{B}^k_i$ is homeomorphic to a closed $k$-simplex with a set of faces identified. A semi-simplicial complex is \newterm{regular} if every $\overline{B}^k_i$ is homeomorphic to a $k$-simplex. A regular semi-simplicial complex is \newterm{simplicial} if any finite set of $0$-cells of cardinality $k+1$ is the vertex set of a unique $k$-cell.

For classifying spaces, we follow the convention in \cite{Qui73}.  Namely, we define the \newterm{classifying space} $B\mathcal{C}$ of a category $\mathcal{C}$ to be the geometric realization of its nerve $N(\mathcal{C})$.  By \newterm{geometric realization} of a (semi-)simplicial set, we mean the CW complex with one $n$-cell for each non-degenerate $n$-simplex
as in \cite{Mil57}.

\subsection{Categories}
All categories considered are small categories. All categories and functors are equipped with the derived notation if derived and usual notation if not.  For functors which are already exact (e.g. the proper pushforward for a locally-closed embedding) we do not use the derived notation.

For constructible sheaves, $Sh_S(X)$ is the abelian category of $S$-constructible sheaves on $X$, and $DSh_S(X)$ is the bounded derived category of $Sh_S(X)$.  For a continuous map $f: X \to Y$, the right adjoint to the derived proper pushforward in $DSh_S(X)$ is denoted by $f^!$ since this functor is taken to always live in the derived category.

For a finite dimensional $\mathsf{k}$-algebra $A$, $A\Mod$ means the category of finitely generated left $A$-modules, and $D(A)$ means the bounded derived category of finitely generated left $A$-modules.

For an algebraic stack $\mathfrak X$, $D(\mathfrak X)$ denotes the bounded derived category of coherent sheaves on $\mathfrak X$.

\section{Homotopy Path Algebras}

% We define
% \[
% X_A := B\mathcal C_A.
% \]

% Concrete description of $X_A$

% Examples: square, $\mathbb P^2$, products, no relations, $A_{k_1} \otimes ... A_{k_n}$

% \begin{theorem}
% The functor
% \begin{align*}
% F: A-\mathop{mod} & \to Sh(X_A) \\
% P_v & \mapsto T_v
% \end{align*}
% is fully-faithful, inducing an equivalence of categories $A-mod \cong \langle T_v \rangle.$
% \end{theorem}
% \begin{proof}
% \end{proof}

\subsection{HPAs and classifying spaces}
Consider a finite acyclic quiver $Q$. By a \newterm{path} in $Q$, we mean concatenation of a sequence of arrows and constant paths at the vertices. Let $\Path_Q$ be the set of all paths in $Q$ i.e.\ the standard $\mathsf k$-basis of the finite dimensional path algebra $\mathsf kQ$, equipped with the partial order
\begin{equation}\label{eq: path poset}
   p\leq q \Leftrightarrow \exists r\in \Path_Q \text{\ s.t.\ } q=pr \in \mathsf kQ
\end{equation}
For a path $p$ in the quiver, let $h(p)$ and $t(p)$ denote the head of $p$ and the tail of $p$. 

Given any subset $S\subset \Path_Q$, we can form the following two-sided ideal 
\begin{equation}
    I_S:=\langle p-q: p, q\in S,\ h(p)=h(q) \text{ and } t(p)=t(q) \rangle\subset \mathsf kQ.
\end{equation}

% For each pair of vertices $(v, v')$, choose a (possibly empty) collection $C_{v, v'}$ of disjoint subsets $S^\alpha_{v,v'}$, $\alpha \in C_{v,v'}$ of paths from $v$ to $v'$, and let $S=\coprod_{v,v'\in Q_0, \alpha \in C_{v,v'}} S^\alpha_{v, v'}$. From $S$ we can form the following ideal
% \begin{equation}
%     I_S:=\langle p-p': p, p'\in S^\alpha_{t(p), h(p)} \text{ for some } \alpha \in C_{t(p),h(p)}\rangle\subset \mathsf kQ.
% \end{equation}
\begin{definition}\label{defn: HPA}
The quiver algebra $A=\mathsf kQ/I_S$ is a \newterm{homotopy path algebra (HPA)} if $I_S$ is left and right cancellative, which means:
\begin{itemize}
    \item For any paths $r,p,q$ and nonzero $rp-rq\in I_S$, $p-q\in I_S$.
    \item For any paths $r,p,q$ and nonzero $pr-qr\in I_S$, $p-q\in I_S$.
\end{itemize}
\end{definition}

\begin{example}
The path algebra of the quiver below, with relations $I=\langle a_1b_1-a_2b_2\rangle$, is a homotopy path algebra.
\[\begin{tikzcd}
	&&& {} \\
	\bullet & \bullet & \bullet
	\arrow["{b_1}", curve={height=-12pt}, from=2-2, to=2-3]
	\arrow["{b_2}"', curve={height=12pt}, from=2-2, to=2-3]
	\arrow["{a_2}"', curve={height=12pt}, from=2-1, to=2-2]
	\arrow["{a_1}", curve={height=-12pt}, from=2-1, to=2-2]
\end{tikzcd}\]
\end{example}

\begin{example}
The path algebras of the two quivers below, with relations $\langle ab-ac \rangle$ (left) and $\langle ac-bc \rangle$ (right), are not homotopy path algebras. The quiver on the left fails the first condition, and the quiver on the right fails the second condition.
\[
    \begin{tikzcd}
	&&&& {} \\
	\bullet & \bullet & \bullet && \bullet & \bullet & \bullet
	\arrow["a", from=2-1, to=2-2]
	\arrow["b", curve={height=-12pt}, from=2-2, to=2-3]
	\arrow["c", curve={height=12pt}, from=2-2, to=2-3]
	\arrow["a", curve={height=-12pt}, from=2-5, to=2-6]
	\arrow["b"', curve={height=12pt}, from=2-5, to=2-6]
	\arrow["c", from=2-6, to=2-7]
    \end{tikzcd}
\]
\end{example}

\begin{definition}
    The \newterm{path poset} of an HPA $A=\mathsf kQ/I_S$ is defined to be
\[
\Path_A := \Path_Q / \sim
\]
where $\sim$ is the equivalence relation given by
$$p\sim q \Leftrightarrow p-q \in I_S$$
\end{definition}
\begin{notation}
    We use $[p]$ to denote an elements in $\Path_A$ i.e. the class of path $p\in \Path_Q$ under this equivalence relation.
\end{notation}

We notice that the partial order we give to $\Path_A$ suggests that if $p\leq q$, then $t(p)=t(q)$. In other words, two paths with distinct starting points are incomparable under $\leq$. One can then write the poset $\Path_A$ as a disjoint union of subposets indexed by $Q_0$ whose elements are incomparable. Namely,
\begin{align*}
\Path_A & = \coprod_{v\in Q_0} \Path_{A,v},\\
\text{where } \Path_{A,v} & := \{ p: t(p)=v \}\subset \Path_A \text{ as a subposet}.
\end{align*}
As a result, the order complex $K(\Path_A)$ has $|Q_0|$ connected components
\[
    K(\Path_A) = \coprod_{v\in Q_0} K(\Path_{A,v}).
\]
\begin{example}\label{P2}
Consider the HPA $A=\mathsf kQ/I$, where
\[Q=\begin{tikzcd}
	{{e_0\ \bullet}} & \textcolor{rgb,255:red,153;green,92;blue,214}{{e_1\ \bullet}} & \textcolor{rgb,255:red,92;green,214;blue,214}{{e_2\ \bullet}}
	\arrow["x", color={rgb,255:red,65;green,186;blue,54}, curve={height=-12pt}, from=1-1, to=1-2]
	\arrow["{x'}", color={rgb,255:red,149;green,213;blue,179}, curve={height=-12pt}, from=1-2, to=1-3]
	\arrow["z", color={rgb,255:red,206;green,59;blue,59}, curve={height=12pt}, from=1-1, to=1-2]
	\arrow["{z'}", color={rgb,255:red,225;green,137;blue,137}, curve={height=12pt}, from=1-2, to=1-3]
	\arrow["y", color={rgb,255:red,63;green,80;blue,228}, from=1-1, to=1-2]
	\arrow["{y'}", color={rgb,255:red,117;green,149;blue,245}, from=1-2, to=1-3]
\end{tikzcd},\ I=\langle xy'-yx', xz'-zx', yz'-zy'\rangle \]

We can write down the elements living in each subposet:
\begin{align*}
    \Path_{A,e_0}&:=\{[e_0],[x],[y],[z], [xx'], [yy'], [zz'], [xy']=[yx'], [xz']=[zx'], [yz']=[zy']\} \\
    \Path_{A,e_1}&:=\{[e_1],[x'],[y'],[z']\} \\
    \Path_{A,e_2}&:=\{[e_2]\} \\
\end{align*}
The order complex $K(\Path_A)$ has three connected components: $K(\Path_{A,e_0})$ (left), $K(\Path_{A,e_1})$ (middle), and $K(\Path_{A,e_2})$ (right). 
\[
    \includegraphics[scale=0.6]{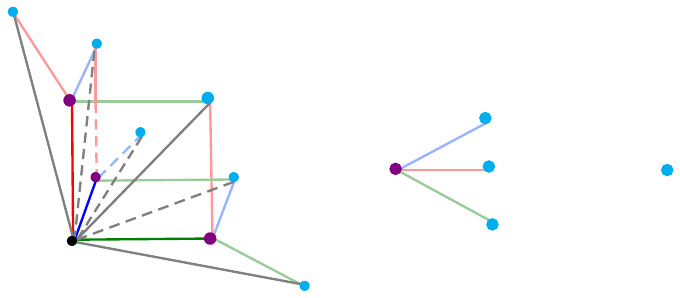}
\]
The $n$-simplices in $K(\Path_{A, e_i})$ are in 1-1 correspondence with increasing chains of elements in $\Path_A$ of length $n$ starting at $[e_i]$. For example, the left most 2-simplex in $K(\Path_{A,e_0})$ corresponds to the increasing chain $[e_0\leq z \leq zz']$.
\end{example}

Next, we discuss the geometric realization of an HPA.

\begin{definition} \label{def: geometric realization}
Given an HPA, $A = \mathsf kQ/I_S$, define its \newterm{path category} $\mathcal C_A$ as follows.
\begin{align*}
\text{Objects of }\mathcal C_A & := Q_0 \\
\op{Hom}_{\mathcal C_A}(v, v') & := \{ [p] \in  \Path_A : t(p) = v \text{ and } h(p) = v' \}
\end{align*}
with composition given by concatenation of paths. The \newterm{geometric realization} of $A$ is the classifying space of $\mathcal C_A$.  We denote this by
\[
X_A := B(\mathcal C_A)=B(\mathcal{C}_A^{op}).
\]
\end{definition}

We notice that the geometric realization of an HPA has an alternative description as a quotient space of $K(\Path_A)$ we have described. 

\begin{proposition}\label{prop: KPath to X_A}
    There is a continuous map 
    \begin{align*}
    f: K(\Path_A)\twoheadrightarrow X_A
    \end{align*}
    which sends each $n$-simplex $[p_0\leq\cdots \leq p_{n-1}]$ in $K(\Path_A)$ homeomorphically onto the $n$-simplex in $X_A$ corresponding to the composition of subquotient paths    
    in $\mathcal C_A$.
\end{proposition}
\begin{proof}
Viewing the poset $(\Path_A, \leq)$ itself as a category with morphisms given by $\leq$, there is a well-defined functor
    \begin{align*}
     F: \Path_A & \rightarrow \mathcal C_A \\
        [q] & \mapsto h(q)\\
        [p]\leq [q]=[pr]& \mapsto [r]   
    \end{align*}
    which induces a continuous map on the geometric realization
    $$
    f:=BF: K(\Path_A)=B(\Path_A) \rightarrow B(\mathcal C_A)=X_A
    $$
    Since $F$ sends a composition of $n$ morphisms in $\Path_A$, which is an increasing chain $[p_0\leq\cdots \leq p_{n-1}]$, to the composition of subquotient paths
    $$t(p_0) \xrightarrow{p_0} h(p_0)\xrightarrow{p_1/p_0} h(p_1)\xrightarrow{p_2/p_1}  \cdots  \xrightarrow{p_{n-2}/p_{n-3}} h(p_{n-2})\xrightarrow{p_{n-1}/p_{n-2}} h(p_{n-1})$$ in $\mathcal C_A$
    the induced map on the classifying space identifies simplices $[p_0 < \dots < p_n]$ with $[e_{t(p_0)}<p_1/p_0<\dots < p_n/p_0]$. In particular, it is surjective and has the claimed property.
\end{proof}

% \begin{proposition}
%  We have the following explicit realization of $X_A$:
% \begin{equation}
%     \begin{split}
%     & X_A=K(\Path_A)/\sim , \text{ where } \\
%     & [p_0< \dots < p_n] \sim [q_0 < \dots < q_n] \text{ iff } \frac{p_i}{p_0}-\frac{q_i}{q_0}\in I_S, \ \forall 1\leq i \leq n.
%     \end{split}
% \end{equation}
% \end{proposition}
% \begin{proof}
% Starting with the definition of the classifying space, we can rewrite
% \begin{align*}
%     X_A & =(\coprod [p_0<\cdots< p_n]\times \Delta_n) / \sim \\
%     & = (\coprod_{v\in Q_0}  (\coprod [e_v=p_0<\cdots p_n] \times \Delta_n)  / \sim) ) /\sim \\
%     & = (\coprod_{v\in Q_0} (K(\Path_{Q,e_v})/\sim ))/\sim \\
%     & = (\coprod_{v\in Q_0} K(\Path_{A,e_v}))/\sim \\
%     & = K(\Path_A)/\sim.
% \end{align*}

% \end{proof}

% \begin{remark}
% From now on, we will abuse the notation and use $[e_{t(p_0)}<p_1/p_0<\dots < p_n/p_0]$ to label simplices in $X_A$ for convenience.
% \end{remark}

We also have the following corollary:
\begin{corollary}
The geometric realization $X_A$ is a regular semi-simplicial complex.
\end{corollary}
\begin{proof}
By construction, $X_A$ is a semi-simplicial complex. Since the functor $F$ in Proposition \ref{prop: KPath to X_A} sends each ascending sequence of paths starting at a particular vertex to the composition of the subquotient paths, each closed simplex in $X_A$ must have distinct vertices i.e.\ every closed simplex is embedded homeomorphically into $X_A$. Hence $X_A$ is regular.
\end{proof}

% \begin{remark}
% Equivalently, one can enhance $Q$ to the category $\mathsf{Q}$ whose objects are vertices and morphisms are $I$-equivalence classes of paths between vertices, and define $X_A:=|N(\mathsf{Q)}|$.
% \end{remark}

We now justify the name ``homotopy path algebra'' by showing that the paths identified by the ideal $I_S$ are in fact homotopic in $X_A$. By construction, the underlying graph of the quiver $Q$ embeds into $X_A$ as a subspace. Any path in the quiver, viewed as a continuous map $p:[0,1]\rightarrow Q$, becomes a path in $X_A$.

\begin{remark}
In general, $Q$ is only a subspace of the 1-skeleton of $X_A$. However, the CW structure of $X_A$ can sometimes be reduced to a minimal one whose 1-skeleton is $Q$. This can be done using a Morse matching as discussed in \S\ref{sec: Morse}). 
\end{remark}

\begin{proposition}\label{prop: homotopy in X_A}
Let $A=\mathsf kQ/I$ be an HPA. Then, given two paths $p=a_1\cdots a_n,\ q= b_1\dots b_m \in \Path_Q$ with $a_i, b_j\in Q_1$, $p-q\in I$, the paths $p:[0,1]\rightarrow X_A$ and $q:[0,1]\rightarrow X_A$ corresponding to these concatenations of arrows are path homotopic. 
\end{proposition}
\begin{proof}
%This is evident from the fact that $X_A$ is the classifying space of $\mathcal C_A$. 
From the composition of morphisms $a_1 \circ \cdots a_n$, we obtain an $n$-simplex in $X_A$.  Inside this simplex, the path $p$ is homotopic to the edge of this simplex corresponding to the composition $a_1 \circ \cdots \circ a_n$.  Now by assumption $a_1 \circ \cdots \circ a_n = b_1 \circ \cdots \circ b_m$ in $\mathcal C_A$ and hence the corresponding edges are equal in $X_A$.
\end{proof}

Conversely, we can show
\begin{proposition} \label{prop: A(Y)}
    Given a topological space $Y$ and an embedding $Q \hookrightarrow Y$, the algebra $A(Y)=\mathsf{k}Q/I_Y$, where
\begin{equation*}
    I=\langle p-q: p \text{ is path homotopic to } q \text { in } Y \rangle,
\end{equation*}
is an HPA.
\end{proposition}
\begin{proof}
    Clearly, the pair of paths identified by the ideal have the same endpoints. Since path homotopy is left and right cancellative, the ideal is left and right cancellative. Therefore $A(Y)$ is an HPA by definition.
\end{proof}

This allows us to give an alternative, topological definition of an HPA

\begin{definition}\label{defn: HPA2}
    An HPA is algebra of the form $A(Y)$ as in Proposition \ref{prop: A(Y)} associated to a topological space $Y$ with an embedded quiver $Q\hookrightarrow Y$.
\end{definition}

Using Propositions \ref{prop: homotopy in X_A} and \ref{prop: A(Y)}, one can see that Definition \ref{defn: HPA2} and Definition \ref{defn: HPA} are equivalent definitions.

We now give some examples of HPAs and their classifying spaces, and implement in these simple examples a procedure known as Morse matching \cite{For98}, which allows us to produce a space homotopic to $X_A$ with a much simpler CW structure whose 2-skeleton is the quiver and generators of the HPA relations.

\subsection{Examples}
\begin{example}[No relation]
For any acyclic quiver $Q$ (not necessarily finite) without relations, one can construct a deformation retraction from $X_{\mathsf kQ}$ onto $Q$ by collapsing each simplex $[p_0<\dots<p_n]$ to the composition of arrows from $t(p_0)$ to $h(p_n)$.
\end{example}

\begin{example}[Example \ref{P2} continued] 
We return to the HPA
$A=\mathsf kQ/I$, where
\[Q=\begin{tikzcd}
	{{e_0\ \bullet}} & \textcolor{rgb,255:red,153;green,92;blue,214}{{e_1\ \bullet}} & \textcolor{rgb,255:red,92;green,214;blue,214}{{e_2\ \bullet}}
	\arrow["x", color={rgb,255:red,65;green,186;blue,54}, curve={height=-12pt}, from=1-1, to=1-2]
	\arrow["{x'}", color={rgb,255:red,149;green,213;blue,179}, curve={height=-12pt}, from=1-2, to=1-3]
	\arrow["z", color={rgb,255:red,206;green,59;blue,59}, curve={height=12pt}, from=1-1, to=1-2]
	\arrow["{z'}", color={rgb,255:red,225;green,137;blue,137}, curve={height=12pt}, from=1-2, to=1-3]
	\arrow["y", color={rgb,255:red,63;green,80;blue,228}, from=1-1, to=1-2]
	\arrow["{y'}", color={rgb,255:red,117;green,149;blue,245}, from=1-2, to=1-3]
\end{tikzcd},\ I=\langle xy'-yx', xz'-zx', yz'-zy'\rangle \]
The map $f$ in Proposition \ref{prop: KPath to X_A} identifies the vertices and edges of the same color and opacity (notice that $x,x'$ etc. have different opacity despite being of the same color) in the picture of $K(\Path_A)$ we described earlier in Example \ref{P2}.

The resulting simplicial complex is $X_A$. It is more challenging to draw a picture of $X_A$. To understand its topology, one can instead notice that $X_A$ is homotopic to the CW complex obtained from $X_A$ by a simplicial collapse of 2-cells with gray edges. At the level of $K(\Path_A)$, the simplicial collapse gives
\[
    \includegraphics[scale=0.6]{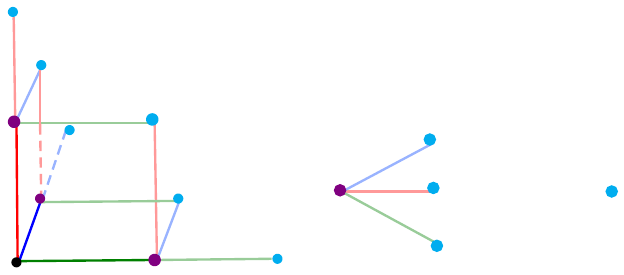}
\]
with gluing induced by the map $f$ on $K(\Path_A)$. The resulting CW complex can now be identified as a 2-torus. We draw it in three different fundamental domains.
\[
    \includegraphics[scale=0.6]{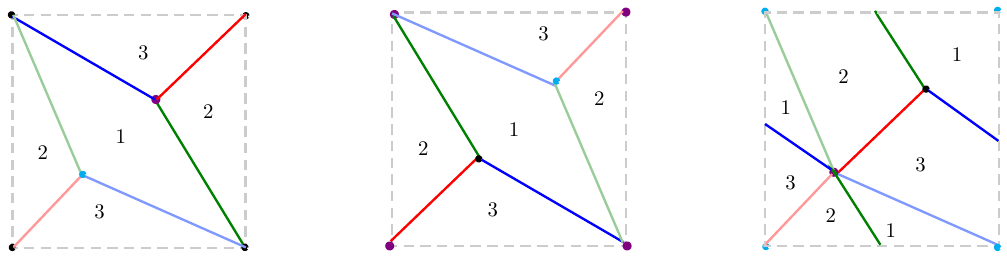}
\]
We can see that the quiver itself embeds into this 2-torus as the 1-skeleton of the CW structure, and the three 2-cells correspond to the generators of the ideal $I$. Call this CW complex $X_{\mathbb{P}^2}^{min}$.
\end{example}

We note there is a K\"{u}nneth-type formula for our construction
\begin{proposition}
Let $A$, $A'$ be homotopy path algebras. Let $A\otimes A'$ be the homotopy path algebra for the product quiver with relations. Then $X_{A\otimes A'}\simeq X_A\times X_A'$.
\end{proposition}
\begin{proof}
This follows from
\begin{align*}
    B(\mathcal{C}_{A\otimes A'}) & =B(\mathcal{C}_A\times \mathcal{C}_{A'}) & \\
&  \simeq  B(\mathcal{C}_A)\times B(\mathcal{C}_{A'}) & \text{ by \cite{Qui73}, Eq. (4), } .
\end{align*}
\end{proof}
\begin{example}\label{A_k}
Let $A_k$ be the path algebra for the $A_k$-quiver with no relations and fix positive integers $k_0, ..., k_n$.  Then, we can consider the tensor product 
\[
A:= A_{k_0}\otimes \cdots \otimes A_{k_n}
\]
as an HPA.
In this case, each  $X_{A_{k_i}}$ is itself homotopic to the $A_{k_i}$ Dynkin diagram.  Therefore, $X_A$, by the above, is homotopic to a union of $n+1$-cubes.  Specifically, this is a CW complex structure on the cube $[0, k_0] \times ... \times [0, k_n]$ broken into cubes of unit volume.  In what follows, we call this CW complex $X_A^{min}$.
\end{example}

The reduction of the CW structures in these examples can be formulated rigorously using discrete Morse theory \cite{For98}, as discussed in \S\ref{sec: Morse}. The CW complexes $Q$ (forgetting directions of the arrows), $X_{\mathbb{P}^2}^{min}$, and $X^{min}_A$, are examples of a \newterm{minimal cellular resolution} of the diagonal bimodule of the quiver algebra, which can sometimes be obtained from a Morse matching.

\section{Stratifications and Entrance Paths}
\subsection{Poset stratification and entrance/exit paths}
Let $Y$ be a locally-compact, Hausdorff, path-connected, locally-path connected, semilocally simply-connected\footnote{Locally-compact ensures the existence of a Verdier dual \cite{Iver86} and  path-connected, locally-path connected, semilocally simply-connectedness provides the existence of a universal cover \cite{Hat02}, \S 1.3.} topological space. By stratification of $Y$, we mean \newterm{stratification by a poset}. We now give the definition.

\begin{definition}
    Let $\mathscr{I}$ be a poset. The \newterm{Alexandrov topology} on $\mathscr I$ is a topology on $\mathscr I$ whose open sets are upwardly closed sets
\begin{equation}
U\subset \mathscr I \text{ is open iff } \forall x\in U,\ y>x\Rightarrow y\in U
\end{equation}
\end{definition}

\begin{definition}[Stratification by a poset]\label{defn: poset stratification}
Let $\mathscr I$ be a poset. An \newterm{$\mathscr I$-stratification} on $Y$ is a partition of $Y$ into the level sets of a surjective continuous map $f: Y \to \mathscr I^{op}$.\footnote{To be consistent with the fact that we took our entrance path to be an increasing sequence, we need to take the opposite ordering here.}
\end{definition}

We give a simple example that is \emph{not} a poset stratification

\begin{example}
    Let $Y=S^1$. Consider a partition of $Y=S_1\coprod S_2$ into half-open, half-closed semicircles $S_1$ (left) and $S_2$ (right).
\[
\begin{tikzpicture}
\draw[black] (1,0) arc (90:270:1);
\draw[black] (2,0) arc (90:-90:1);
\filldraw[black] (1,0) circle (2pt);
\filldraw[black] (2,-2) circle (2pt);
\draw[black] (1.92,0) circle (2pt);
\draw[black] (1.08,-2) circle (2pt);
\end{tikzpicture}
\]
There is no continuous surjection from $Y$ to $\mathscr I^{op}=\{1,2\},\ 1>2$ with $S_1$ and $S_2$ as level sets, since it would require $S_1$ to be open in $Y$ under our definition.
\end{example}

Entrance and exit paths can be defined naturally on a space stratified by poset.
\begin{definition}
An \newterm{entrance path} is a continuous map
\[
\gamma : [0,1] \to Y
\]
such that there exists a sequence of increasing strata $S_1<  ... < S_s$ such that $\gamma^{-1}(S_1) < ... < \gamma^{-1}(S_s)$ is an increasing sequence of intervals stratifying $[0,1]$. 

Choose a collection of base points $y_i \in S_i$ (indexed by $i \in I$). The \newterm{entrance path category} with respect to $S$ is the category whose objects are the base points $y_i$ and whose morphisms from $y_i$ to $y_j$ are homotopy classes of entrance paths $\gamma$ with $\gamma(0) = y_i$ and $\gamma(1) = y_j$.  We denote this category by $\Ent_S(Y)$.

Similarly, an \newterm{exit path} is a continuous map
\[
\gamma : [0,1] \to Y
\]
such that there exists a sequence of decreasing strata $S_1>  ... > S_s$ such that $\gamma^{-1}(S_1) < ... < \gamma^{-1}(S_s)$ is an increasing sequence of intervals stratifying $[0,1]$.  The \newterm{exit path category} is defined similarly and denoted $\op{Exit}_S(Y)$.
\end{definition}

We recall the following standard fact about regular CW complexes:
\begin{proposition} \label{prop: Nanda}
Given a regular CW complex $Y$, let $S_{CW}$ be the stratification by open cells regarded as a poset under the face relations.  The classifying space of entrance paths, $B(\Ent_{S_{CW}}(Y))$, is the barycentric subdivision of $Y$. 
\end{proposition}

\begin{example}
A nice consequence of the above is that any regular CW complex is homeomorphic to the geometric realization of an HPA.
Given a regular CW complex $Y$, we can construct a quiver $Q_Y$  together with an ideal $I_Y \subseteq kQ_Y$ as follows:
\begin{align*}
(Q_Y)_0 & := \{ \text{cells of }Y \} \\
(Q_Y)_1 & := \{ \text{boundary relations} \} \\
I_Y & := \langle p-q : p \sim_Y q \rangle
\end{align*}
The HPA associated to $Y$ is then defined as $A_Y := kQ_Y / I_Y$, also known as the incidence algebra of the cell poset.  Notice that by Proposition~\ref{prop: Nanda}, $X_{A_Y}$ is homeomorphic to $Y$ since $\mathcal C_{A_Y} = \Ent_{S_{CW}}Y$.
\end{example}

\begin{example}\label{A_k CW}
Let us consider an example of the above example!  Fix positive integers $k_0, ..., k_n$.  We consider the cube from Example~\ref{A_k}:
\[
X_A^{min} := [0, k_0] \times \cdots \times [0, k_n] \subseteq \mathbb R^{n+1}.
\]
This is a union of $\prod k_i$ unit cubes $[0,1]^{n+1} +v$ where $v \in \mathbb Z^{n+1}$ is an integer valued point satisfying $0 \leq v_i < k_i$.  We can stratify this by the relative interiors of all faces of these cubes.  This is a CW stratification so, as above, the associated HPA is the incidence algebra of the cell poset which in this case, is isomorphic to the tensor product of path algebras of $A_{k_i}$ Dynkin diagrams:
\[
A \cong A_{k_0} \otimes ... \otimes A_{k_n}.
\]
\end{example}

The definition of poset stratification applies to the cell partition for the CW complexes satisfying the axiom of the frontier, and Whitney stratified spaces. In both cases, the poset ordering is induced by the ``axiom of frontier" condition 
$$\overline S_i\cap S_j \neq \emptyset \iff S_j \subset \overline S_i \iff S_i\leq S_j$$

Note that if the quotient map from the strata of a poset stratified space to a coarsening induces a poset ordering on the coarsening, then the space with coarsened strata is also a poset stratified space. Furthermore, the poset ordering is the one we gave
$$\overline S_i\cap S_j \neq \emptyset \iff S_i \leq S_j.$$ 

We will be interested in poset stratifications of the above type in this paper. Poset stratifications are oftentimes used as a natural setup to investigate the relation between representations of exit paths and constructible sheaves, which is also our goal in this section. 

\subsection{Exceptional stratifications}

\begin{definition}
A stratification $Y = \coprod_{i \in \mathscr I} S_i$ indexed by a finite poset $\mathscr I$ is called \newterm{exceptional} if each of the $S_i$ is contractible to a point and the transitive closure of the condition $\overline{S_i} \cap S_j \neq \emptyset \Leftrightarrow i\leq j$ recovers the poset ordering.
\end{definition}

We give a simple example of a non-exceptional poset stratification. 

\begin{example}
Let $Y$ be the 1-dimensional topological space
\[
\begin{tikzpicture}
\draw[black] (0,0) circle (0.5) ;
\draw[black] (0.5,0) -- (1.5,0) ;
\draw[black] (2,0) circle (0.5);
\end{tikzpicture}
\]
with strata $S_1$ (left) and $S_2$ (right).
\[
\begin{tikzpicture}
\draw[black] (0,0) circle (0.5) ;
\draw[black] (0.5,0) -- (0.9,0) ;
\draw[black] (1,0) circle (2pt);
\filldraw[black] (1.3, 0) circle (2pt);
\draw[black] (1.3,0) -- (1.8,0) ;
\draw[black] (2.3,0) circle (0.5) ;
\end{tikzpicture}
\]
We have $\overline{S_1}\cap S_2\neq \emptyset$ and $\overline{S_2}\cap S_1= \emptyset$ which gives a poset ordering $S_1\leq S_2$ of the strata, however neither stratum is contractible to a point, so this poset stratification is not exceptional.
\end{example}

\begin{lemma}\label{lem: closed}
Let $Y = \coprod_{i \in \mathscr I} S_i$ be an exceptional stratification.  Then, the set
\[
S_{\geq i_0} := \bigcup_{i \geq i_0} S_i 
\]
is closed and
\[
S_{\leq i_0} := \bigcup_{i \leq i_0} S_i 
\]
is open.
\end{lemma}
\begin{proof}
Notice that
\begin{align*}
\overline{S_{\geq i_0}} & =  \overline{S_{\geq i_0}} \cap Y & \\
& = \overline{S_{\geq i_0}} \cap \bigcup S_j & \\
& = \bigcup \overline{S_{\geq i_0}} \cap S_j & \\
& = \bigcup_{i \geq i_0, j} \overline{S_{i_0}} \cap S_j & \text { because $\mathscr I$ is finite }\\
& =  S_{\geq i_0} & \text{ because } \overline{S_{i_0}} \cap S_j \subseteq  \begin{cases}
S_j  & \text{if } j \geq i_0 \\
\emptyset & \text{if } j < i_0
\end{cases}
\end{align*}
i.e.\ $S_{\geq i_0}$ is closed.

Now 
\begin{align*}
Y \backslash S_{\leq i_0} & = Y \backslash \bigcup_{i \leq i_0} S_i \\
& = \bigcup_{j \nleq i_0} S_j \\
& = \bigcup_{j \nleq i_0} S_{\geq j}
\end{align*}
is a finite union of closed sets by the above (hence closed).
\end{proof}

% \begin{lemma}
% Let $M$ be a manifold and $M = \coprod S_i$ be a finite stratification.  Assume that all of the $S_i$ are contractible and that the transitive closure of $S_i \leq S_j \Leftrightarrow \overline{S_i} \cap S_j \neq \emptyset$ is a partial order.  Then $DSh^c_{\Lambda_S}$ is generated by the collection of sheaves $(\iota_i)_! \mathsf k_{S_i}$.
% \end{lemma}

\begin{proposition} \label{prop: exc is exc}
If $Y = \coprod_{i \in \mathscr I} S_i$ is an exceptional stratification, then the set 
\[
\{ (\iota_i)_! \mathsf k_{S_i} : i \in \mathscr I \}
\]
forms a full exceptional collection in $DSh_S(Y)$.
\end{proposition}
\begin{proof}
First we prove the objects are exceptional.  Namely, we calculate:
\[
R\Hom((\iota_i)_! \mathsf k_{S_i}, (\iota_i)_! \mathsf k_{S_i}) = \mathsf k.
\]
We have
\begin{align*}
R\Hom((\iota_i)_! \mathsf k_{S_i}, (\iota_i)_! \mathsf k_{S_i}) & =  R\Hom( \mathsf k_{S_i}, (\iota_i)^!(\iota_i)_! \mathsf k_{S_i}) & \text{ by adjunction} \\
& = R\Hom(\mathsf k_{S_i}, \mathsf k_{S_i}) & \text{ by Proposistion 6.6 of \cite{Iver86}}\\
& = \mathsf k & \text{ since } S_i \text{ is contractible}
\end{align*}

Now we check that morphisms give a partial ordering on the exceptional objects, $(\iota_i)_! \mathsf k_{S_i}$.
For this, suppose $\overline S_j \cap S_i = \emptyset$.  Then
\begin{align*}
R\Hom((\iota_i)_! \mathsf k_{S_i}, (\iota_j)_! \mathsf k_{S_j}) & =  R\Hom( \mathsf k_{S_i}, (\iota_i)^!(\iota_j)_! \mathsf k_{S_j}) \\
& = R\Hom( \mathsf k_{S_i}, \mathbb{D}\iota_i^*\iota_{j*}\mathbb{D}\mathsf k_{S_j}) \\
& = 0
\end{align*}
since $\iota_i^*\iota_{j*} \mathcal F$ is supported on $\overline S_j \cap S_i = \emptyset$.  Now, by the assumption that the stratification is exceptional, the partial ordering on the strata is the opposite partial ordering on the exceptional objects.

Now we prove that the strata generate.   Let $\kappa_i: M \backslash S_i \to M$ be the inclusions and $i_0$ be a maximal element of $I$.  
Then, for any $\mathcal F \in DSh_S(Y)$ we have an exact triangle
\[
{\kappa_{i_0}}_!\kappa_{i_0}^!\mathcal F \to  \mathcal F \to {\iota_{i_0}}_*{\iota_{i_0}}^*\mathcal F \xrightarrow{[1]}
\]
Notice that since $\mathcal F$ is constructible with respect to the stratification, $(\iota_{i_0})^*\mathcal F$ is locally constant, hence generated by $\mathsf k_{S_{i_0}}$ since $S_{i_0}$ is contractible.  Since $S_{i_0}$ is closed by Lemma~\ref{lem: closed}, it follows that $(\iota_{i_0})_*(\iota_{i_0})^*\mathcal F = (\iota_{i_0})_!(\iota_{i_0})^*\mathcal F$ is generated by $(\iota_{i_0})_!\mathsf k_{S_{i_0}}$.  Now we are done by induction on the number of strata.
\end{proof}

\begin{definition}\label{defn: ent/exit space}
Let $S$ be an exceptional stratification of $Y$ with universal cover $\pi: \widetilde{Y} \to Y$.   This induces a stratification $\widetilde{S}$ of $\widetilde{Y}$.  For a point $\widetilde{y} \in \widetilde{Y}$ we define the \newterm{entrance space} at $\widetilde{y}$ to be the subspace
\[
\widetilde{Y}_{\Ent}(\widetilde{y}) := \{ x \in \widetilde{Y} : \exists \widetilde{\gamma} \in \Ent_{\widetilde{S}}(\widetilde{Y}) \text{ with } \widetilde{\gamma}(0) = \widetilde{y}, \widetilde{\gamma}(1) = x \}\xhookrightarrow{i_{\widetilde y}} \widetilde Y.
\]
Similarly, we define the \newterm{exit space} at $\widetilde{y}$ to be the subspace
\[
\widetilde{Y}_{\Exit}(\widetilde{y}) := \{ x \in \widetilde{Y} : \exists \widetilde{\gamma} \in \Exit_{\widetilde{S}}(\widetilde{Y}) \text{ with } \widetilde{\gamma}(0) = \widetilde{y}, \widetilde{\gamma}(1) = x \} \xhookrightarrow{j_{\widetilde y}} \widetilde Y.
\]

\end{definition}

\begin{definition}\label{def: simple}
An exceptional stratification of $Y$ is called \newterm{simple} if for all $\widetilde y \in \widetilde Y$ the entrance space at $\widetilde y$ is contractible and for all $\widetilde y, \widetilde y'$ the difference  $\widetilde{Y}_{\Ent}(\widetilde{y}) \backslash \widetilde{Y}_{\Ent}(\widetilde{y'})$ is contractible whenever it is non-empty.
\end{definition}

\begin{definition}
Given $y \in Y$ choose a lift $\widetilde y \in \widetilde Y$.  The \newterm{entrance sheaf} at $y$ is defined as follows
\[
\mathcal I_y := (\pi \circ i_{\widetilde{y}})_* \mathsf k_{\widetilde{Y}_{\Ent}(\widetilde{y})}.
\]
Similarly, the \newterm{exit sheaf} at $y$ is defined as follows
\[
\mathcal P_y := (\pi \circ j_{\widetilde y})_! \mathsf k_{\widetilde Y_{\Exit}(\widetilde y)}.
\]
\end{definition}

\begin{definition} \label{def: A=end Iw}
Let $A := \End(\bigoplus_{w \in I} \mathcal I_w)^{\op{op}}$ and $e_w \in A$ be the idempotent corresponding to the summand $\mathcal I_w$.  We define
\[
P_w := Ae_w \ \ \ \text{ and } \ \ \ I_w := (e_wA)^* \ \ \ \text{ and } \ \ \ \mathsf k_w := e_w Ae_w. 
\]
\end{definition}

\begin{remark}
By Serre duality for finite dimensional algebras, $P_w$ is a projective left $A$-module and $I_w$ is an injective left $A$-module.  Also note that $P_w^* = (Ae_w)^*= A^{\op{op}}e_w$ is a projective left $A^{op}$-module or right $A$-module.
\end{remark}

\begin{lemma} \label{lem: simple injectives}
% For a fixed set of basepoints $y_i \in S_i$ and lifts $\widetilde{y_i}$ and assume that 
% \[
% \widetilde Y_{\Ent}(\widetilde y_i) \backslash \bigcup_{j > k} \widetilde Y_{\Ent}(\widetilde y_j) \]
% is contractible for all $i,k$.
Let $S$ be a simple stratification of $Y$.
Then,
\[
R\Hom((i_v)_!\mathsf{k}_{S_i}, \mathcal I_j) = \begin{cases}
\mathsf{k} & \text{if } i=j \\
0 & \text{otherwise}
\end{cases}
\]
\end{lemma}
\begin{proof}
We proceed by induction on the number of strata.  For the base case with $\mathscr I = \{ i \}$, the assertion follows from that fact that $S_i$ is contractible.

Now for the inductive step, consider a minimal element $i_0 \in \mathscr I$.  The induction hypothesis says that the statement holds for $\mathscr I \backslash \{i_0\}$, hence it remains to compute $\Hom((i_{i_0})_!k_{S_{i_0}}, \mathcal I_j)$ and $\Hom((i_{i})_!k_{S_{i}}, \mathcal I_{i_0})$.

Notice that $S_{i_0}$ is open by Lemma~\ref{lem: closed} hence, for $j \neq i_0$ 
\[
R\Hom((i_{i_0})_!k_{S_{i_0}}, \mathcal I_j) = R\Hom(k_{S_{i_0}}, \mathcal I_j|_{S_{i_0}}) = 0
\]
since $\mathcal I_j|_{S_{i_0}} = 0$.

Now we compute
\begin{align*}
R\Hom((i_i)_!k_{S_{i}}, \mathcal I_{i_0}) 
& = R\Hom(k_{S_{i}}, (i_{i})^! (\pi \circ i_{\widetilde{y_{i_0}}})_* \mathsf k_{\widetilde{Y}_{\Ent}(\widetilde{y_{i_0}})}) & \text{ by adjunction}\\
& = R\Hom(k_{S_{i}}, \pi_* (\widetilde i_i)^! \mathsf k_{\widetilde{Y}_{\Ent}(\widetilde{y_{i_0}})}) & \text{ by base change}\\
& = R\Hom(k_{\pi^{-1}(S_{i}) \cap {\widetilde{Y}_{\Ent}(\widetilde{y_{i_0})}}},  (\widetilde i_i)^! \mathsf k_{\widetilde{Y}_{\Ent}(\widetilde{y_{i_0}})}) & \text{ by adjunction}\\
& = R\Hom(k_{\pi^{-1}(S_{\geq i}) \cap {\widetilde{Y}_{\Ent}(\widetilde{y_{i_0})}}},  (\widetilde i_i)^! \mathsf k_{\widetilde{Y}_{\Ent}(\widetilde{y_{i_0}})}) & \text{ by induction for $j > i $}\\
& = R\Hom(k_{ {\widetilde{Y}_{\Ent}(\widetilde{y_{i})}}},  (\widetilde i_i)^! \mathsf k_{\widetilde{Y}_{\Ent}(\widetilde{y_{i_0}})}) & \\
& = R\Hom( (\widetilde i_i)_!\mathsf k_{ {\widetilde{Y}_{\Ent}(\widetilde{y_{i})}}},  \mathsf k_{\widetilde{Y}_{\Ent}(\widetilde{y_{i_0}})}) & \text{ by adjunction} \\
& = R\Hom( (\widetilde i_i)_* \mathsf k_{ {\widetilde{Y}_{\Ent}(\widetilde{y_{i})}}},  \mathsf k_{\widetilde{Y}_{\Ent}(\widetilde{y_{i_0}})}) &  \\
& = H^*(\widetilde{Y}_{\Ent}(\widetilde{y_{i_0}}), \widetilde{Y}_{\Ent}(\widetilde{y_{i_0}}) \backslash \widetilde{Y}_{\Ent}(\widetilde{y_{i}}))
& \text{ by Lemma~\ref{lem: rel coh}} \\
& = \begin{cases} \mathsf k & \text{ if } i = i_0 \\ 0 & \text{ if } i \neq i_0 \end{cases} & \text{ since $S$ is simple}
\end{align*}

\end{proof}

\begin{proposition}
Let $S$ be a simple stratification of $Y$. On each stratum $S_w$, pick a base point $x_w\in S_w$. Then, the functor
\begin{align*}
    F_w: Sh_S(Y)  & \rightarrow \mathsf k\Mod \\
    \mathcal{F} & \rightarrow  (\mathcal{F}_{x_w})^*
\end{align*}
is corepresented by $\mathcal I_w$. 
\label{prop: rep stalk}
\end{proposition}
\begin{proof}
Let $w$ be a maximal element of the poset $\mathscr I$.  Then we observe
\begin{equation} \label{eq: maximal}
i_w^*\mathcal F = \mathsf k_{S_w} \otimes_{\mathsf k}\mathcal F_{x_w}
\end{equation}
since $\mathcal F$ is $S$-constructible, $S_w$ is contractible, and closed by Lemma~\ref{lem: closed}.

Again, we proceed by induction on the number of strata.  Fix consider the (base) case where $w$ is a maximal element of the poset $\mathscr I$.  Then $\mathcal I_w = (i_w)_*\mathsf k_{S_w}$ so
\begin{align*}
    R\Hom(\mathcal F, \mathcal I_w) & =  R\Hom(\mathcal F, {i_w}_*\mathsf k_{S_w}) & \\
    & =  R\Hom(i_w^*\mathcal F, \mathsf k_{S_w}) &\text{ by adjunction} \\
     & =  R\Hom(\mathsf k_{S_w} \otimes_{\mathsf k}\mathcal F_{x_w}, \mathsf k_{S_w}) & \text{ by Eq.~\eqref{eq: maximal}}\\
          & = (\mathcal F_{x_w})^* & \text {since $S_w$ is contractible.}
\end{align*}

% This time, let $i_0$ be a maximal element of $I$ and let $\mathcal F$ be a constructible sheaf with respect to $S$. 
% Then, 
% \begin{align*}
%     \Hom(\mathcal F, (i_{i_0})_! \mathsf k_{S_{i_0}}) & =   \Hom(\mathcal F, (i_{i_0})_* \mathsf k_{S_{i_0}})& \text{ by Lemma~\ref{lem: closed}}\\
%   & =  \Hom(\mathcal F|_{S_{i_0}},  \mathsf k_{S_{i_0}}) & \text{ by adjunction}\\
%     & =  (\mathcal{F}_{y_{i_0}})^* & \text{ since $S_{i_0}$ is contractible}
% \end{align*}
Now, suppose $w$ is not maximal and let $i_0$ be a maximal element.
Let $j_{i_0} : Y \backslash S_{i_0} \to Y$ be the inclusion of the open set.  Consider the exact triangle
 \[
 j_{i_0!}j_{i_0}^!\mathcal F \to \mathcal F \to  {i_{i_0}}_* i_{i_0}^*\mathcal F \xrightarrow{[1]}
\]
and apply $R\Hom( -, \mathcal I_w)$ to get
 \[
 R\Hom({i_{i_0}}_*i_{i_0}^*\mathcal F, \mathcal I_w) \to R\Hom(\mathcal F, \mathcal I_w) \to R\Hom({j_{i_0}}_!{j_{i_0}}^!\mathcal F,  \mathcal I_w) \xrightarrow{[1]}
 \]
 which by adjunction, Eq.~\eqref{eq: maximal}, and the fact that $i_{i_0}$ is proper (Lemma~\ref{lem: closed}), is the same as 
  \[
 R\Hom({i_{i_0}}_!(\mathsf k_{S_{i_0}} \otimes_{\mathsf k} \mathcal F_{x_{i_0}}), \mathcal I_w) \to R\Hom(\mathcal F, \mathcal I_w) \to R\Hom(\mathcal F|_{Y \backslash S_{i_0}},  \mathcal I_w|_{Y \backslash S_{i_0}}) \xrightarrow{[1]}.
 \]
 Hence by Lemma~\ref{lem: simple injectives} we get an isomorphism
\[
R\Hom(\mathcal F, \mathcal I_w) \cong R\Hom(\mathcal F|_{Y \backslash S_{i_0}},  \mathcal I_w|_{Y \backslash S_{i_0}})
\]
and we are done by the induction hypothesis.

\end{proof}

\begin{remark}
   We remind the reader that $\mathcal I_w$ corepresents costalk at $w$ and is an injective object in the abelian category $Sh_S(Y)$. In our setup, they are Serre dual to Nadler's microlocal skyscraper sheaves \cite{GPS18, Nad16}, which we notate by $\mathcal P_w$. See Table \ref{tab: section4 summary} and Remark \ref{rem: Serre functor} for a detailed comparison.
\end{remark}

\begin{proposition} \label{prop: main equivalence}
Let $S$ be a simple stratification of $Y$.
Then, the stalk functor
\begin{align*}
    F: Sh_S(Y)  & \rightarrow \op{mod-}\End(\bigoplus_i \mathcal{I}_i) = A\Mod\\
    \mathcal{F} & \rightarrow \bigoplus_{i} \Hom( \mathcal{F}, \mathcal{I}_i)^*=\bigoplus_i \mathcal{F}_{x_i}
\end{align*}
has an inverse
\begin{align*}
    G: \op{mod-}\End(\bigoplus_i \mathcal{I}_i) = A\Mod & \rightarrow Sh_S(Y) \\
    M & \mapsto \bigoplus_i \mathcal{I}_i \otimes_{\mathsf k_A} \mathsf k_M 
\end{align*}
\end{proposition}
\begin{proof}
We have a coevaluation map
\[
 \mathcal F  \xrightarrow{ev^*}  \bigoplus_{i} \mathcal I_i \otimes_{\mathsf k_A} \bigoplus_{j}  \mathsf k_{\Hom(\mathcal F, \mathcal I_j)^*}   =  G \circ F (\mathcal F)  
\]

It is enough to check that $ev^*$ is an isomorphism on stalks and since the  sheaves are constructible we can check this on the base points of the strata.  For $y_t$ we compute the stalk as
\small
\begin{align*}
& (\bigoplus_{i} \mathcal I_i \otimes_{\mathsf k_A} \bigoplus_{j} \mathsf k_{\Hom(\mathcal F, \mathcal I_j)^*}  )_{y_t} & \\
= & \bigoplus_{i}  (\mathcal I_i)_{y_t} \otimes_A \bigoplus_j \Hom( \mathcal F, \mathcal I_j)^*  & \text{ since stalks commutes with tensor product }\\
= &  \bigoplus_{i}  \Hom(\mathcal I_i, \mathcal I_t)^* \otimes_A \bigoplus_{j} \Hom( \mathcal F, \mathcal I_j)^* & \text{ by Proposition~\ref{prop: rep stalk}}\\
= &  \bigoplus_i \Hom(\mathcal I_i, \mathcal I_t)^* \otimes_A \Hom(\mathcal F, \mathcal I_t)^* & \text{ since $e_t$ is the identity on the left of the tensor product} \\
= & P_t^* \otimes_A  \Hom(\mathcal F, \mathcal I_t)^* & \text{ since $\bigoplus_i \Hom(\mathcal I_i, \mathcal I_t) = P_t$ by definition}\\
= & \Hom(\mathcal F, \mathcal I_t)^* & \text{ since $e_t$ is the identity on the right of the tensor product}\\
= &  \mathcal F_{y_t} & \text{ by Proposition~\ref{prop: rep stalk}}
\end{align*}
\normalsize
and the coevaluation map induces this isomorphism on stalks.

Next we observe that there are natural isomorphisms
\begin{align*}
G \circ F(M) &= \bigoplus_{i,j} (\mathcal I_i \otimes_{\mathsf k_A} \mathsf k_M )^*_{y_j} \\
& = \bigoplus_{i,j} (\mathcal I_i)^*_{y_j} \otimes_{A} M \\
& = A \otimes_{A} M \\
& = M
\end{align*}
\end{proof}

\begin{lemma} \label{lem: rel coh}
Let $X$ be a finite CW complex. For closed subcomplexes $F_1\xhookrightarrow{i} X$ and $F_2\xhookrightarrow{j} X$,
\begin{equation}
    R\Hom(i_*\mathsf{k}_{F_1}, j_*\mathsf{k}_{F_2})\simeq H^*(F_2, F_2 \backslash F_1; \mathsf{k})
\end{equation}
\end{lemma}
\begin{proof}
Let $s: F_1 \backslash F_2 \to F_1$  and $t: F_1 \cap F_2 \to F_1$ and be the inclusions.
Consider the exact sequence
\[
0 \to  s_!\mathsf{k}_{F_1 \backslash F_2} \to \mathsf{k}_{F_1} \to t_!\mathsf{k}_{F_1 \cap F_2} \to 0
\]
Applying $R\Hom(-, \mathsf{k}_{F_1})$, we get an exact triangle
\[
R\Hom(t_!\mathsf k_{F_1 \cap F_2}, \mathsf k_{F_1}) \to R\Hom(\mathsf k_{F_1}, \mathsf k_{F_1}) \to R\Hom(s_!\mathsf k_{F_1 \backslash F_2}, \mathsf k_{F_1}),
\]
which is equal to
\[
R\Hom(t_!\mathsf k_{F_1 \cap F_2}, \mathsf k_{F_1}) \to H^*(F_1; \mathsf{k}) \to H^*(F_1 \backslash F_2; \mathsf{k}).
\]
Hence we identify 
$R\Hom(t_!\mathsf k_{F_1 \cap F_2}, \mathsf k_{F_1}) $ with the relative cohomology of the pair $(F_1, F_1 \backslash F_2)$
i.e.
\begin{equation}
    R\Hom(t_!\mathsf k_{F_1 \cap F_2}, \mathsf k_{F_1}) \cong H^*(F_1, F_1 \backslash F_2; \mathsf{k}).
    \label{eq: rel coh}
\end{equation}

The following cartesian square comes from the fiber product of two closed subsets:
\begin{equation}
   \begin{tikzcd}
F_1 \cap F_2 \ar[r, "f"] \ar[d, "g"]& F_1 \ar[d, "i"]\\
F_2 \ar[r, "j"] & X
\end{tikzcd}
\label{eq: fiber product}
\end{equation}

Now we compute
\begin{align*}
    R\Hom(i_*\mathsf k_{F_1}, j_*\mathsf k_{F_2}) & = R\Hom(i_!\mathsf k_{F_1}, j_*\mathsf k_{F_2}) & \text{ since } F_2 \text{ is closed} \\
    & = R\Hom(j^*i_!\mathsf k_{F_1}, \mathsf k_{F_2})
    & \text{ by adjunction} \\
      & = R\Hom(g_!f^*\mathsf k_{F_1}, \mathsf k_{F_2})
    & \text{ by proper base change of \eqref{eq: fiber product}} \\
          & = R\Hom(g_*f^*\mathsf k_{F_1}, \mathsf k_{F_2})
    & \text{ since } F_1 \cap F_2 \text{ is closed}\\
              & = R\Hom(g_* \mathsf k_{F_1 \cap F_2}, \mathsf k_{F_2})
    & \\
    & = H^*(F_2, F_2 \backslash F_1; \mathsf k). & \text{ by \eqref{eq: rel coh}}.
\end{align*}
\end{proof}

\begin{proposition}\label{QCC}
Let $A'$ be the HPA whose vertices are base points of the stratification $S$ and whose paths are homotopy classes of entrance paths between the base points.  We have the following identification
\begin{align*}
A'& = \End(\bigoplus_w \mathcal I_w)^{\op{op}} \\
& = A\  (\text{Definition \ref{def: A=end Iw}}). & 
\end{align*}
\end{proposition}
\begin{proof}
We need to compute
$\Hom(\mathcal{I}_v, \mathcal{I}_w)$. We have a fiber square
$$
\begin{tikzcd} \widetilde Y_{\Ent}(\widetilde{v}) \times_{Y} \widetilde Y_{\Ent}(\widetilde{w}) \ar[d, "f"] \ar[r] & \widetilde Y_{\Ent}(\widetilde{v}) \ar[d, "\alpha=\pi i_{\widetilde{w}}"] \\
\widetilde Y_{\Ent}(\widetilde{w}) \ar[r, "\beta=\pi i_{\widetilde{v}}"] & Y 
\end{tikzcd}
$$
Since there are finitely many arrows between any two vertices,  $\alpha$ and $\beta$ are both proper. By proper base change we get 
\begin{align*}
R\Hom(\mathcal{I}_v, \mathcal{I}_w) & = R\Hom(\alpha_* \mathsf k_{\widetilde Y_{\Ent}(\widetilde{v})}, \beta_* \mathsf k_{\widetilde Y_{\Ent}(\widetilde{w})})\\
& = R\Hom( \beta^*\alpha_* \mathsf k_{\widetilde Y_{\Ent}(\widetilde{v})},  \mathsf k_{\widetilde Y_{\Ent}(\widetilde{w})})\\
& = 
R\Hom ( f_*\mathsf k_{\widetilde Y_{\Ent}(\widetilde{v}) \times_{Y} \widetilde Y_{\Ent}(\widetilde{w})}, \mathsf k_{\widetilde Y_{\Ent}(\widetilde{w})})
\end{align*}
Now we compute 
\begin{align*}
& \widetilde Y_{\Ent}(\widetilde{v}) \times_{Y} \widetilde Y_{\Ent}(\widetilde{w})\\
= & \coprod_{ \{ (p,q) |  p(0) =v, q(0) = w, p(1) = q(1) \} } \widetilde Y_{\Ent}(\widetilde{q}(1)))
\end{align*}
Hence 
\begin{align*}
R\Hom ( f_*\mathsf k_{\widetilde Y_{\Ent}(\widetilde{v}) \times_{Y} \widetilde Y_{\Ent}(\widetilde{w})}, \mathsf k_{\widetilde Y_{\Ent}(\widetilde{w})}) &
= R\Hom( \bigoplus_{\{ (p,q) |  p(0) =v, q(0) = w, p(1) = q(1) \} } \mathsf k_{\widetilde Y_{\Ent}(\widetilde{q}(1))}, \mathsf k_{\widetilde Y_{\Ent}(\widetilde{w})} )\\
&
= \bigoplus_{\{ (p,q) |  p(0) =v, q(0) = w, p(1) = q(1) \} } 
R\Hom(\mathsf k_{\widetilde Y_{\Ent}(\widetilde{q}(1))}, \mathsf k_{\widetilde Y_{\Ent}(\widetilde{w})})
    \end{align*}
By Lemma~\ref{lem: rel coh}, this amounts to computing relative cohomology.  Hence, all terms vanish when $\widetilde{q}(1)\neq \widetilde{w}$ since $H^*(\widetilde Y_{\Ent}(\widetilde{w}), \widetilde Y_{\Ent}(\widetilde{w})\setminus Y_{\Ent}(\widetilde{q}(1))) = 0$ as both spaces are contractible (by assumption).  When $\widetilde{q}(1) = \widetilde{w}$ this contributes a copy of $\mathsf k$ in degree 0 since  $R\Hom(\mathsf k_{\widetilde Y_{\Ent}(\widetilde w)}, \mathsf k_{\widetilde Y_{\Ent}(\widetilde{w})}) = \mathsf k$ since $Y_{\Ent}(\widetilde{w})$ is contractible.

This gives  
\begin{align*}
    R\Hom(\mathcal{I}_v, \mathcal{I}_w) & = \Hom(\mathcal{I}_v, \mathcal{I}_w) \\
    & = \bigoplus_{\{ p :  p(0) =v, p(1) = w \}} \mathsf k
\end{align*}
On the other hand, we have injective left modules $I_v=(e_vA')^*$, and 
\[
\Hom_{A'}(I_v,I_w)= \bigoplus_{\{ p :  p(0) =v, p(1) = w \}} \mathsf k.
\]
\end{proof}

\begin{theorem} \label{thm: main equivalence}
Let $S$ be a simple stratification. There are equivalences of categories
\[
Sh_S(Y)\cong A^{\op{op}}\Mod \cong Fun(\Ent_S(Y), \mathsf k\Mod)^{\op{op}}
\]
\end{theorem}
\begin{proof}
We have
\begin{align*}
    Sh_S(Y) & \cong \End(\bigoplus_w \mathcal I_w)\Mod & \text{ by Proposition~\ref{prop: main equivalence} } \\
    & \cong A^{\op{op}}\Mod & \text{ by Proposition~\ref{QCC}} \\
    & \cong Fun(\Ent_S(Y), \mathsf k\Mod)^{\op{op}} & \text{ because $A$ is, by definition, the HPA of entrance paths.} 
\end{align*}

\end{proof}

\begin{remark} A version of the above theorem was proven in \cite{CP16}, Theorem 6.1, for conical stratifications.  Unfortunately, we were unable to apply this result as the stratifications of primary interest herein are not conical.  
\end{remark}

% \begin{remark}
% The derived statement $DSh_S(Y)\simeq Fun^{ex}(DSh^w_S(Y)^{op}, D(\Vect_{\mathsf{k}}))$ follows from \cite{Nad16}, Theorem 3.21 by taking the homotopy category.
% \df{I think maybe this remark should be removed or expanded.  I don't understand the relevance here.}
% \end{remark}

\subsection{Localization and homotopy}
In this section, we introduce block stratifications and demonstrate that for a block stratification $S$ of $Y$ there is a homotopy between $Y$ and $B \Ent_S(Y)$. %which induces the correspondence seen in Table~\ref{tab: section4 summary}.
For this purpose, we recall the following from Quillen's Higher Algebraic K-Theory \cite{Qui73}, \S 1.

\begin{definition}[Quillen, \cite{Qui73} pg 6]
Consider a functor
\[
F : \mathcal C \to \mathcal C'
\]
Given $v' \in \mathcal C'$, we denote by $v' \backslash F$ the category whose objects are pairs $(v, \alpha)$ where $v$ is an object of $C$ and $\alpha$ is a morphism from $v'$ to $F(v)$.  A morphism in $v' \backslash F$ from $(v_1, \alpha_1)$ to $(v_2, \alpha_2)$ is a morphism $\phi$ in $\mathcal C$ from $v_1$ to $v_2$ such that $F(\phi) \circ \alpha_1 = \alpha_2$.
\end{definition}

\begin{theorem}[Quillen, \cite{Qui73} Theorem A] \label{thm: Quillen}
Assume that $F$ is essentially surjective and $B(v' \backslash F)$ is contractible for all $v' \in \mathcal C'$.  Then $B \mathcal C$ is homotopic to $B \mathcal C'$.
\end{theorem}

\begin{lemma}[Quillen, \cite{Qui73} Corollary 1] \label{lem: equivalent implies homotopic}
If a functor has a left or right adjoint then it induces a homotopy equivalence on classifying spaces.
\end{lemma}
%\begin{proof}
%Let $F:\mathcal{C}\rightarrow \mathcal{C}'$ be the equivalence. Since $F$ is fully-faithful, objects in $v'\backslash F$ are all morphisms mapping out of $v'$ in $\mathcal{C}'$. The identity morphism $Id_{v'}$ is then an initial object in $v'\backslash F$. Hence $B(v'\backslash F)$ is contractible.
%\end{proof}

Our goal now is to combine Quillen's work with Proposition~\ref{prop: Nanda}.  This will allow us to understand the homotopy type of the classifying space of entrance path categories for a more general class of stratifications which we now define.

\begin{definition} \label{def: block}
Let $Y$ be a regular CW complex.  We say that a simple stratification $Y = \coprod S_i$ is a \newterm{block stratification} if each stratum is a union of cell interiors
\[
S_i = \bigcup_{e \in I_i} \op{int}(e).
\]
\end{definition}

\begin{proposition}\label{prop: CW to block}
Let $S$ be a block stratification of a regular CW complex $Y$.  Then the natural map
\[
F: \Ent_{S_{CW}}(Y) \to \Ent_{S}(Y)
\]
induces an equivalence
\[
\overline{F}: \Ent_{S_{CW}}(Y) / L_S \cong \Ent_{S}(Y)
\]
where $L_S$ is the full subcategory of paths which lies entirely in a single stratum of $S$.
\end{proposition}
\begin{proof}
The functor $F$ simply takes a point $y \in Y$ to the same point in $Y$ and a path $\gamma \in \Ent_{S_{CW}}(Y)$ to the same path viewed as a path in $\Ent_{S}(Y)$.  Clearly, if $\gamma$ is a morphism in $L_S$, then $F(\gamma)$ is invertible with inverse the reverse path.

On the other hand, if $\gamma$ is a path in $\Ent_{S}(Y)$, then we can write $\gamma$ as a concatenation
\[
\gamma = \gamma_1 \star ... \star \gamma_n
\]
where each $\gamma_j([0,1))$ lies in a single stratum of $S_{CW}$.  Notice that $\gamma_j$ is an entrance path or $\gamma_j^{-1}$ is an entrance path for all $j$.  Hence, we can construct an inverse functor
\[
\overline{G}: \Ent_{S}(Y) \to \Ent_{S_{CW}}(Y) / L_S
\]
where $\overline{G}(\gamma)$ formally reverses all $\gamma_j$ which are not entrance paths.  Clearly $\overline F \circ \overline G = \op{Id}_{\Ent_{S}(Y)}$ and $\overline G \circ \overline F = \op{Id}_{\Ent_{S_{CW}}(Y)}$ i.e. $\overline F$ is an equivalence.  

\end{proof}

% Let $S$ be a cellular stratification of a regular CW complex $Y$ with universal cover $\pi: \widetilde{Y} \to Y$.   This induces a cellular stratification $\widetilde{S}$ of $\widetilde{Y}$.  For a point $\widetilde{y} \in \widetilde{Y}$ consider the subspace
% \[
% \widetilde{Y}(\widetilde{y}) := \{ x \in \widetilde{Y} : \exists \widetilde{\gamma} \in \Ent_{\widetilde{S}}(\widetilde{Y}) \text{ with } \widetilde{\gamma}(0) = \widetilde{y}, \widetilde{\gamma}(1) = x \}
% \]

\begin{lemma} \label{lem: cover fiber}
Let $S$ be a block stratification of $Y$. For all $y \in Y$ and all lifts $\widetilde y$, there is an equivalence of categories
\begin{align*}
    L: y \backslash F & \to \Ent_{\widetilde{S}}(\widetilde{Y}_{\Ent}(\widetilde{y})). \\
\end{align*}
\end{lemma}
\begin{proof}
By definition, objects of $y \backslash F$ are pairs $(y_0, \gamma_0)$ where $\gamma_0$ is an entrance path from $y$ to $y_0$ in $\Ent_{S}(Y)$.  We define $L$ on objects by $L(y_0, \gamma_0) = \widetilde{\gamma_0}(1)$ where $\widetilde{\gamma_0}$ is the unique lift of $\gamma_0$ starting at $\widetilde{y}$.
By definition, a morphism from $(y_1, \gamma_1)$ to $(y_2, \gamma_2)$ is an entrance path $\alpha$ from $y_1$ to $y_2$ in $\Ent_{S}(Y)$ with
\[
\gamma_2  \sim_p \alpha \circ \gamma_1
\]
We define $L(\alpha)$ to be the unique lift of $\alpha$ starting at $\widetilde{\gamma}(1)$.

Conversely, we can define $L^{-1}$ which takes objects $\widetilde{y_0}$ to $(y_0, \pi(\widetilde{\gamma_0}))$ where $\widetilde{\gamma_0}$ is the unique homotopy class of path from $\widetilde y$ to $\widetilde{y_0}$.  Similarly, $L^{-1}$ takes morphisms $\widetilde{\alpha}: \widetilde{y_1} \to \widetilde{y_2}$ to $\pi(\widetilde{\alpha})$.  These functors are mutually inverse.

% We know that $y_0$ lies in some unique stratum $S_i$.  

% Let $y_i$ be a point in the interior of the unique maximal cell of $S_i$.  Fix a path $\gamma_i$ in $S_i$  from $y_0$ to $y_i$.  We claim that $(y_i, \gamma_i)$ is an initial object of $y_0 \backslash F$.

% Namely, let $y' \in Y$ and $\gamma'$ be an entrance path  from $y_0$ to $y'$ in $\Ent_{S}(Y)$.
\end{proof}
 
The following is now a consequence of the cited result of Quillen.
\begin{theorem} \label{thm: block homotopy}
Let $S$ be a block stratification of a regular CW complex $Y$.  Then $B(\Ent_S(Y))$ is homotopic to $Y$.
\end{theorem}
\begin{proof}
Consider the functor $F: \Ent_{S_{CW}}(Y) \to \Ent_S(Y)$.  We have
\begin{align*}
B(y\backslash F) & \simeq B(\Ent_{\widetilde{S_{CW}}}(\widetilde{Y}_{\Ent}(\widetilde{y})) & \text{ by Lemma~\ref{lem: equivalent implies homotopic} and Lemma~\ref{lem: cover fiber}} \\
& \simeq \widetilde{Y}_{\Ent}(\widetilde{y}) & \text{ by Proposition~\ref{prop: Nanda}}.
\end{align*}
Hence, since $S$ is simple, $B(y\backslash F)$ is contractible. 
Therefore
\begin{align*}
Y & \simeq B(\Ent_{S_{CW}}(Y))&  \text{ by Proposition~\ref{prop: Nanda}} \\
& \simeq B(\Ent_S(Y)) & \text{ by Theorem~\ref{thm: Quillen} and Proposition~\ref{prop: CW to block}}.
\end{align*}
\end{proof}

\begin{example}
There are easy counterexamples to Theorem~\ref{thm: block homotopy} when $S$ is not block (that is, when $\widetilde Y_{\Ent}(\widetilde y)$ is not contractible).  For example, consider the stratification of the sphere into the closed northern hemisphere and open southern hemiphere.  The geometric realization of the entrance path category in this case is an interval (which is not homotopic to a sphere).
\end{example}

\subsection{The tree stratification}
Let $A$ be an HPA with vertices the poset $\mathscr I$ and $\mathcal C := \mathcal C_A$.
For each $i \in \mathscr I$ we may define the full subcategory $\mathcal C_{\geq i}$ consisting of all objects greater than or equal to $i$.  We view $B(\mathcal C_{\geq i})$ as a subset of $B(\mathcal C)$.

\begin{definition} \label{def: tree strata}
 The \newterm{tree stratification}, denoted by $S^{tr}$ of $B\mathcal C$ is the stratification given by
 \begin{align}
X_A = \coprod_{i \in \mathscr I} S^{tr}_i
 \end{align}
 where
 \begin{align}
      S^{tr}_i := B(\mathcal C_{\geq i}) \backslash \bigcup_{j > i} B(\mathcal C_{\geq j}) = B(\mathcal C^{op}_{\leq i}) \backslash \bigcup_{j < i} B(\mathcal C^{op}_{\leq j}).
 \end{align}
 We denote the exit sheaves for the tree stratification by $\mathcal Q_w$ and the entrance sheaves by $\mathcal T_w$.
\end{definition}

% \begin{definition}
% Let $\widetilde{X_A}\xrightarrow{\pi}X_A$ be the universal cover of $X_A$. A \newterm{thick tree} rooted at $v\in Q_0$ is defined to be the contractible closed subspace 
% \begin{align}
% \widetilde Y_{\Ent}(\widetilde{v}):=K(\Path_{A,\widetilde{v}})\xhookrightarrow{i_{\widetilde{v}}} \widetilde{X_A}
% \end{align}
% with a choice of lift $\widetilde{v}\in \pi^{-1}(v)$.
% \end{definition}

\begin{proposition} \label{prop: tree is block}
For each $v \in \mathscr I$, choose a lift $\widetilde v$ to the universal cover $X_A$.  Then
\[
(\widetilde{X_A})_{\Ent}(\widetilde v) = K(\Path_{A,\widetilde{v}}).
\]
Hence, the tree stratification is a block stratification.
\end{proposition}

\begin{proof}
Let $\widetilde{\mathcal{C}}$ be the category whose objects are vertices of $\widetilde{X_A}$, and morphisms are the edges of $\widetilde{X_A}$. Then,
\begin{align*}
    K(\Path_{A,\widetilde{v}}) & =B\widetilde{\mathcal C}_{\geq \widetilde{v}} \\
    & =(\widetilde{X_A})_{\Ent}(\widetilde v).
\end{align*}
To justify that the tree stratification is a block stratification,
we regard $B\mathcal C$ as a regular CW complex and observe that the strata are unions of open cells by definition.  Furthermore, since $K(\Path_{A,\widetilde{v}})$ is the order complex of a bounded below poset, it is contractible and the contraction (to the vertex $\widetilde v$) also contracts the subspace $K(\Path_{A,\widetilde{v}}) \backslash K(\Path_{A,\widetilde{v'}})$.
\end{proof}

% \begin{lemma}\label{lem: s-t}
% There is a space $B(s(F))$ such that the homotopy in Theorem \ref{thm: block homotopy} comes from the following diagram
% \[
% \begin{tikzcd}
% & \ar[ld, "p"'] B(s(F)) \ar[rd, "q"] & \\
% Y & & B\Ent_S(Y).
% \end{tikzcd}
% \]
% Furthermore, the functor $q_! p^*$ preserves entrance sheaves.
% \end{lemma}
% \begin{proof}
% We apply the proof of \cite{Qui73}, Theorem A to $F: \Ent_{S_{CW}}(Y)\to \Ent_S(Y)$. The proof shows that $B\Ent_{S_{CW}}(Y)$ and $B\Ent_S(Y)$ are both homotopic to the classifying space $B(s(F))$, where $s(F)$ is the bisemisimplicial set realizing this functor $F$. Under the homotopy equivalences via $B(s(F))$, $B(y\backslash F)$ is taken to $B(F(y)\backslash \Ent_S(Y))$. We notice that by Lemma~\ref{lem: cover fiber}
% \[
%   B(y\backslash F) =\widetilde{Y}_{\Ent_{S_{CW}}}(\widetilde{y}) \ \ \ \text{ and } \ \ \  B(F(y)\backslash \Ent_S(Y)) = \widetilde{Y}_{\Ent_{S}}(\widetilde{y})
% \]
% % \begin{align*}
% %     B(y\backslash F) & =\widetilde{Y}(\widetilde{y}) & \text{ by Theorem \ref{thm: block homotopy} }
% % \end{align*}
% % and
% % \begin{align*}
% %     B(F(y)\backslash \Ent_S(Y))^{op}= B\Ent(\widetilde{Y}, \widetilde{S})^{op}_{\leq \widetilde{F(y)}}.
% % \end{align*}
% % Since $S_i=\widetilde{Y}(\widetilde{y_i})\setminus \bigcup_{j>i} \widetilde{Y}(\widetilde{y_j})$ and $S^{tr}_i= B\Ent(\widetilde{Y}, \widetilde{S})^{op}_{\leq \widetilde{F(y_i)}}\setminus\bigcup_{j<i} B\Ent(\widetilde{Y}, \widetilde{S})^{op}_{\leq \widetilde{F(y_j)}}$, this means Quillen's homotopy takes $S$ to $S^{tr}$.
% \end{proof}

Consider a point $y \in Y$.  The trivial path $e_y$ is an entrance path hence becomes a vertex in the simplicial complex $B(\Ent_S(Y))$.  Similarly, for an entrance path $\gamma \in \Ent_S(Y)$.  As a morphism in $\Ent_S(Y)$ this naturally corresponds to the edge in $B(\Ent_S(Y))$ which we denote by $F(\gamma)$.  

\begin{lemma} \label{lem: entrance to tree entrance}
Let $S$ be a block stratification of a regular CW complex, $Y$. The functor  \begin{align*}
F: \Ent_S(Y) & \to \Ent_{S^{tr}}(B(\Ent_S(Y))) \\
y & \mapsto e_y \in B(\Ent_S(Y)) \\
\gamma & \mapsto F(\gamma)
 \end{align*}
is an equivalence. 
\end{lemma}
\begin{proof}
Notice that the tree stratification of $B(\Ent_S(Y)))$ has a canonical set of base points $e_{y_i}$ and the functor is, by definition, a bijection on base-points, hence  essentially surjective by definition.

 We may choose $\gamma$ so that the homotopy in Theorem \ref{thm: block homotopy} takes $\gamma$ to $F(\gamma)$.  This implies $F$ is faithful.  
 
 On the other hand since $S^{tr}$ is a block stratification by Proposition~\ref{prop: tree is block}, any entrance map in $\Ent_{S^{tr}}(B(\Ent_S(Y)))$ can be moved homotopically to an edge.  This implies that $F$ is full.
\end{proof} 

\begin{corollary}\label{cor: tree strata}
We have an equivalence of categories given by the composition
\begin{equation}
\begin{tikzcd}
Fun(\Ent_S(Y), \mathsf k\Mod)^{\op{op}} \ar[rrr, "Lemma~\ref{lem: entrance to tree entrance}"]  & & & Fun(\Ent_{S^{tr}}(B(\Ent_S(Y))), \mathsf k\Mod)^{\op{op}}  \\
Sh_S(Y) \ar[rrr, dashed] \ar[u, "Theorem~\ref{thm: main equivalence}"]& & &  Sh_{S^{tr}}B(\Ent_S(Y)) \ar[u, "Theorem~\ref{thm: main equivalence}"]
\end{tikzcd}
\end{equation}
% \begin{align*}
% Sh_S(Y) & \to Sh_{T}B(\Ent_S(Y)) \\
% S_i & \mapsto T_i
% \end{align*}
 \end{corollary}

Note that $\mathcal{T}_w$ is the entrance sheaf at $w$ on $B\Ent_S(Y)$ for the tree stratification. We define $\mathcal{Q}_w$ as the exit sheaf at $w$ on $B\Ent_S(Y)$ for the tree stratification. Corollary~\ref{cor: tree strata} says that the equivalence in Theorem~\ref{thm: main equivalence} and the homotopy in Theorem~\ref{thm: block homotopy} connect objects in the following table.

\begin{table}[H]
\begin{tabular}{ c | c | c | c }
\hline
  & projectives & injectives & simples \\
  \hline
 $A\Mod$ & $P_w$ & $I_w$& $\mathsf{k}_w$ \\  
 $Sh_S(Y)$ & $\mathcal{P}_w$ & $\mathcal{I}_w$ & ${i_w}_!\mathsf k_{S_w}$ \\
 $Sh_{S^{tr}}(B\Ent_S(Y))$ & $\mathcal{Q}_w$ & $\mathcal{T}_w$ & ${i_w}_!\mathsf{k}_{S^{tr}_w}$ \\ 
 \hline
\end{tabular} 
\hspace{\linewidth}
\caption{Summary of the results of \S 4}\label{tab: section4 summary}
\end{table}

\begin{remark} \label{rem: Serre functor}
Notice that in Table~\ref{tab: section4 summary}, as $A$ is a finite dimensional algebra, the Serre functor for $D(A)$ is $\Hom_A(-,A)^*$.  This takes $P_w$ to $I_w$.  Hence, the Serre functor for $DSh_S(Y)$ takes $\mathcal P_w$ to $\mathcal I_w$ and similarly  the Serre functor for $DSh_{S^{tr}}(B\End_S(Y))$ takes $\mathcal Q_w$ to $\mathcal T_w$.
\end{remark}

\begin{remark}
Given a block stratification, we know that $Y$ is homotopic to $B\Ent_S(Y)$ by Theorem~\ref{thm: block homotopy}.  However, the homotopy need not preserve the strata (although it preserves the entrance spaces), see Example~\ref{Ex: P2 section 5}. On the other hand, there is a correspondence (in the proof of Quillen's Theorem A) which should induce the equivalence in Corollary~\ref{cor: tree strata} by Fourier-Mukai transform.
\end{remark}

\section{Toric HPAs} \label{sec: toric}

In this section, we discuss HPAs that arise as endomorphism algebras of line bundles on toric varieties. In particular, we prove that the stratification on the torus considered by Bondal \cite{Bon06} is a block stratification  (Corollary \ref{cor: BR is block}).

This result has a few implications. First, as an immediate consequence of Theorem \ref{thm: main equivalence}, under a properness assumption, the entrance sheaves always form a full strong exceptional collection of the bounded derived category of constructible sheaves with respect to the Bondal stratification.  Second, in the special case when Bondal's collection of line bundles is a full strong exceptional collection, the classifying space of the endomorphism HPA is homotpic to $\mathbb T$ (Theorem \ref{thm: torus}). This suggests that important topological information about the mirror can be obtained directly from a full strong exceptional collection of line bundles. Conversely, the HPA $A$ of the Bondal stratification as the diagonal bimodule over itself has a simplicial resolution whose underlying space is $X_A$, which we will describe in Corollary \ref{cor: cell res}. This implies there is a resolution of the structure sheaf of the diagonal of a toric variety with terms in the Bondal collection and with length at most the maximal path length of the HPA. We refer to \cite{FH23} for a direct consequence built on our results here. In comparison to our approach, we also refer to \cite{HHL23} for a direct construction of a cellular resolution of the diagonal supported on a cellular refinement of the Bondal stratification with length $=\dim \mathbb T$.

Another advantage of our result is that the Bondal stratification only depends on the rays and not the whole fan. Hence, our construction provides a constructible sheaf category which is invariant under variation of GIT but still ``close enough'' to the FLTZ mirror of the toric variety from any chamber. It would be interesting to investigate the exact role this category plays in connecting (mirrors of) toric varieties under variation of GIT.

\subsection{HPAs from line bundles}

Let $G \subseteq \op{Gl}_{n+k}$ be an abelian group acting diagonally on a vector space $V = \mathbb A^{n+k}$.  We denote by $\mathfrak X := [V / G]$ the associated Artin stack.

Applying the exact functor $\widehat{-} = \Hom(-, \mathbb G_m)$ we obtain an exact sequence
 \begin{equation}
     0\rightarrow M \rightarrow \mathbb{Z}^{n+k} \xrightarrow{\mu} \widehat{G} \rightarrow 0.
 \end{equation}
 where $M$ is defined to be the kernel of $\mu$.
 
 \begin{definition}
We say $\mathfrak X$ is \newterm{cohomologically proper} if $\mu(\mathbb R^{n+k}_{\geq 0})$ is a strongly convex cone in $\widehat{G}_{\mathbb R}$ and the fixed locus of $G$ is zero. 
 \end{definition}

\begin{lemma}
The (underived) endomorphism algebra of a collection of line bundles on a cohomologically proper toric stack is an HPA.
\end{lemma}
\begin{proof}
The vertices of the quiver can be identified with elements in the Picard lattice. Since the stack is cohomologically proper, this guarantees that the endomorphism algebra is finite.  Morphisms between line bundles are generated by monomial sections (realized as graded pieces of the Cox ring). Relations among the sections are generated by equalities of monomials $m_1...m_t = n_1...n_s$.  The corresponding ideal satisfies the conditions of Definition \ref{defn: HPA}.
\end{proof}
\begin{definition}
Let $L_1, ..., L_t$ be a collection of line bundles on a toric stack.  The HPA $A = \End(\oplus L_i)$ is called a  \newterm{toric HPA}.  If $L_1, ..., L_t$ is a full strong exceptional collection then $A$ is called a \newterm{toric FSEC HPA}.
\end{definition}
\subsection{Bondal-Ruan type HPAs}
In this section, we study a stratification on $\mathbb{T}^n$ defined by Bondal-Ruan \cite{Bon06}, cohomologically properness assumption.

% Consider a toric DM stack $\mathfrak{X}$ coming from the GIT construction.  That is, assume we have an exact sequence 
% \begin{equation}
%     0\rightarrow M \rightarrow \mathbb{Z}^{n+k} \xrightarrow{\mu} \widehat{G} \rightarrow 0.
% \end{equation}
% and $\mathfrak{X} = [(\mathbb A^{n+k})^{ss}_{\theta} / G]$ for some generic choice of $\theta$. Let $D_i$, $1\leq i\leq n+k$ be the standard basis of $\mathbb{Z}^{n+k}$.

Let $D_i$ for $1\leq i\leq n+k$ be the standard basis of $\mathbb{Z}^{n+k}$.  We consider the \newterm{Bondal-Ruan map} 
\begin{align}
\Phi: \mathbb{T}^n & \rightarrow \widehat{G}\\
\sum_i a_i D_i + M & \mapsto \mu(-\sum_i\floor{a_i}D_i)
\end{align}
naturally defined on the fiber $\mu_{\mathbb{R}}^{-1}(0)/M=M_{\mathbb{R}}/M=\mathbb{T}^n$, where $\floor{a_i}$ is the floor of $a_i$. Let
\begin{align}
\widetilde{\Phi}:\mu_{\mathbb{R}}^{-1}(0)& \rightarrow \widehat{G}\\
\sum_i a_i D_i & \mapsto \mu(-\sum_i\floor{a_i}D_i)
\end{align}
be the natural lift of $\Phi$ to an $M$-periodic map.

We take the associated stratification $S$ on $\mathbb{T}^n$ whose strata are given by the level sets of $S_D := \Phi^{-1}(D)$, equipped with the \newterm{Picard order}:
\begin{align}\label{Pic}
S_E \leq S_D \text{ iff } E-D \text{ is effective.}
\end{align}

For effective $D=\sum_\rho b_\rho D_\rho\in \mathbb{Z}^{n+k}$, we consider the following regions
\begin{align*}
    {\op{Poly}_D} & = \mathbb{R}^{n+k}_{\geq -D}\cap \mu^{-1}_{\mathbb{R}}(0)=\{m\in M_{\mathbb{R}}: \langle m, D_\rho \rangle\geq -b_\rho \} & \\
    \widetilde{S_D} & = (-D+[0,1)^{n+k} ) \cap \mu_{\mathbb{R}}^{-1}(0) &
%    \\
 %   S_D & = \Phi^{-1}(D) & \\
  %   & =\widetilde{S_D}/M  & \text{ by Lemma~\ref{lem: image zonotope}}
\end{align*}

\begin{lemma} \label{lem: image zonotope}
There exists a canonical homeomorphism $f: \widetilde{S_D}\rightarrow S_D$. Furthermore,
\[
\im \Phi=\mu_\mathbb{R}([0,1)^{n+k})\cap \widehat G=\{\mu(D) :S_D\neq \emptyset\}.
\]
\end{lemma}
\begin{proof}
For any $D=\sum_i b_i D_i\in \mathbb{Z}^{n+k}$,
\begin{align*}
  \widetilde{S_D} & =\mu_{\mathbb{R}}^{-1}(0)\cap [-b_i, -b_i+1)_{i=1}^{n+k}&\\
  & = \{\sum_i a_i D_i: -b_i\leq a_i<-b_i+1, \mu_{\mathbb{R}}(\sum_i a_i D_i)=0\}\subset M_{\mathbb{R}} \\
  & =\{\sum_i a_i D_i + M:\Phi(\sum_i a_i D_i + M)= \mu(D)\}\subset \mathbb{T}^n = M_{\mathbb R} / M \\
  & \text{ (since $\mu_{\mathbb{R}}^{-1}(0)\cap [-b_i, -b_i+1)_{i=1}^{n+k}$ is already in a fundamental domain)}\\
  & = \Phi^{-1}(D) = S_D \hspace{25mm} \text{ by definition of $S_D$}
\end{align*}

In other words, the composition of the above gives a canonical homeomorphism
\begin{align*}
f: \widetilde{S_D}  & \to S_D \\
\sum_ia_iD_i & \mapsto \sum_i a_i D_i + M
\end{align*}

In particular, this implies $\widetilde{S_D} + D = \mu_{\mathbb{R}}^{-1}(D)\cap [0,1)^{n+k}\neq \emptyset \Leftrightarrow S_D = \Phi^{-1}(D)\neq \emptyset$.
\end{proof}

% \begin{lemma} \label{lem: BR strata lifts}
% The map
% \begin{align*}
% f: \widetilde{S_D}  & \to S_D \\
% m & \mapsto [m+D]
% \end{align*}
% is a homeomorphism.
%  \end{lemma}
%  \begin{proof}
%  The map $f$ is injective since $[0,1)^{n+k}$ is a fundamental domain. 
 
%  Conversely, if $[t] \in S_D = \Phi^{-1}(D)$ then, by Lemma~\ref{lem: image zonotope}, there exists a lift $m +D \in [0,1)^{n+k}$ such that $\widetilde{\Phi}(m+D) = D$.  Hence, $f(m) = [t]$ and $f$ is surjective.
%  \end{proof}

\begin{proposition}\label{prop: picard order}
Let $\mathbb R^n_{\Ent}(D)$ be the entrance space at a point in $\widetilde{S_D}$ (see Definition \ref{defn: ent/exit space}). We have an equality
\[
\mathbb R^n_{\Ent}(D) = \op{Poly}_D. 
\]
In particular, if we fix a base-point $x_D \in S_D$ for each stratum.  Then, there is a 1-1 correspondence between the monomial basis 
\[
m \in \Hom_{D(\mathfrak{X})}(\mathcal{O}(-D),\mathcal{O}(-E))
\]
and entrance paths starting at $x_D$ and ending at $x_E$.  Similarly, if $\gamma$ is an entrance path originating in $S_D$, then the unique lift  $\widetilde{\gamma}$ to $M_{\mathbb R}$ starting in $\widetilde{S_D}$ lies inside ${\op{Poly}_D}$. Finally, $\op{Poly}_D \backslash \op{Poly}_{E}$ is star-shaped with center $\widetilde{x_D}$ for any $\mu(D),\mu(E) \in \op{Im} \Phi$.  
% sequences of strata $S_D=S_{E_0}$, $S_{E_1}$, \dots, $S_{E_t}$, such that $E_t$ is linearly equivalent to $E$ and $\overline{S_{E_i}}\cap S_{E_{i+1}}\neq \emptyset$ for each $i$.
\end{proposition}
\begin{proof}
Consider the floor of the projection onto the $\rho^{\op{th}}$ factor $\floor{\langle \widetilde{\gamma}, D_\rho \rangle}: [0,1]\rightarrow \mathbb{Z}$ for each $\rho$. If $\widetilde{\gamma}$ is not entirely inside ${\op{Poly}_D}$, then there exists $t\in [0,1]$ and $\epsilon>0$ such that $\widetilde{\gamma}(t)\in {\op{Poly}_D}$ and $\widetilde{\gamma}|_{(t, t+\epsilon]}\cap {\op{Poly}_D}=\emptyset$.
This means there exists $\rho_0$ and $\epsilon'\leq \epsilon$, s.t.
\begin{enumerate}
    \item $ \floor{\langle \gamma(t), D_{\rho_0} \rangle }  >  \floor{\langle \gamma(t+\epsilon'), D_{\rho_0} \rangle }$
    \item $\floor{\langle \gamma|_{(t, t+\epsilon']},D_\rho \rangle}$ does not increase.
\end{enumerate} 
Hence $\widetilde{\gamma}(t)$ and $\widetilde{\gamma}|_{(t, t+\epsilon']}$ live in separate strata, $S_{D_t}$ and $S_{D_{t+\epsilon'}}$, 
%with $S_{D_t}\subset {\op{Poly}_D}$, $S_{D_{t+\epsilon'}}\cap {\op{Poly}_D}=\emptyset$,
but then $-D_{t+\epsilon'}+D_t$ is anti-effective and hence not effective since $\mathfrak X$ is homologically proper. This is a contradiction as $\widetilde{\gamma}$ is therefore not an entrance path. Hence
\[
\mathbb R^n_{\Ent}(D) \subseteq \op{Poly}_D. 
\]

Now suppose $m \in \Hom_{D(\mathfrak{X})}(\mathcal{O}(-D),\mathcal{O}(-E))$ is an element of the monomial basis. This is equivalent to having $m+{\op{Poly}_{E}}\subset {\op{Poly}_{D}}$.
 %and a linearly equivalent divisor defined by $E_t - E = m$.
 %We write
% \begin{align*}
%     {\op{Poly}_D} & =\{m\in M_{\mathbb{R}}: \langle m, D_\rho \rangle \geq -b_\rho\}\\
%     S_D & =\{m\in M_{\mathbb{R}}: -b_\rho+1> \langle m, D_\rho \rangle \geq -b_\rho\}
% \end{align*}
Let $\widetilde{x_D}$ be the lift of $x_D$ in $\widetilde{S_D}$. Consider the straight line path $l_m: [0,1]\rightarrow M_{\mathbb{R}}$ from $\widetilde{x_D}$ to $m+ \widetilde{x_E}$. Since for each $\rho$, $\langle -, D_\rho \rangle|_{l_m}: [0,1]\rightarrow \mathbb{R}$ is a linear function, it is either increasing or decreasing. This means the discretization $\floor{\langle -, D_\rho \rangle|_{l_m}}: [0,1]\rightarrow \mathbb{Z}$ is also either increasing or decreasing, but since $l_m$ starts at $\widetilde{S_D}$ and never leaves ${\op{Poly}_D}$, the discretization can only increase.  This implies that $l_m$ is an entrance path.
Hence
\[
\mathbb R^n_{\Ent}(D) \supseteq \op{Poly}_D. 
\]

Finally, we check that $\op{Poly}_D \backslash \op{Poly}_{E}$ is star-shaped.  To the contrary, suppose there is a point $x \in \op{Poly}_D \backslash \op{Poly}_{E}$ such that the line $l_x$ from $x$ to $\widetilde{x_D}$ passes through $\op{Poly}_{E}$.  Choose $x' \in l_x \cap \op{Poly}_{E}$.  Then, from the above, $l_x$ is an entrance path and hence, the line from $x'$ to $x$ is also an entrance path.  However, we have already proven that all entrance paths  starting at $x'$ lie entirely in $\op{Poly}_E$, which is a contradiction since $x \notin \op{Poly}_E$.

% since at critical times when $l$ leaves one stratum and enters another the difference between the str. This gives the sequence of strata $S_D=S_{E_0}$, $S_{E_1}$, \dots, $S_{E_l}$ where $E_{i+1}-E_i$ is effective. Since a critical time $t_i$ always occurs at the stratum to be entered, we have
% \[
% l(t_i)\in \overline{S_{E_i}}\cap S_{E_{i+1}}\neq \emptyset.
% \]

%  Conversely, if $\gamma$ is an entrance path starting at $x_D$ and ending at $x_E$, then by Lemma~\ref{lem: torus fibers}, $\widetilde{\gamma}$ lies in ${\op{Poly}_D}$ and we can recover $m = \widetilde{\gamma}(1) - \widetilde{x_E}.$  %Hence $\widetilde{\gamma}$ if $E'$ is linearly equivalent to $E$ then $S_{E'} = m + S_E$ for some $m \in M$ and $\gamma$ is an entrance path with $\gamma(0) \in S_D$ and $\gamma(1) \in S_{E'}$, then by Lemma~\ref{lem: torus fibers}, $\gamma$ lifts to an entrance path $\widetilde{\gamma}$ in ${\op{Poly}_D}$.  
% Observe that  $\widetilde{\gamma}(1) \in m +\widetilde{S_E}$ and $\widetilde{\gamma}(1) \in {\op{Poly}_D}$.  This forces $\widetilde{S_E} \subseteq {\op{Poly}_D}$ and therefore $m + {\op{Poly}_{E}} \subseteq {\op{Poly}_D}$ so $m \in \Hom_{D(\mathfrak{X})}(\mathcal{O}(-D),\mathcal{O}(-E))$.

% lies in some $S_E$ and $\lfloor $$D - E$ gives a monomial given a sequence of such strata, one can connect them by a sequence of straight line paths such that the total increase on $\floor{\langle -, D_\rho \rangle|_{l}}$ gives a monomial section.
\end{proof}

\begin{lemma} \label{lem: contractible}
For any $D\in \im \Phi$, $S_D =\bigcap_{i} ({\op{Poly}_D}\setminus {\op{Poly}_{D-D_i}})$/M, in particular all strata in $S$ are contractible.
\end{lemma}
\begin{proof}
Since
\begin{align*}
   -D+[0,1)^{n+k} = \bigcap_i (\mathbb{R}^{n+k}_{\geq -D} \setminus \mathbb{R}^{n+k}_{\geq -D+D_i}), 
\end{align*}
$\widetilde{S_D}=\bigcap_{i}{\op{Poly}_D}\setminus {\op{Poly}_{D-D_i}}$ is by definition.  Hence, $\widetilde{S_D}$ is contractible because it is the intersection of star-shaped regions with center $\widetilde{x_D}$ (hence star-shaped) by Proposition~\ref{prop: picard order}.  Hence $S_D$ is contractible by Lemma~\ref{lem: image zonotope}.
\end{proof}

\begin{corollary} \label{cor: BR is block}
The stratification $S$  of $\mathbb T^n$ is a block stratification.
\end{corollary}
\begin{proof}
This follows from Proposition~\ref{prop: picard order} and Lemma~\ref{lem: contractible}.
\end{proof}

Combining Corollary \ref{cor: BR is block} with our results about block stratifications in Section 4, we obtain a full strong exceptional collection of exit sheaves for the constructible category with respect to the Bondal's stratification.

\begin{proposition} \label{cor: BR is fsec}
    The $S$-entrance sheaves on $\mathbb T^n$ form a full strong exceptional collection of $DSh_S(\mathbb T)$.
\end{proposition}
\begin{proof}
    This is immediate from Corollary \ref{cor: BR is block} and Theorem \ref{thm: main equivalence}.
\end{proof}

\begin{remark}
    In fact, the cohomological properness assumption can be removed, and one can prove using a similar argument involving exit sheaves on $\mathbb R^n$, that the direct sum of $S$-exit sheaves on $\mathbb T$ form a tilting object in the wrapped constructible sheaf category.  The cohomological properness is only used to ensure we do not have cycles in our HPAs since that is the context we discussed in this paper (as the theory is a bit simpler).
\end{remark}

\begin{definition}
The \newterm{Bondal-Ruan HPA} is defined as the endomorphism algebra
\[
A_\Phi := \op{End}(\bigoplus_{\mu(D) \in \im \Phi}\mathcal O_{\mathfrak X}(-D))=R\op{End}(\bigoplus_{\mu(D) \in \im \Phi}\mathcal O_{\mathfrak X}(-D)).
\]
\end{definition}

\begin{corollary} \label{cor: Bondal entrance}
There is an equivalence of categories
$\mathcal C_{A_\Phi} \cong \Ent_S(\mathbb{T}^n)$.
\end{corollary}
\begin{proof}
The 1-1 correspondence from Proposition~\ref{prop: picard order} induces an equivalence
\[
\mathcal C_{A_\Phi} \rightarrow \Ent_S(\mathbb{T}^n)
\]
which takes a vertex $D \in (A_\Phi)_0$ to the base point $x_D$ and a path in $m \in A_\Phi$ to the corresponding entrance path $l_m$.  
\end{proof}

\begin{theorem}\label{thm: torus}
Let $A_\Phi$ be the Bondal-Ruan HPA.   Then $X_{A_{\Phi}}$ is homotopic to $\mathbb{T}^n=M_\mathbb{R}/M$.
\end{theorem}
\begin{proof}
The condition of Theorem~\ref{thm: block homotopy} is satisfied by Corollary~\ref{cor: BR is block}.  Hence,
\begin{align*}
X_{A_{\Phi}} & = B(\mathcal C_{A_{\Phi}}) & \text{ by definition} \\
 & \simeq B(\Ent_S(\mathbb{T}^n)) & \text{ by Corollary~\ref{cor: Bondal entrance}} \\
 & \simeq \mathbb{T}^n & \text{ by Theorem~\ref{thm: block homotopy}}.
\end{align*}

% By Lemma \ref{lem: picard order}, paths in the HPA of $Q=\im \Phi$ coincide with the entrance paths for $\mathcal{S}_{\Phi}$, and the vertex ordering of $\mathcal{S}_{tree}$ coincide with the Picard ordering of $\mathcal{S}_{\Phi}$, therefore we have reconstructed $X_A=B(\Ent(\mathbb{T}^n, \mathcal{S}_{\Phi}))$. The result follows from Theorem~\ref{thm: block homotopy} whose conditions are satisfied by Lemma~\ref{lem: torus fibers}.  
% Since $\mathcal{S}_{\Phi}$ is a cellular stratification coarsened from $\mathcal{S}_{CW}$, by Theorem~\ref{thm: block homotopy}, we have $B(\Ent(\mathbb{T}^n), \mathcal{S}_{\Phi})$ is homotopic to $B(\Ent(\mathbb{T}^n), \mathcal{S}_{CW})$, which is homotopic to $\mathbb{T}^n$ by Proposition~\ref{prop: Nanda}.

% By the same argument, 
% \begin{align*}
%     T_{D}=B(\mathcal{T}_D)\simeq B(\Ent(\bigcup_{D'\geq D} \widetilde{S_{D'}}, \widetilde{\mathcal{S}}_\Phi))=B(\Ent({\op{Poly}_D}, \widetilde{S}_{\Phi}))\simeq {\op{Poly}_D}
% \end{align*}
% Recall that the strata in $\mathcal{S}_{tree}$ are $T_D\setminus \bigcup_{\rho}T_{D-D_\rho}$, and strata in $\mathcal{S}_\Phi$ are $S_D={\op{Poly}_D}\setminus \bigcup_{\rho} P_{D-D_\rho}$, hence the rest of the theorem follows.

%\df{Still need to address why it preserves the tree stratification and what the correct skeleton means}

\end{proof}

% \begin{corollary}
% If $\im \Phi$ is saturated, then $\Hom^0_{D(\mathfrak{X})}(\mathcal{L}_{\mu(D)},\mathcal{L}_{\mu(E)})\neq 0$ iff there is a sequence of strata $S_E=S_{E_0}$, $S_{E_1}$, \dots, $S_{E_l}=S_D$, such that $\overline{S_{E_i}}\cap S_{E_{i+1}}\neq \emptyset$ for each $i$.
% \end{corollary}
% \begin{proof}
% For $\mu(D), \mu(E)\in \im \Phi$, we have
% \begin{align*}
%       \Hom^0(\mathcal{L}_{\mu(D)},\mathcal{L}_{\mu(E)})\neq 0 \Leftrightarrow &  -(m+D)=-E+s \text{ for some } s\in \mathbb{Z}^I_{>0},\ I\subset[n+k], \ m\in M.
% \end{align*}
% Choose any sequence of adjacent unit cubes $-E+s_i+[0,1)^{n+k}$ connecting $-E+[0,1)^{n+k}$ to $-E+s+[0,1)^{n+1}$ with $s_{i+1}-s_{i}\in [0,1]^{J_i}$ s.t. $I=\sqcup_i J_i$. Since $\im \Phi$ is saturated, $\mu (-E+s_i)\in \im \Phi$ for each $s_i$.
% This means $-E_i+[0,1)^{n+k} \cap \overline {-E_{i+1}+[0,1)^{n+k}}\neq \emptyset$ for $E_i:=E-s_i$. We also have
% \begin{align*}
%     -E_i+[0,1)^{n+k} \cap \overline {-E_{i+1}+[0,1)^{n+k}} \neq \emptyset \Leftrightarrow &  \overline{S_{E_i}}\cap S_{E_{i+1}}\neq \emptyset
% \end{align*}
% \end{proof}
% \begin{remark}
% In fact, one can see from the proof that the generating monomial sections are in 1-1 correspondence with such sequences.
% \end{remark}

 Consider a toric DM stack $\mathcal{X}$ coming from the GIT construction for the action of $G$ on $V$.  That is, consider any character $\theta \in \widehat{G}$, and define 
\[
\mathcal{X} := [(\mathbb A^{n+k})^{ss}_{\theta} / G].
\]
This is an open substack of $\mathfrak X$.  Furthermore, 
for generic $\theta$, this is a DM stack.

\begin{definition} \label{def: BR type}
If $\im \Phi$ forms a full strong exceptional collection of line bundles on $\mathcal{X}$, we say $\mathcal{X}$ is of \newterm{Bondal-Ruan type}.
\end{definition}

When $\mathcal X$ is of Bondal-Ruan type then, by assumption,
\[
A_\Phi = \op{End}(\bigoplus_{D \in \im \Phi}\mathcal O_{\mathfrak X}(-D)) = R\op{End}(\bigoplus_{D \in \im \Phi}\mathcal O_{\mathcal X}(-D))
\]
and
\[
D(\mathcal X) = D(A_\Phi).
\]

\begin{theorem}[Bondal-Ruan Homological Mirror Symmetry] \label{thm: BR HMS}
Let $\mathcal X$ be a toric DM stack of Bondal-Ruan type.  Then there are equivalences of categories 
\begin{align*}
D(\mathcal X) & = D(A_\Phi) & \text{ by the assumption that $\mathcal X$ is of Bondal-Ruan type} \\
& = DSh_{S}(\mathbb T^n) & \text{ by Corollary~\ref{cor: Bondal entrance} and Corollary~\ref{cor: tree strata}}
% & = \langle \bigoplus_{D\in \im(\Phi)} \pi_*i_{T_D*}\mathsf{k}_{T_D} \rangle = \langle \bigoplus_{D\in \im(\Phi)} i_{{T^\circ_D}!}\mathsf{k}_{T^{\circ}_D}\rangle \subseteq DSh^c(X_{A_\Phi}) \\
% & = \langle \bigoplus_{D\in \im(\Phi)}\pi_* i_{{{\op{Poly}_D}}*}\mathsf{k}_{{\op{Poly}_D}}\rangle = \langle \bigoplus_{D\in \im(\Phi)} i_{{S_D}!}\mathsf{k}_{S_D}\rangle \subseteq DSh^c(\mathbb T^n) \\
% & = DSh^c_{\Lambda_{\mathcal{X}}}(\mathbb T^n)\subset DSh^c(\mathbb T^n)
\end{align*}
\end{theorem}

\begin{remark}
Theorem \ref{thm: BR HMS} realizes Bondal's original idea \cite{Bon06}.  In particular, it says the toric varieties of Bondal-Ruan type have a full strong exceptional collection of line bundles coming from $\im \Phi$.  Second, it says that homological mirror symmetry for these varieties can be realized through this exceptional collection as an equivalence between the derived category of the toric variety and the derived category of $S_\Phi$-constructible sheaves. 
\end{remark}

\begin{remark}
The celebrated Coherent-Constructible Correspondence (CCC) proved in \cite{FLTZ11, FLTZ12, FLTZ14, Kuw20} is a microlocal generalization of Bondal's approach.
To compare, first there is a Lagrangian skeleton associated to Bondal's stratification defined as the union of singular supports of the $S$-exit sheaves 
$$\Lambda_S:=\bigcup_{\alpha\in \im \Phi} ss(\mathcal P_\alpha).$$  
For a toric DM stack $\mathcal X_\Sigma$, there is also its FLTZ mirror skeleton $\Lambda_\mathcal X\subset T^*\mathbb T$ and the category of constructible complexes whose singular support is in $\Lambda_\mathcal X$, denoted by $Sh^c(\Lambda_{\mathcal X})$.  The CCC says that the derived category $D(\mathcal X)$ is equivalent to $Sh^c(\Lambda_{\mathcal X})$.

When $\mathcal X$ is of Bondal-Ruan type, our result is simply a special case of the CCC, and gives a full strong exceptional collection for $Sh^c(\Lambda_{\mathcal X})$. In this case, we have
$$\Lambda_S=\Lambda_\mathcal X,\text{ and } DSh_{S}(\mathbb T^n)=Sh^c(\Lambda_S)=Sh^c(\Lambda_{\mathcal X}).$$
More generally, for any proper toric simplicial DM stack, there is always a containment $\Lambda_\mathcal X\subset \Lambda_S$. 
  This implies that $Sh^c(\Lambda_{\mathcal X})$ is always a full subcategory of $D_S(\mathbb T^n)$.
 % generally differs from $Sh^c(\Lambda_{\mathcal X})$, there is always an inclusion of closed subskeleton $\Lambda_\mathcal X\subset \Lambda_S$. 
 
 Furthermore, since $S$ only depends on the rays in the fan of $\mathcal X$, our category $D_S(\mathbb T^n)=Sh^c(\Lambda_S)$  contains $Sh^c(\Lambda_\mathcal X)$ for all $\mathcal X$ with the same rays.
Hence, the Bondal stratification can be used as an ambient category to compare different FLTZ mirrors $Sh^c(\Lambda_\mathcal X)$ where $\mathcal X$ varies under birational morphisms.

\end{remark}

We finish this section by discussing the class of toric varieties of Bondal-Ruan type originally given in \cite{Bon06}. 

% Consider the toric Frobenius map 
% \begin{align*}
% F_l: \mathfrak X & \to \mathfrak X
% \end{align*}
% defined by the map of rings
% \begin{align*}
% (F_l)^\# : \mathsf k[x_1, ..., x_{n+k}] & \to \mathsf k[x_1, ..., x_{n+k}] \\
% x_i & \mapsto x_i^l.
% \end{align*}
% and the group homomorphism
% \begin{align*}
%   m_l: G & \to G \\
%   g & \mapsto g^l.
% \end{align*}

% Denote by $G^l$ the group $G$ acting on $\mathbb A^{n+k}$ by 
% \[
% g \cdot_{G^l} x := g^l \cdot_G x.
% \]
% Then, $F_l$ induces a map of stacks
% \[
% F_l:  \mathfrak X \to [\mathbb A^{n+k} / G^l].
% \]
% \df{I wanted to converse the Frobenius into the Cox/Stacky language.  This doesn't seem to be quite correct.  In any case, I think the lemma we want appears in \href{https://arxiv.org/pdf/1205.6861.pdf}{HERE}}

% \begin{lemma} \label{lem: frobenius image}
% There exists $l \in \mathbb N$ such that $(F_l)_*\mathcal O_{\mathfrak [\mathbb A^{n+k} / G^l]}$ is a sum of line bundles of the form $\mathcal O_{\mathfrak X}(-D)$ with $\mu(D) \in \im \Phi$ and so that all $\mu(D) \in \im \Phi$ occur as summands.
% \end{lemma}
% \begin{proof}
% We can decompose $k[x_1, ..., x_{n+k}]$ as a $\widehat{G}$-graded $k[x_1^l, ..., x_{n+k}^l]$-module as follows
% \begin{align*}
% k[x_1, ..., x_{n+k}] & = \bigoplus_{0 \leq e_i \leq l-1} \prod x_i^{e_i}k[x_1^l, ..., x_{n+k}^l] \\
% & = \bigoplus_{0 \leq e_i \leq l-1} k[x_1^l, ..., x_{n+k}^l](-\sum \mu(e_iD_i))
% \end{align*}
% \end{proof}

\begin{lemma} \label{lem: BR nef}
Assume that $C \cap \mu(D_\rho) \geq -1$ for all torus-invariant complete irreducible curves $C \subseteq X_\Sigma$ and all $\rho \in \Sigma(1)$, and $C \cap D_\rho= -1$ for no more than one $\rho$. Then every $\mu(D) \in \im \Phi$ is nef.
\end{lemma}
\begin{proof}
Let $\mu(D) \in \im \Phi$.  Then by Lemma~\ref{lem: image zonotope},
\[
\mu(D) = \mu(\sum a_\rho D_\rho)
\]
with $0 \leq a_\rho < 1$.  Hence,
\begin{align*}
\langle \mu(D), C \rangle & = \sum a_\rho \langle \mu(D_\rho), C \rangle 
\end{align*}
is greater than $-1$ by the assumption.  Now since $\langle \mu(D), C \rangle \in \mathbb Z$ it follows that  $\langle \mu(D), C \rangle \geq 0$, as desired.
\end{proof}

\begin{proposition}[Bondal-Ruan~\cite{Bon06}]
Let $X_\Sigma$ be a smooth proper toric variety.  Assume that $C \cap D_\rho \geq -1$ for all irreducible curves $C \subseteq X_\Sigma$ and all $\rho \in \Sigma(1)$, and $C \cap D_\rho= -1$ for no more than one $\rho$. Then $\im \Phi$ forms a full strong exceptional collection of line bundles.
\end{proposition}
\begin{proof}
By \cite{BDM} (see also \cite{Ueh14}), $\im \Phi$ generates $D(X_\Sigma)$.   It remains to show that 
\[
\op{Ext}^i(\bigoplus_{\mu(D) \in \op{Im}(\Phi)} \mathcal O_{\mathcal X}(-D), \bigoplus_{\mu(D) \in \op{Im}(\Phi)} \mathcal O_{\mathcal X}(-D)) = 0
\]
for all $i > 0$.

For this, we follow the proof in \cite{Bon06}.  For $l$ sufficiently divisible we have
%By \cite{OU13}, Theorem 4.5, this is equivalent to showing that
\begin{align*}
\op{Ext}^i(\bigoplus \mathcal O_{\mathcal X}(-D), \bigoplus \mathcal O_{\mathcal X}(-D)) & = 
\op{Ext}^i((F_l)_*\mathcal O_{\mathcal X}, (F_l)_*\mathcal O_{\mathcal X}) & \text{ by \cite{OU13}, Theorem 4.5} \\
& = \op{Ext}^i((F_l)^*(F_l)_*\mathcal O_{\mathfrak X}, \mathcal O_{\mathfrak X}) & \text{ by adjunction} \\
& = \bigoplus \op{Ext}^i((F_l)^*\mathcal O_{\mathfrak X}(-D), \mathcal O_{\mathfrak X}) & \text{ by \cite{OU13}, Theorem 4.5} \\
&= \op{H}^i(F_l^*\mathcal O_{\mathcal X}(D)) \\
& = \op{H}^i(\mathcal O_{\mathcal X}(lD)).
\end{align*}

%This in turn is equivalent to showing
%\begin{align*}
%\op{H}^i(F_l^*\mathcal O_{\mathcal X}(D)) & = \op{H}^i(\mathcal O_{\mathcal X}(lD))\\ 
%& = \op{H}^i(\mathcal O_\mathcal X(lD)) \\
%& = 0
%\end{align*}
%for $i>0$ and $\mu(D) \in \im \Phi$. 

By Demazure vanishing, it enough to know that $\mu(D)$ is nef for all $\mu(D) \in \im \Phi$.  This is precisely Lemma~\ref{lem: BR nef}.

% Namely, consider the toric Frobenius map 
% \begin{align*}
% F_l: (\mathbb{C}^*)^n & \rightarrow (\mathbb{C}^*)^n\\
% (x_1, \dots, x_n) & \mapsto (x_1^l, \dots, x_n^l)
% \end{align*}
% We want to show that in this case, $\mathcal{O}_\chi^*=(F_l^*\mathcal{O}_\chi)^*$ is nef for each $\mathcal{O}_{\chi}$ in the decomposition
% \[
% F_{l*}\mathcal{O}=\bigoplus_{\chi\in M/lM} \mathcal{O}_\chi.
% \]

% \jh{?}

\end{proof}

\begin{remark}
As pointed out in \cite{Bon06}, all 18 smooth toric Fano threefolds except for (II)(d) and (III)(k) in \cite{WW82} are of Bondal-Ruan type.

\end{remark}

% \begin{lemma} \label{lem: torus fibers}
% For any $\widetilde{y} \in \widetilde{S_D}$ we have $\mathbb R^n(\widetilde{y}) = {\op{Poly}_D}$.
% \end{lemma}

% The following is Theorem 4.5 of \cite{OU13}. \df{Actually this has a whole history.  A lot have people have proven versions of this.  This is the result we need I believe.}
% \begin{theorem}[Ohkawa, Uehara \cite{OU13}] \label{thm: frobenius image}
% There is an isomorphism
% \[
% (F_l)_* \mathcal O \cong \bigoplus_{-D \in \mu([0,1)^{n+k} \cap \frac{1}{l}\mathbb Z^{n+k})}  \mathcal O(-D).
% \]
% \end{theorem}

% We also note the following generation theorem

% \begin{theorem}\cite{BDM} \label{thm: BDM}
% For a smooth DM toric stack $\mathcal{X}$ in the sense of \cite{BCS05}, $\im (\Phi)$ generates $D(\mathcal{X})$.
% \end{theorem}

% This theorem shows that any $D(\mathcal{X})$ can be obtained as a localization of the subcategory generated by $\im \Phi$ on the Artin stack $\mathfrak X$. On the A-side, this means mirrors for all chambers can be constructed as appropriate stop removals from $\bigcup_{D\in \im \Phi} ss(\pi_*i_*\mathsf{k}_{{\op{Poly}_D}})$. It would be interesting to develop an explicit general method to separate various chambers by taking certain (complexes of) objects from this zonotope. For the purpose of this paper, we are content with considering only the following class of cases

\begin{example} \label{Ex: P2 section 5}
We go back to Example \ref{P2} where the poset is $I = \{-2,-1,0\}$. Below are the tree strata $S^{tr} $of $X_A$ (with $S_2^{tr}<S_1^{tr}<S_0^{tr}$).
\[
    \includegraphics[scale=1.0]{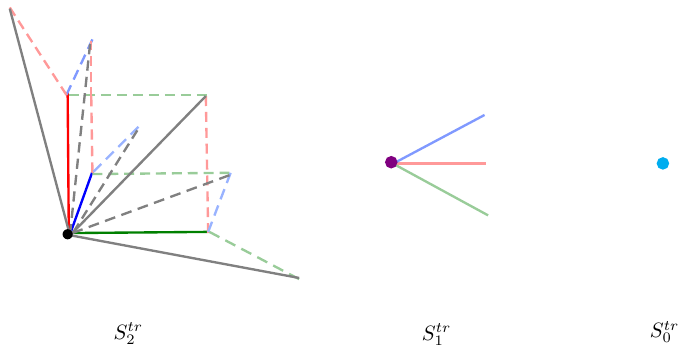}
\]
We compare $S^{tr}$ with the stratification $S$ of the torus $\mathbb{T}^2$ (with $S_2<S_1<S_0$, see Eq. \ref{Pic}):
\[
    \includegraphics[scale=1.0]{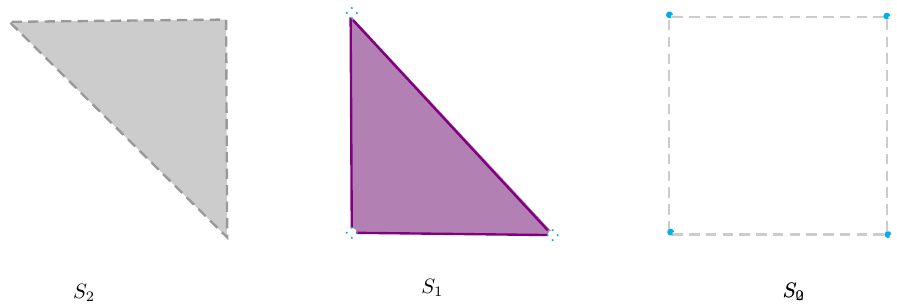}
\]
Specifically, gray, purple, and cyan are the three strata in $S$, corresponding to the black, purple and cyan vertices in the quiver. The gray stratum $S_2$ is the minimal ($\Rightarrow$ open) stratum. The cyan stratum $S_0$ is the maximal ($\Rightarrow$ closed) stratum. The purple stratum $S_1$ is a closed triangle minus 3 points and is locally closed. The $S$-entrance and $S$-exit sheaves are as follows:

\[
    \includegraphics[scale=1.0]{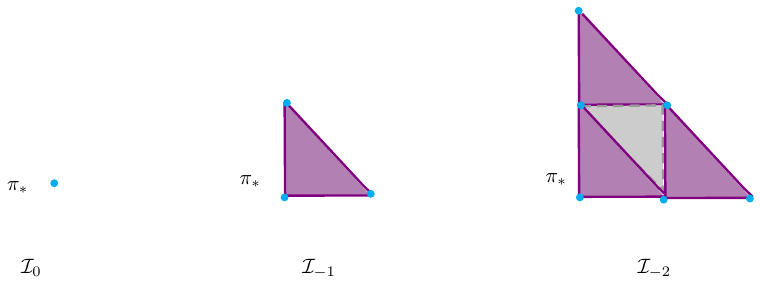}
\]

\[
    \includegraphics[scale=1.0]{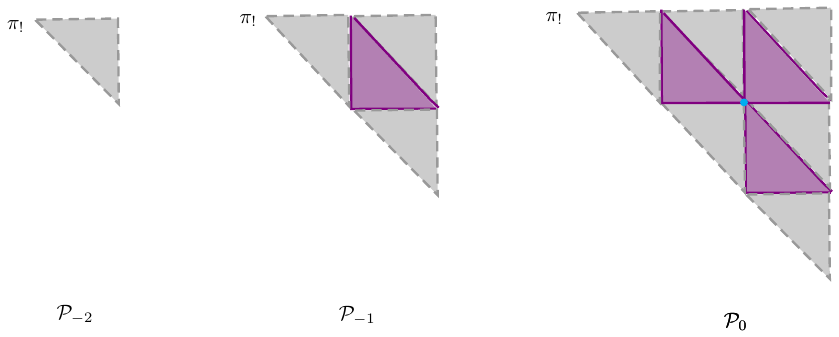}
\]

Although the homotopy $X_A\simeq \mathbb{T}^2$ does not preserve the strata, it preserves the upward closures of the strata and classes of entrance paths. By Corollary \ref{cor: tree strata}, the sheaf categories are equivalent.  The following table summarizes notable equivalent objects in this example.

\begin{table}[H]
\begin{tabular}{ c | c | c | c }
\hline
 $w$ & $D(\mathbb P^2)$ & $Sh_S\mathbb T^2$ & $A\Mod$ \\
  \hline
 $-2$ & $\mathcal O(-2)$ & $\mathcal I_{-2} = \pi_*\mathsf k_{\op{Poly}_{2}}$& $I_{-2}$ \\ 
  $-1$ & $\mathcal O(-1)$ & $\mathcal I_{-1} = \pi_*\mathsf k_{\op{Poly}_{1}}$& $I_{-1}$ \\ 
   $0$ & $\mathcal O$ & $\mathcal I_{0} = \pi_* \mathsf k_{pt}$& $I_{0}$ \\ 
    $-2$ & $\mathcal O(1)[-2]$ & $\mathcal P_{-2}$& $P_{-2}$ \\ 
  $-1$ & $\mathcal O(2)[-2]$ & $\mathcal P_{-1}$& $P_{-1}$ \\ 
   $0$ & $\mathcal O(3)[-2]$ & $\mathcal P_{0}$& $P_{0}$ \\ 
       $-2$ & $\mathcal O(1)[-2]$ & $\mathsf k_{S_2}$ & $S_{-2}$ \\ 
  $-1$ & $\Omega_{\mathbb P^2}(2)[-1]$ & $\mathsf k_{S_1}$& $S_{-1}$ \\ 
   $0$ & $\mathcal O$ & $\mathsf k_{S_0}$ & $S_{0}$ \\ 
   
\hline
\end{tabular} 
\hspace{\linewidth}
\caption{Objects under mirror symmetry for $\mathbb P^2$}
\end{table}

\end{example}

\section{Cellular resolution}
In this section, we study projective cellular resolutions of the diagonal bimodule of an HPA. We observe that $X_A$ itself always gives a projective cellular resolution.  Furthermore, we provide a sequence of simplicial collapses of $X_A$ which, at least in a large class of examples, provide a projective cellular resolution which is equal to the minimal resolution.  As an application, we show, under a mild directability assumption, that the Bondal-Ruan HPA is Koszul and admits a minimal cellular resolution.

\subsection{Cellular bimodule resolution}
%definition of cellular resolution
For a finite CW complex $X$ satisfying the axiom of the frontier, let $\Cell(X)^{op}$ be the set of cells, with partial ordering
\begin{equation}
    \sigma \leq \tau \text{ iff } \tau \subset \overline{\sigma}.
\end{equation}
Let $\mathsf{Cell}(X)^{op}$ be the category of cells on $X$ associated to the poset $\Cell(X)^{op}$, whose objects are cells and morphisms be such that
\begin{equation}
\begin{split}
        & \exists! \text{ morphism }\sigma \rightarrow \tau \text{ if } \sigma \leq \tau \text{ in }\Cell(X)^{op}\\
        & \text{no morphism otherwise}
\end{split}
\end{equation}
Let $R$ be a ring and $R\Mod$ be the category of left $R$-modules. We first define a cellular resolution, using the cosheaf language in \cite{Cur14} for convenience.

\begin{definition}
A \newterm{cellular cosheaf} of $R$-modules on $X$ is a cosheaf of $R$-modules on $\mathsf{Cell}(X)$ equipped with the Alexandrov topology. 
\end{definition}

\begin{theorem}[\cite{Cur14}, Theorem 4.2.10]
For any poset $P$, and category $D$ that is both complete and co-complete, the left Kan extension
$\mathsf{Lan}_{i_X}(-): \mathsf{Fun}(P^{op}, D)\rightarrow \mathsf{CoSh}(P, D)$ is an isomorphism.
\end{theorem}

The theorem says that a cellular cosheaf of $R$-modules on $X$ is simply determined by a functor 
\begin{equation}
    F: \mathsf{Cell}(X)^{op}\rightarrow R\Mod,
\end{equation}
which assigns each $k$-cell $\eta_k$ an $R$-module $F(\eta_k)$, and a module map $F_{\eta_{k+1}}^{\eta_k}:=F(\eta_{k+1})\rightarrow F(\eta_k)$ if $\eta_k>\eta_{k+1}$.

For each facet relation $\sigma>\tau$, write $[\tau:\sigma]$ for its degree. There is a chain complex associated to a cellular cosheaf $F$, namely
\begin{equation}
\begin{split}
        \mathsf{C}_\bullet(F)& :=\dots \bigoplus_{\eta_{k+1}\in \Cell_{k+1}(X)} F(\eta_{k+1})\xrightarrow{\delta_{k+1}} \bigoplus_{\eta_k\in \Cell_k(X)} F(\eta_k)\rightarrow \dots \bigoplus_{\eta_0\in \Cell_0(X)} F(\eta_0), \text{ where }\\
        \delta_{k+1}&:=[\eta_{k+1}:\eta_k]F_{\eta_{k+1}}^{\eta_k}.
\end{split}
\end{equation}

\begin{example}
The chain complex associated to the cellular cosheaf 
\begin{align*}
    \mathsf{CW}: \mathsf{Cell}^{op}(X) & \rightarrow \mathsf{k}\Mod\\
    \sigma & \mapsto \mathsf k\\
    \sigma \geq \tau & \mapsto \op{Id}: \mathsf k \rightarrow \mathsf k
\end{align*}
is the CW-chain complex computing the CW-homology of $X$ with coefficients in $\mathsf k$.
\end{example}

\begin{definition} \label{def: cellular res}
Let $M$ be an $R$-module. A \newterm{cellular resolution} of $M$ supported on $X$ is a cellular cosheaf $F: \mathsf{Cell}(X)^{op}\rightarrow R\Mod$ such that $\mathsf{C}_\bullet(F)$ is a resolution of $M$. The cellular resolution is \newterm{projective} if $\mathsf{C}_\bullet(F)$ is a projective resolution of $M$. 
\end{definition}

\begin{remark}
When $X$ is regular, the degree is $0$ or $\pm 1$; an incidence function $\epsilon: \Cell(X)^{op}\times \Cell(X)^{op} \rightarrow \{-1,0,1\}$ can be chosen such that the value of the incidence function equals the degree. When $X$ is a regular semi-simplicial complex, one can pick a canonical incidence function, and a cellular cosheaf of $R$-modules in this case is simply a semi-simplicial object in $R\Mod$.
\end{remark}

We now describe a cellular resolution of the diagonal bimodule $A$ supported on $X_A$. For a vertex $v\in Q_0$, consider the projective modules 
\[
P_v := Ae_v \text{ and } P_v^{op} :=e_vA.
\]
We view $P_v$ as a left $A$-module, and $P_v^{op}$ as a right $A$-module. Consider the $\mathsf{k}$-algebra 
\begin{align*}
A_0 := \bigoplus_{v\in Q_0} \mathsf k & \rightarrow A \\
1_v & \mapsto e_v.
\end{align*}
We also consider, for each $k$-cell $\eta_k$ in $X_A$, the projective $(A,A)$-bimodule
\begin{equation}
    P_{\eta_k}:= P_{t(\eta_k)}\boxtimes P_{h(\eta_k)}^{op}=P_{t(\eta_k)}\otimes_{A_0} \mathsf{k}_{\eta_k} \otimes_{A_0} P_{h(\eta_k)}^{op} \in A\otimes_{\mathsf k} A^{op}\textnormal{-mod},
\end{equation}
where the left action of $A$ is given by concatenation with paths ending at $t(\eta_k)$, and the right action of $A$ is given by concatenation with paths starting at $h(\eta_k)$.

Recall that each $k$-simplex $\eta_k$ can be represented by a sequence of paths $[p_0< \dots < p_k]$ in $K(\Path)$ up to the equivalence relation $\sim$. We say $[p_0< \dots < p_k]$ is a \newterm{canonical representative} of $\eta_k$ if its image under the quotient $/\sim$ is the cell $\eta_k$ and $p_0=e_{t(\eta_k)}$. By construction, every cell in $X_A$ has a unique canonical representative. We denote the canonical representative of $\eta_k$ by $[\eta_k]$.

We first show that the projectives
\begin{equation}
    \mathsf{C}_k:=\bigoplus_{\eta_k \in \Cell_k(X_A)} P_{t(\eta_k)}\boxtimes P_{h(\eta_k)}^{op}
\end{equation}
form a semi-simplicial object in $A\otimes_{\mathsf k} A^{op}$-mod i.e.\ $\mathsf{C}: \mathsf{Cell}(X_A)\rightarrow A\otimes_{\mathsf k} A^{op}\mathsf{-mod}$ forms a cellular cosheaf. For each cell $\eta_k$, we define the semi-simplicial maps $\partial_i$ on the canonical representative $[\eta_k]=[e_{t(\eta_k)}< p_1 < \dots < p_k]$ as follows.
\begin{equation}
\begin{split}
    \partial_i & :=\sum_{\eta_k\in \Cell_k(X_A)}\partial_{i,\eta_k}: \mathsf{C}_k \rightarrow \mathsf{C}_{k-1} \\
    \partial_{i,\eta_k} (1 \otimes [\eta_k] \otimes 1) & :=\begin{cases}
    p_1 \otimes [e_{h(p_1)}< \dots < p_k] \otimes 1 & i=0\\
    1\otimes [e_{t(\eta_k)}< \dots \hat{p_i} \dots < p_k] \otimes 1 & 0<i<k\\
    1\otimes [e_{t(\eta_k)}< \dots < p_{k-1}] \otimes p_k/p_{k-1} & i=k
    \end{cases},
\end{split}
\end{equation}
\begin{lemma}
The data $(\mathsf{C}_k, \partial_i)$ forms a semi-simplicial $(A,A)$-bimodule.  In particular, there is an associated chain complex
$\mathsf{C}_\bullet$ whose $k^{\text{th}}$ component is $\mathsf{C}_k$ and with differential
\[
d_k:=\sum_{i=0}^k(-1)^i\partial_i
\]
\end{lemma}
\begin{proof}
One only needs to check the semi-simplicial identities
$\partial_i\partial_j=\partial_{j-1}\partial_i$ for all $i<j$. For a canonical representative $[e_{t(p_1)}< \dots < p_k]$, we compute this in 5 cases:
\begin{enumerate}
    \item For $0< i<j< k$: 
\begin{align*}
\partial_i \partial_j(1\otimes [e_{t(p_1)}< \dots < p_k]\otimes 1) & =\partial_{j-1} \partial_i (1\otimes [e_{t(p_1)}< \dots < p_k]\otimes 1) \\
& =1\otimes [e_{t(p_1)}< \dots \hat{p_i} \dots \hat{p_j} \dots < p_k ]\otimes 1
\end{align*}
\item For $i=1,j=1$: 
\begin{align*}
\partial_0 \partial_1(1\otimes [e_{t(p_1)}< \dots < p_k ]\otimes 1) & = p_2\otimes [e_{h(p_2)}< p_3/p_2\dots < p_k/p_2]\otimes 1 \\
& = \partial_0 \partial_0(1\otimes [e_{t(p_1)}< \dots < p_k ]\otimes 1)
\end{align*}
\item For $1<j< k$: 
\begin{align*}
\partial_0 \partial_j(1\otimes [e_{t(p_1)}< \dots < p_k ]\otimes 1) & = p_1\otimes [e_{h(p_1)}< \dots \hat{p_j}/p_1\dots < p_k/p_1]\otimes 1 
\\
& = \partial_{j-1} \partial_0(1\otimes [e_{t(p_1)}< \dots < p_k ]\otimes 1)
\end{align*}
\item For $0<i< k-1$: 
\begin{align*}
\partial_i \partial_k(1\otimes [e_{t(p_1)}< \dots < p_k ]\otimes 1) & = 1 \otimes [e_{t(p_1)}< \dots \hat{p_i}\dots < p_{k-1}]\otimes p_k/p_{k-1} \\
& = \partial_{k-1} \partial_i(1\otimes  [e_{t(p_1)}< \dots < p_k ]\otimes 1) 
\end{align*}
\item For $i=k-1,j=k$: 
\begin{align*}
\partial_{k-1} \partial_k(1\otimes  [e_{t(p_1)}< \dots < p_k ]\otimes 1) & = 1 \otimes  [e_{t(p_1)}< \dots < p_{k-2} ]\otimes p_k/p_{k-2} \\
& = \partial_{k-1} \partial_{k-1}(1\otimes [e_{t(p_1)}< \dots < p_k ]\otimes 1).
\end{align*}
\end{enumerate}

%Hence $\partial_i\partial_j=\partial_{j-1}\partial_i$ for all $i<j$. 

%This makes $C:\Delta^{op}_+\rightarrow A\otimes_{\mathsf k} A^{op}$-mod a semi-simplicial object. The differentials in the associated chain complex $C_\bullet$ are precisely $d_\bullet$.
\end{proof}

In particular, we have $\mathsf{C}_0= \bigoplus_{v \in Q_0} P_v \boxtimes P_v^{op}$. Let \[
\widetilde{\mathsf{C}}_\bullet= \cdots \mathsf{C}_1\rightarrow \mathsf{C}_0 \xrightarrow{m} A\rightarrow 0
\]
be the augmented chain complex where $m$ is the multiplication map. Consider the collection of right $A$-module morphisms
\begin{equation}
    \begin{split}
        h_k & : \mathsf{C}_k \rightarrow \mathsf{C}_{k+1}\\
        h_k (a\otimes [e_v < p_1 < \dots < p_k] \otimes b) & :=  \begin{cases} 1\otimes [e_{t(a)}< a < ap_1 \dots < ap_k]\otimes b & v=h(a)\\
        0 & \textnormal{else}
        \end{cases}\\ 
        & \textnormal{ for every path } a, \textnormal{ extended $\mathsf{k}$-linearly on the left} \\
        h_{-1} & : A \rightarrow \mathsf{C}_0\\
        h_{-1}(b) & :=1\otimes [e_{t(b)}]\otimes b, \textnormal{ extended $\mathsf{k}$-linearly on the left}
    \end{split}
\end{equation}

\begin{corollary} \label{cor: cell res}
The collection of maps $\{h_k\}$ form a contracting homotopy of $\widetilde{\mathsf{C}}_\bullet$.  In particular, $\mathsf{C}_\bullet$ is a projective cellular resolution of $A$ as as a $(A,A)$-bimodule.
\end{corollary}
\begin{proof}
For convenience, set $p_0:=a$. We show $d_{k+1}h_k+h_{k-1}d_k=\textnormal{Id}$ by computation. 
\begin{enumerate}
    \item For $k\geq 1$, we compute  $d_{k+1}h_k (a\otimes [e_{h(a)}<p_1\dots <p_k]\otimes b) $
\small
\begin{align*}
 = & d_{k+1} (1\otimes [e_{t(a)}< a < ap_1 \dots < ap_k]\otimes b) \\
 = &  a \otimes [e_{h(a)}<p_1\dots <p_k] \otimes b  + \sum_{i=0}^{k-1}(-1)^{i+1} 1\otimes [e_{t(a)}< \dots \hat{ap_i} \dots < ap_k] \otimes b  \\
& + (-1)^{k+1} 1\otimes [e_{h(a)}<\dots <ap_{k-1}]\otimes p_kb/p_{k-1}.
\end{align*}
\normalsize
\item and from the other direction we compute $h_{k-1}d_k (a\otimes [e_{h(a)}<p_1\dots <p_k]\otimes b) $
\small
\begin{align*}
  = &  h_{n-1} (ap_1\otimes [e_{h(p_1)}<\dots p_k/p_1]\otimes b  + \sum_{i=1}^{k-1} (-1)^i a\otimes [e_{h(a)}<\dots \hat{p_i}< p_k]\otimes b \\
& + (-1)^k a \otimes [e_{h(a)}< \dots <p_{k-1}]\otimes p_kb/p_{k-1})\\
= &  1\otimes [e_{t(ap_1)}<ap_1< \dots <ap_k]\otimes b  + \sum_{i=1}^{k-1} (-1)^i 1 \otimes [e_{t(a)}<a<\dots \hat{ap_i}\dots<ap_k]\otimes b\\
& + (-1)^k 1\otimes [e_{t(a)}<a<\dots <p_{k-1}]\otimes p_kb/p_{k-1}\\
 = &  \sum_{i=0}^{k-1} (-1)^i 1 \otimes [e_{t(a)}<\dots \hat{ap_i}\dots<ap_k]\otimes b +(-1)^k 1\otimes [e_{t(a)}<a<\dots <p_{k-1}]\otimes p_kb/p_{k-1}\\
 = &  -d_{k+1}h_k (a\otimes [e_{h(a)}<p_1\dots <p_k]\otimes b) + \textnormal{Id}(a \otimes [e_{h(a)}<p_1\dots <p_k] \otimes b).
\end{align*}
\normalsize
\item Similarly for $k=0$ we compute,
\small
\begin{align*}
    d_1h_0(a\otimes [e_{h(a)}]\otimes b)  = & d_1(1\otimes [e_{t(a)}<a]\otimes b) \\ = & a\otimes [e_{h(a)}]\otimes b - 1\otimes[e_{t(a)}]\otimes ab
    \end{align*}
  \normalsize
\item and
\small  
\begin{align*}
    h_{-1}m(a\otimes [e_{h(a)}]\otimes b)  = & h_{-1}(ab) \\
    = & 1\otimes[e_{t(ab)}]\otimes  ab  \\ = & 1\otimes[e_{t(a)}]\otimes ab.
\end{align*}
  \normalsize
\item Finally
\small  
\begin{align*}
    mh_{-1}(b)  = & m(1 \otimes [e_{t(b)}] \otimes b) \\
    = & b.
\end{align*}
\end{enumerate}
\normalsize
\end{proof}

%\begin{corollary}
%$C_\bullet$ is a bimodule resolution of $A$ in $A\otimes_{\mathsf k} A^{op}$-mod.
%\end{corollary}
%\begin{proof}
%By the above lemma, $\widetilde{C}_\bullet$ is acyclic, which makes $C_\bullet$ a right module resolution of $A$ whose differentials are bimodule morphisms. By the exactness of the forgetful functor, $C_\bullet$ is also a bimodule resolution of $A$.
%\end{proof}

\begin{corollary}
Let $A$ be an HPA and $Y$ be the underlying CW complex in a projective cellular resolution of $A$ as a bimodule.  Then, there are isomorphisms of CW homology groups 
\[
H_i(Y, \mathsf k) \cong H_i(X_A, \mathsf k)
\]
for all $i$.
\end{corollary}
\begin{proof}
 Let $P_\bullet \cong A$ be the projective resolution coming from $Y$ and
\begin{align*}
F: \mathcal C_A & \to mod-\mathsf k \\
 v & \mapsto \mathsf k \\
 a & \mapsto \text{id}
\end{align*}
be the trivial representation.  

We have
\begin{align*}
(F \otimes_{\mathsf k} F^{\op{op}}) \otimes_{A \otimes_{\mathsf k} A^{\op{op}}} \mathsf{C}_\bullet & = (F \otimes_{\mathsf k} F^{\op{op}}) \otimes^{L}_{A \otimes_{\mathsf k} A^{\op{op}}} A & \text{ by Corollary~\ref{cor: cell res}} \\
& = (F \otimes_{\mathsf k} F^{\op{op}}) \otimes_{A \otimes_{\mathsf k} A^{\op{op}}} P_\bullet & \text{ by assumption}
\end{align*}
but $(F \otimes_{\mathsf k} F^{\op{op}}) \otimes_{A \otimes_{\mathsf k} A^{\op{op}}} \mathsf{C}_\bullet$ is exactly the CW homology of $X_A$ and $(F \otimes_{\mathsf k} F^{\op{op}}) \otimes_{A \otimes_{\mathsf k} A^{\op{op}}} P_\bullet$ is exactly the CW homology of $Y$.
\end{proof}

We make the following conjecture

\begin{conj}
If $Y$ is the underlying CW complex in a projective cellular resolution of $A$ as an $A \otimes_{\mathsf k} A^{\op{op}}-$module, then $Y$ is homotopic to $X_{A}$.
\end{conj}

\begin{remark}
In the next section, we produce smaller (and sometimes minimal) projective cellular resolutions using discrete Morse theory.  In particular, all examples produced in the next section satisfy the above conjecture.
\end{remark}

% \begin{proof}
% By definition, $X_{A(Y)}=B\mathcal{C}_{A(Y)}$. One can choose a regular CW complex subdivision of $Y$ such that the cell stratification on $Y$ is a cellular stratification. The classifying space of the fiber category of $\mathcal{C}_{A_Y^{reg}}=\mathsf{Cell}(Y_{reg})\rightarrow \Ent(Y)$ is contractible. By Theorem \ref{thm: block homotopy}, we have $Y\simeq B\Ent(Y)$. Now it suffices to show $X_{A(Y)}\simeq B\Ent(Y)$. 

% By the universal property of classifying space \cite{}, we have a bijection 
% \begin{align*}
%     [Y, B\Ent(Y)] = \{ \text{sheaves of $\Ent(Y)$-sets on $Y$ with representable stalks}\}/ \simeq
% \end{align*}

% We can define a fully faithful functor
% \begin{align*}
% F:DSh_{\mathcal{T}_{A(Y)}}(B\mathcal{C}_{A(Y)}) & \rightarrow DSh_{\mathcal{T}_{A_Y}}(B\Ent(Y))\\
% \iota_{i!}P_{e_i} & \mapsto \iota_{i!} P^\bullet_{e_{\leq i} }
% \end{align*}

% Since both $B\mathcal{C}_{A(Y)}$ and $B\Ent(Y)$ give resolutions of the diagonal bimodule, it can be computed in $D(A(Y)\otimes A(Y)^{op})$ that $F(P_\bullet(B\mathcal{C}_{A(Y)}))\simeq P_\bullet(B\Ent(Y))\simeq P_\bullet(Y)$. Universal property of the classifying space suggests 
% \end{proof}

\subsection{Morse Matching} \label{sec: Morse}
Our next goal is to reduce $\mathsf{C}_\bullet$ to a cellular resolution supported on a CW complex with fewer cells. The resulting complex is an analog of the Morse complex.  Throughout this section, we use $C_\bullet$ to denote CW chain complexes of $\mathsf{k}$-modules, and use $\mathsf{C}_\bullet$ to denote chain complexes of $(A,A)$-bimodules.  

We first recall the main theorem in discrete Morse theory. Fix a regular CW complex $X$.
\begin{definition}
A \newterm{matching} $M$ is a collection of subsets $M_k^{top}\sqcup M_k^{bottom} \subset \Cell_k(X)$, together with a bijection $m_k: M_k^{top}\rightarrow M_{k-1}^{bottom}$ for every $k$ which takes a $k$-cell to one of its facets.
\end{definition}
Write $M^{top}:=\coprod_k M^{top}_k$ and $M^{bottom}:=\coprod_k M^{bottom}_k$. A matching gives a bijection $m:M^{top}\rightarrow M^{bottom}$. We think of a matching as $M=\textnormal{Graph}(m)\subset M^{top}\times M^{bottom}$, where an element in $M$ is a pair of matched cells in consecutive dimensions. A cell is called \newterm{critical} if it is unmatched. 

Consider the Hasse diagram $\Gamma_X$ of the cell poset of $X$. A matching $M$ is simply a graph-theoretic matching on $\Gamma_X$. Now orient each matched edge towards the vertex whose corresponding cell has smaller dimension, and each unmatched edge towards the vertex whose corresponding cell has larger dimension. Denote the Hasse diagram with this orientation $\Gamma_X^M$.

\begin{definition}
A matching $M$ is \newterm{acyclic} if $\Gamma_X^M$ has no oriented cycles.
\end{definition}

\begin{theorem}[Fundamental Theorem of Discrete Morse Theory, \cite{For98}]\label{forman}
An acyclic matching of cells in $X$ gives a homotopy from $X$ to a CW complex $X^{crit}$ consisting of critical cells. 
\end{theorem}

\begin{remark}
Writing $C_\bullet(X)$ for the CW-homology of a CW complex, the above implies there is a short exact sequence of chain complexes of $\mathsf{k}$-modules
\begin{equation}
    0\rightarrow C_\bullet^M(X) \rightarrow C_\bullet(X) \rightarrow C_\bullet (X^{crit})
    \rightarrow 0,
\end{equation}
where $C_\bullet^M(X)$ is defined as the kernel and $M$ is the matching which defines $X^{crit}$. Theorem \ref{forman} says that $M$ is acyclic iff $C_\bullet^M(X)$ is acyclic.
\end{remark}

Ideally, one would like the critical cells in the minimal resolution to describe vertices and arrows in the quiver, the relations and higher syzygies. Motivated by such consideration, we do not consider all matchings on $X_A$ but only those \newterm{internal} to the path algebra:

\begin{definition} \label{def: internal}
    A matching $M$ on $X_A$ is \newterm{internal to $A$} if it satisfies
\begin{enumerate}
    \item Cells corresponding to $Q_0$ and $Q_1$ are unmatched;
    \item For each $(\eta_k,\eta_{k-1})\in M$, $h(\eta_k)=h(\eta_{k-1})$ and $t(\eta_k)=t(\eta_{k-1})$.
\end{enumerate}
\end{definition}

An internal matching can be written as $M=\coprod_{v_1,v_2}M_{v_1,v_2}$, where
\begin{equation}
    M_{v_1,v_2}:=\{(\eta_k, \eta_{k-1}): h(\eta_k)=h(\eta_{k-1})=v_1, t(\eta_k)=t(\eta_{k-1})=v_2 \}
\end{equation}
For each $(v_1,v_2)$, take $I_{max}(v_1,v_2)$ to be the set of saturated cells from $v_1$ to $v_2$, and define the subcomplex 
\begin{equation}
    X_{v_1,v_2}:=\bigcup_{\eta\in I_{max}(v_1,v_2)}\bigcup_{\sigma\leq \eta}\sigma\subset X_A.
\end{equation}

\begin{definition}
The \newterm{matching complex} associated to an internal matching $M$ is defined to be the complex $\mathsf{C}_\bullet^M := \bigoplus_{v_1,v_2} \mathsf{C}_\bullet^{M_{v_1,v_2}}$, where
\begin{equation}
    \mathsf{C}_\bullet^{M_{v_1,v_2}} := C_\bullet^{M_{v_1,v_2}}(X_{v_1,v_2}) \otimes (P_{v_1}\boxtimes P_{v_2}^{op}).
\end{equation}

\end{definition}

Algebraically, a matching peels off pairs of terms from the resolution, so that the quotient complex is quasi-isomorphic.

\begin{proposition}\label{prop: acyclic}
Let $M$ be an internal matching. The following are equivalent
\begin{enumerate}
    \item $M$ is acyclic.
    \item $\mathsf{C}^M_\bullet$ is acyclic.
    \item $\mathsf{C}^M_\bullet \simeq \bigoplus_{v_1,v_2} \bigoplus_{(\eta_k,\eta_{k-1})\in M_{v_1,v_2}} [\mathsf k_{\eta_k}\xrightarrow{\textnormal{Id}} \mathsf k_{\eta_{k-1}}]\otimes (P_{v_1}\boxtimes P_{v_2}^{op})= \bigoplus_{(\eta_k, \eta_{k-1})\in M } [\mathsf k_{\eta_k}\xrightarrow{\textnormal{Id}} \mathsf k_{\eta_{k-1}}]\otimes P_{\eta_k}$.
\end{enumerate}
\end{proposition}
\begin{proof}
(1) $\Leftrightarrow$ (2): Notice that any oriented cycle $\gamma$ in $\Gamma_{X_A}^{M}$ is a composition of alternating upward and downward arrows, therefore one can write $\gamma=a_1^da_1^u\dots a_n^da_n^u$, where $a_i^u$ and $a_i^d$ are upward and downward arrows in $\Gamma_{X_A}^M$ such that $h(a_n^u)=t(a_1^d)$. Furthermore, since $M$ is internal, once the head or tail is cut off by a downward arrow, it can never be restored by the subsequent upward arrows. Therefore, each $a_i^d$ in $\gamma$ has to fix the head and the tail of the cell, which means a cycle can only occur in $\Gamma_{X_A}^{M_{v_1,v_2}}$ for some $v_1$ and $v_2$. 

The fundamental theorem of discrete Morse theory demonstrates $C_\bullet^{M_{v_1,v_2}}(X_{v_1,v_2})$ is acyclic iff $M_{v_1, v_2}$ is acyclic. Now viewing $P_{v_1}\boxtimes P_{v_2}^{op}$ as a free $\mathsf k$-module we have,
\begin{equation}
     H_*(\mathsf C_\bullet^{M_{v_1,v_2}}) \simeq  H_*(C_\bullet^{M_{v_1,v_2}}(X_{v_1,v_2})) \otimes_{\mathsf k} (P_{v_1}\boxtimes P_{v_2}^{op}).
\end{equation}
Therefore, $\mathsf{C}_\bullet^{M_{v_1,v_2}}$ is acyclic iff $M_{v_1,v_2}$ is acyclic. Now by definition of $\mathsf{C}_\bullet^M$ as a direct sum, hence $\mathsf{C}_\bullet^M$ is acyclic iff $C_\bullet^{M_{v_1,v_2}}$ is acyclic for all $v_1$ and $v_2$.

(3) $\Rightarrow$ (2):  This is immediate since each direct summand $[\mathsf k_{\eta_k}\xrightarrow{\textnormal{Id}} \mathsf k_{\eta_{k-1}}]$ is acyclic.

(2) $\Rightarrow$ (3): Suppose $\mathsf{C}_\bullet^{M_{v_1,v_2}}$ is acyclic for all $v_1$ and $v_2$.
The homotopy equivalence between
$X_{v_1,v_2}$ and $X_{v_1,v_2}^{crit}$ is a sequence of simplicial collapses where each step collapses a matched cell to the union of its unmatched boundary.
 Let $X_{v_1,v_2}(n)$ be the remaining CW complex after the $(n-1)$-th step and let $M_{v_1,v_2}(n)$ be the remaining matching on $X_{v_1,v_2}(n)$. 
 
 The snake lemma gives the following diagram.
\begin{equation}
\begin{CD}
0 @>>>   [\mathsf k_{\eta_k}\xrightarrow{\textnormal{Id}} \mathsf k_{\eta_{k-1}}]  @>>> C_\bullet^{M_{v_1,v_2}(n)}  @>>> C_\bullet^{M_{v_1,v_2}(n+1)}   @>>> 0\\
@. @| @VVV @VVV @.\\
0 @>>> [\mathsf k_{\eta_k}\xrightarrow{\textnormal{Id}} \mathsf k_{\eta_{k-1}}] @>>> C_\bullet(X_{v_1,v_2}(n)) @>>> C_\bullet(X_{v_1,v_2}(n+1)) @>>> 0\\
@. @. @VVV @VVV @.\\
 @. 0 @>>> C_\bullet(X_{v_1,v_2}^{crit}) @= C_\bullet(X_{v_1,v_2}^{crit}) @>>> 0.
\end{CD}
\end{equation}
%\begin{equation}
%0 \to    [\mathbb{C}_{\eta_k}\xrightarrow{\textnormal{Id}} \mathbb{C}_{\eta_{k-1}}]  \to  C_*^{M_{v_1,v_2}(n)} \to C_*^{M_{v_1,v_2}(n+1)} \to 0
%\end{equation}
Since the top line is a short exact sequence of acyclic complexes with projective components, the top exact sequence of complexes splits giving as isomorphism
\begin{equation}
    C_\bullet^{M_{v_1,v_2}(n)} \cong  C_\bullet^{M_{v_1,v_2}(n+1)} \oplus [\mathsf k_{\eta_k}\xrightarrow{\textnormal{Id}} \mathsf k_{\eta_{k-1}}].
\end{equation}
By induction we obtain 
\begin{equation}
    C_\bullet^{M_{v_1,v_2}}\simeq \bigoplus_{(\eta_k,\eta_{k-1})\in M_{v_1,v_2}} [\mathsf k_{\eta_k}\xrightarrow{\textnormal{Id}} \mathsf k_{\eta_{k-1}}].
\end{equation}
Tensoring both sides by $P_{v_1}\boxtimes P_{v_2}^{op}$ gives the desired isomorphism.
\end{proof}

%\begin{lemma}
%The matching complex $C^{M}_\bullet$ is isomorphic to the complex of bimodules
%\begin{equation}
%     D^M_\bullet:=\cdots \rightarrow \bigoplus_{\eta_k\in M_k}P_{\eta_k} \xrightarrow{d^m_k}\bigoplus_{\eta_{k-1}\in M_{k-1}}P_{\eta_{k-1}}\cdots \xrightarrow{d^m_1}P_{\eta_1} \rightarrow 0,
%\end{equation}
%where
%\begin{equation}
%    d^m_k:=d_k|_{M_k^{top}\rightarrow M_{k-1}^{bottom}}
%\end{equation}
%\end{lemma}

We now describe the \newterm{Morse complex of projective bimodules} for an acyclic internal matching $M$. A discrete gradient path $\gamma=a_0^da_1^ua_1^d\dots a_n^ua_n^d$ from $\eta_k$ to $\eta_{k-1}$ is a composition of downward and upward arrows in the Hasse diagram such that $t(a_0^d)=\eta_k$ and $h(a_n^d)=\eta_{k-1}$. Let $\Gamma(\eta_k, \eta_{k-1})$ be the set of gradient paths from $\eta_k$ to $\eta_{k-1}$. We will also use
\begin{equation}
\begin{split}
    Z_+(\eta_k, \eta'_k):=&\{w=a_1^ua_1^d\dots a_n^ua_n^u: t(w)=\eta_k,\ h(w)=\eta'_k,\ n\in \mathbb{Z}_+\}  \\
    Z_-(\eta_k, \eta_k'):=&\{v=a_1^da_1^u\dots a_n^da_n^u: t(v)=\eta_k,\ h(v)=\eta'_k,\ n\in \mathbb{Z}_+\}  
\end{split}
\end{equation}
We set the Morse complex of bimodules to be 
\begin{equation}
    \mathfrak{C}_\bullet:= 0\rightarrow \cdots \rightarrow \mathfrak{C}_k\xrightarrow{\delta_k} \mathfrak{C}_{k-1} \rightarrow \cdots \rightarrow 0, 
\end{equation}
where $\mathfrak{C}_k=\bigoplus_{\eta_k \in I_k \cap M^{crit}} P_{\eta_k}$, and 
\begin{equation}
    \delta_k(1\otimes [\eta_k]\otimes 1)= \sum_{\eta_{k-1}\in I_{k-1} \cap M^c} \sum_{\gamma\in \Gamma(\eta_k, \eta_{k-1})} \epsilon(\gamma) l(\gamma)\otimes [\eta_{k-1}] \otimes r(\gamma).
\end{equation}
The multiplicity $\epsilon(\gamma)=\pm 1$ is the same as the one picked to define the topological discrete Morse complex. The left and right coefficients are given by
\begin{equation}
    \begin{split}
    l(a_0^da_1^ua_1^d\dots a_n^ua_n^d) & =\prod_{h(a_i^d)\textnormal{ is an initial facet of } t(a_i^d) } [t(a_i^d)/h(a_i^d)]\\
    r(a_0^da_1^ua_1^d\dots a_n^ua_n^d) & =\prod_{h(a_i^d)\textnormal{ is a terminal facet of } t(a_i^d) } [t(a_i^d)/h(a_i^d)]
    \end{split}
\end{equation}
These coefficients keep track of the initial or terminal path being factored out by a non-internal facet relation and make $\delta_k$ a bimodule map.
\begin{lemma}
    $\mathfrak{C}_\bullet$ is a chain complex.
\end{lemma}
\begin{proof}
By definition,
\begin{equation}
    \delta_{k-1}\delta_k ([\eta_k])=\sum_{[\eta_{k-2}]}\sum_{\gamma_k\gamma_{k-1}} \epsilon(\gamma_k)\epsilon(\gamma_{k-1})l(\gamma_{k-1})l(\gamma_k)\otimes [\eta_{k-2}]\otimes r(\gamma_k)r(\gamma_{k-1}),
\end{equation}
where $\gamma_k\gamma_{k-1}$ is the concatenation of the two paths through $h(\gamma_k)=t(\gamma_{k-1})$. This means one can write $\gamma_k\gamma_{k-1}=\gamma_k'a_k^da_{k-1}^d\gamma_{k-1}'$, with $h(a_k^d)=h(\gamma_k)=t(\gamma_{k-1})=t(a_{k-1}^d)$. 

For a fixed codimension two face $\rho_{k-2}\subset\rho_k$ in a $k$-simplex $\rho_k$, there are precisely two facets of $\rho_k$ containing $\rho_{k-2}$. In other words, $\rho_{k-2}$ and $\rho_k$ are connected by precisely two configurations of consecutive downward arrows, $\rho_k \xrightarrow{a_k^d}\rho_{k-1} \xrightarrow {a_{k-1}^d} \rho_{k-2}$ and $\rho_k \xrightarrow{b_k^d}\rho'_{k-1} \xrightarrow {b_{k-1}^d} \rho_{k-2}$. We also have $l(a_k^da_{k-1}^d)=l(b_k^db_{k-1}^d)$ and $r(a_k^da_{k-1}^d)=r(b_k^db_{k-1}^d)$, since the left and right coefficients only depend on the endpoints. For $u\in Z_+(\eta_k, \rho_k)$, $v\in Z_-(\rho_{k-2},\eta_k)$, this gives
\begin{equation}
    \begin{split}
        \gamma_k^a:=ua_k^d, &\  \gamma_{k-1}^a:= a_{k-1}^dv\\
        \gamma_k^b:=ub_k^d, &\  \gamma_{k-1}^b:= b_{k-1}^dv
    \end{split}
\end{equation}
Hence, we can rewrite
\begin{equation}
\begin{split}
        & \delta_{k-1}\delta_k ([\eta_k]) =\\
        & \sum_{\rho_k, \rho_{k-2}, [\eta_{k-2}]} \sum_{u \in Z_+(\eta_k, \rho_k), v\in Z_- (\rho_{k-2}, \eta_{k-2})} \epsilon(\gamma^a_k)\epsilon(\gamma^a_{k-1})l(\gamma^a_{k-1})l(\gamma^a_k)\otimes [\eta_{k-2}]\otimes r(\gamma^a_k)r(\gamma^a_{k-1}) \\
        & + \epsilon(\gamma^b_k)\epsilon(\gamma^b_{k-1})l(\gamma^b_{k-1})l(\gamma^b_k)\otimes [\eta_{k-2}]\otimes r(\gamma^b_k)r(\gamma^b_{k-1}).
\end{split}
\end{equation}
The two summands have the same left and right coefficients, and $\epsilon(\gamma^a_k)\epsilon(\gamma^a_{k-1})+ \epsilon(\gamma^b_k)\epsilon(\gamma^b_{k-1})=0$ because the corresponding two summands in the topological Morse complex have the same multiplicity, whose sum has to vanish. 
\end{proof}

\begin{lemma} \label{lem: Morse complex is resolution} 
For any acyclic matching, the corresponding Morse complex $\mathfrak{C}_\bullet$ is quasi-isomorphic to $\mathsf{C}_\bullet$.  That is  the Morse complex gives a projective cellular resolution of $A$.
\end{lemma}
\begin{proof}
Following \cite{For98}, \S 6, we get an exact sequence
\[
\begin{tikzcd}
0 \ar[r] & \mathsf C_\bullet^M  \ar[r, "f"] & \mathsf C_\bullet \ar[r, "g"] & \mathfrak C_\bullet \ar[r] & 0 \end{tikzcd}
\]
where for $(\eta_k, \eta_{k+1}) \in M$, $f(1_{\eta_k}) = d(1_{\eta_{k+1}})$ and  $f(1_{\eta_{k+1}}) = 1_{\eta_{k+1}}$ and $g = 1 + dm + md + (dm)^2 +(md)^2 + ...$ where $m(\eta_k) = \eta_{k+1}$.
The result now follows from the fact that $C_\bullet^M$ is acyclic by Proposition~\ref{prop: acyclic}.
\end{proof}

We now describe an algorithm to produce an acyclic matching and discuss when the resulting Morse complex agrees with the minimal resolution \cite{Ei56, BK99}.

Let $[p]$ be a path in $A$. Let the open interval $(e_{t(p)},p)$ be the subset in $\Path_A$ with the same poset ordering, where an element is a nonconstant nontrivial subpath of $[p]$ in $A$. Notice that the ordering in $\Cell K((e_{t(p)},p))$ given by insertion of path classes agrees with the ordering in $\Cell(X_A)$.

\begin{lemma} \label{lem: cells to order complex}
There is a bijection of sets
\begin{align*}
\coprod_{p \neq e_v} \Cell(K(e_{t(p)}, p) & \rightarrow \Cell_{\geq 2}(X_A) \\
\alpha & \mapsto [e_{t(p)}<\alpha < p].
\end{align*}
  %  Cells in $K((e_{t(p)},p))$ are in 1-1 correspondence with cells of dimension $\geq 2$ in $X_A$ containing the edge $[e_{t(p)}<p]$. (or cells in $(e_{t(p)},p)$).
\end{lemma}
\begin{proof}
%We send 
%\begin{align*}
 %   K(\Path_{A, \leq p}) & \rightarrow [e_{t(p)}<\dots < p]\\
  %  \alpha & \mapsto [e_{t(p)}<\alpha < p].
%\end{align*}
%The map is clearly surjective. To show that it is 1-1, suppose $[e_{t(p)}<\alpha < p]=[e_{t(p)}<\alpha' < p]$ represent the same cell in $X_A$. This means $\alpha=\alpha'$ is the same chain of paths in $A$.
This map is surjective since any cell of dimension at least 2 has a canonical representation of this form.  The map is also injective since two canonical representations are never equivalent.
\end{proof}

% \begin{remark}
% The above intervals are not the same as intervals in the cell poset $\Cell(X_A)$. Rather, each interval defined above is a union of intervals in $\Cell(X_A)$ where the upper bound cells are all saturated. This is because for an interval in $\Cell(X_A)$, the upper bound cell is a fixed chain of path classes, whereas our intervals have all possible saturated chains ending at the same path class $[p]$.
% \end{remark}

% \begin{definition}
% For any open cell $e$ in $X_A$, let
% \begin{align}
%     \Star(e) & =\bigcup_{c\in \Cell(X_A),\  e\leq c} \overline{c}\\
%     \Star^\circ(e) & = \bigcup_{c\in \Cell(X_A),\ e\leq c} c\\
%     \Link(e) & =\Star(e) \setminus \Star^\circ(e)
% \end{align}
% \end{definition}

% We observe that each open interval is a particular CW poset:
% \begin{align}
%     (e_{t(p)},[p])= \Cell(\Link([e_{t(p)}<p])).
% \end{align}

Now, for each $[p]$, choose a lexicographical ordering on the maximal chains of the poset $(e_{t(p)}, [p])$.  Let $M_{[p]}$ denote the corresponding matching on order complex of the poset $(e_{t(p)},[p])$ defined by Babson-Hersh \cite{BH05}. 

We extend this to an internal matching on $\overline{M_{[p]}}$ on $[e_{t(p)},[p]]$ as follows.  For each matched pair of cells $[p_1 < ... < p_n] \leftrightarrow [q_1 < ... < q_n]$ we match the pair $[e_{t(p)}< p_1 < ... < p_n < p] \leftrightarrow [e_{t(p)}< q_1 < ... < q_n < p]$.  This is well-defined by Lemma~\ref{lem: cells to order complex}.

Furthermore, we choose any unmatched 0-cell $[t]$ and match $[e_{t(p)}< p]$ with $[e_{t(p)} < t < p]$. Observe that, by definition, all matched cells in $\overline{M_{[p]}}$ contain $p$.  Hence for $[p]\neq [p']$, we have $\overline{M_{[p]}}\cap \overline{M_{[p']}}=\emptyset$. It follows that the union
\begin{align*}
    \overline{M}=\coprod_{[p]} \overline{M_{[p]}}
\end{align*}
gives a well-defined global matching on $\Cell(X_A)$.

\begin{definition}
We call the matching $\overline{M}$ the \newterm{Babson-Hersh matching}.
\end{definition}

\begin{lemma}
The Babson-Hersh matching is acyclic. 
\end{lemma}
\begin{proof}
Suppose to the contrary that we have a cycle
\[
[e_{t(p)} < ... < p] \xrightarrow{down} \partial_i[e_{t(p)} < ... < p]  \xrightarrow{up} m_k(\partial_i [e_{t(p)} < ... < p]) \xrightarrow{down} ... \xrightarrow{up} [e_{t(p)} < ... < p].
\]
If at any step $i=0$ or $i=k$, then this is impossible since all upward arrows do not change the end points (as the matching is internal).  If $0 < i < k$ for all steps, then this induces a cycle in $M_{[p]}$ which is acyclic, a contradiction.
\end{proof}

\begin{lemma} \label{lem: Tor is reduced homology}
For $i \geq 2$, there is an isomorphism of $\mathsf k$-modules 
\[
\emph{Tor}_i(S_v, S_w) = \bigoplus_{ t(p)= v, h(p) = w}  \widetilde{H}_{i-2}(K((e_v, p)))
\]
where $\widetilde{H}_i(K((e_v, p)))$ is the reduced homology of the order complex on the poset of paths between $e_v$ and $p$ with coefficients in $\mathsf k$.\footnote{Our convention is that the empty set is an empty simplicial complex so that reduced simplicial chains for the empty set are $\mathsf k$ in degree -1 and $\widetilde{H}_{-1}(\emptyset) = \mathsf k$.}
\end{lemma}
\begin{proof}
Consider the projective resolution $\mathsf C_\bullet$ of $A$.

Then
\[
\text{Tor}_i(S_v, S_w) = H_i(S_v \otimes_A \mathsf C_\bullet \otimes_A S_w)
\]

Now let us consider $S_v \otimes_A \mathsf C_\bullet \otimes_A S_w$.  First notice that 
\[
S_v \otimes_A Ae_v \otimes_{\mathsf k} e_w A \otimes_A S_w = \mathsf k
\]
and is zero otherwise.  Hence 
\[
S_v \otimes_A \mathsf C_i \otimes_A S_w =
 \bigoplus_{ \eta_i \in \Cell_i(X_A), t(\eta_i)=v, h(\eta_i) = w} \mathsf k 
\]
Now, notice that the ``non-internal" differentials in the above complex vanish since tensoring with the simples kills all arrows.  
Hence
\[
S_v \otimes_A \mathsf C_\bullet \otimes_A S_w = \bigoplus_{p | t(p) = v, h(p) = w} \mathsf S^p_\bullet
\]
where $\mathsf S^p_\bullet$ is the summand consisting of all cells equivalent to cells of the form $[e_{t(p)} < ... < p]$.
%\[
%\mathsf S^p_i := \bigoplus_{[e_{t(p)} < ... < p] \sim \eta_i \in \Cell_i(X_A)} P_t(\eta_i) \boxtimes P_{h(\eta_i)}^{op}
%\]

By Lemma~\ref{lem: cells to order complex}, the $i$-cells in $\mathsf S^p_i$ are just the $i-2$-cells of the order complex on $(e_{t(p)}, p)$ except for $\mathsf S^p_1$ which is the cell $[e_t(p) < p]$.  Furthermore, the differential agrees, by definition, with the simplicial differential.  In summary, $\mathsf S^p_\bullet$ is nothing more than the reduced simplicial homology complex for the order complex on $(e_{t(p)}, p)$ shifted by 2.
 
\end{proof}

\begin{definition}
We say that a resolution $\mathsf P_\bullet$ of an $A$-module is \newterm{minimal} if all components of the differential lie in the ideal generated by the arrows.
\end{definition}

\begin{theorem}
 Assume that for all paths $p$ all the homologies of $K((e_{t(p)}, p))$ are free $\mathsf k$-modules and the critical cells for the Babson-Hersh matching form a basis of these homologies.  Then, the Morse complex associated to the Babson-Hersh matching is a minimal projective resolution of $A$ as an $A \otimes A^{op}$-module (which is cellular).  The converse holds when $A$ is a graded with respect to path length.
\end{theorem}
\begin{proof}
By Lemma~\ref{lem: Morse complex is resolution} the Morse complex is a resolution.  Now, 
\begin{align*}
  \mathfrak C_i & = \bigoplus_{t(p)=v, h(p)=w} \widetilde{H}_{i-2}(K((e_v, p))) \otimes_{\mathsf k} P_v  \otimes_{\mathsf k} P_w^{op} & \text{by assumption} \\
 & = 
   \bigoplus_{v \leq w} Tor_i(S_v, S_w) \otimes_{\mathsf k} P_v \otimes_{\mathsf k} P_w^{op} & \text{by Lemma~\ref{lem: Tor is reduced homology}}
\end{align*} 
Since for all $v,w$
\[
Tor_i(S_v, S_w)  = H_i(S_v \otimes_A \mathfrak C_\bullet \otimes_A S_w)
\]
is a free $\mathsf k$-module, this forces the differential on $S_v \otimes_A \mathfrak C_\bullet \otimes_A S_w$ to vanish for all $v,w$.  Hence the differential on $\mathfrak C_\bullet$ is contained in the ideal generated by the arrows i.e. $\mathfrak C_\bullet$ is minimal. 

Conversely, when $A$ is a graded ring, any minimal resolution $\mathfrak{P}_\bullet$ is a summand of the Morse complex $\mathfrak{C}_\bullet$  \cite{Ei56}  (since the Morse complex gives a projective resolution of $A$ by Lemma~\ref{lem: Morse complex is resolution}.) 
On the other hand, the $i^{\text{th}}$ component of the minimal resolution is 
\begin{align*}
  \mathfrak P_i & =  \bigoplus_{v \leq w} Tor_i(S_v, S_w) \otimes_{\mathsf k} P_v \otimes_{\mathsf k} P_w^{op} & \text{see e.g. \cite{BK99}} \\
 & = \bigoplus_{t(p)=v, h(p)=w} \widetilde{H}_{i-2}(K((e_v, p))) \otimes_{\mathsf k} P_v  \otimes_{\mathsf k} P_w^{op} & \text{by Lemma~\ref{lem: Tor is reduced homology}}
\end{align*} 
Hence, the inclusion of $\mathfrak P_i$ into $\mathfrak C_i$ is an isomorphism if and only if for all $p$ the critical $i$-cells of the Babson-Hersh matching form a basis of $\widetilde{H}_{i-2}(K((e_v, p)))$.
\end{proof}

\begin{remark}
The existence of a minimal resolution forces $Tor_i(S_v,S_w)$ to be torsion-free.  Hence by Lemma~\ref{lem: Tor is reduced homology} a minimal resolution cannot exist if there exists a path $p$ such that the reduced homology of the order complex of $(e_{t(p)}, p)$ with coefficients in $\mathsf k$ has torsion.  
\end{remark}

\begin{remark}
The assumption that $A$ is graded above was only used to verify that a minimal resolution exists and is a summand of any other resolution.  The converse above holds whenever this is the case.
\end{remark}

\begin{remark}
We are unsure as to whether or not the Babson-Hersh matching produces the minimal \emph{cellular} resolution in general.
\end{remark}

\begin{proposition}\label{prop: shellable-koszul}
Assume that for all $[p]$, the order complex on the poset $(e_{t(p)},p)$ is either shellable or equal to a finite set of points.  Then, $A$ is Koszul and the Morse complex associated to the Babson-Hersch matching is the minimal projective resolution of $A$ as an $A \otimes_{\mathsf k} A^{op}$-module.

%Assume that for all $[p]$, the regular CW complex $\Link([e_{t(p)}<p])$ is either shellable (\cite{BW97}, Definition 13.1) or equal to a finite set of points. There is a maximal internal acyclic matching whose critical cells are all saturated and give a minimal resolution. In particular, $A$ is Koszul.
\end{proposition}
\begin{proof}
Suppose $(e_{t(p)}<p)$ is a finite set of points.   Match nothing.  Then all cells $[e_{t(p)}<q <p]$ are all saturated by assumption.

Next, suppose $(e_{t(p)}<p)$ is shellable.  Then by \cite{BH05}, Proposition 4.1 all critical cells besides the 0-cell correspond to facets (i.e. saturated chains).

Next, we claim that the minimal resolutions have linear differentials. Since any internal downward path drops an intermediate path, a downward path can not go to a saturated cell at the last step, hence there is no internal gradient paths i.e. all differentials lie in the radical (therefore, the Morse complex is the minimal resolution \cite{Ei56}, Proposition 3, \S 3).  

To see there are no higher order terms in the differentials, let $[p_0=e_{t(p)}<p_1\dots<p_n=p]$ be a critical cell. A downward path will either drop $p_0$ or $p_n$, and we get a saturated cell in either case. But since the matching has to be internal, there is no upward path originating from this cell.  Hence all gradient paths are linear.  This means $A$ is Koszul.
\end{proof}

Now we relate our minimal resolution to our results in the previous section.

\begin{definition}
    Let $A$ be an HPA of Bondal-Ruan type. We say that $A$ is \newterm{directable} if there exists a relabeling of the variables $x_1 < \cdots < x_n$ such that any contatenation of arrows $x_{i_1}\dots x_{i_k}$ is equal in $A$ to the concatenation of arrows $x_{\sigma(i_1)} \cdots x_{\sigma(i_k)}$ with $\sigma(i_1) \leq \cdots \leq \sigma(i_k)$ for some permutation $\sigma$. 
\end{definition}
\begin{remark}
The directablility condition is equivalent to the quadratic condition that there exists a relabeling of the variables such that whenever $x_i x_j$ is a path with $j < i $, $x_ix_j$ is also a path.  %Namely, suppose $x_i x_j$ is a path in $A$ with $j < i$, then, since $A$ is directable, $x_j x_i$ is also a path, and $x_ix_j=x_jx_i$. This means that for any path $x_{i_1}\dots x_{i_k}$, one can interchange the variables in the quadratic monomial $x_{i_j} x_{i_{j+1}}$ as long as $x_{i_j}>x_{i_{j+1}}$. We can then inductively move the maximal remaining variable to the end to get the variables in increasing order. Hence $A$ is quadratic.
\end{remark}

\begin{example}[Non-example]
    We notice that the Hirzebruch surface $\mathbb{F}_1$ is of Bondal-Ruan type. 
\[\begin{tikzcd}
	\bullet & \bullet & \bullet & \bullet
	\arrow["{x_1}", curve={height=-6pt}, from=1-1, to=1-2]
	\arrow["{x_3}", curve={height=6pt}, from=1-1, to=1-2]
	% \arrow["{x_1^2}", curve={height=-12pt}, from=1-2, to=1-3]
	% \arrow["{x_3^2}", curve={height=12pt}, from=1-2, to=1-3]
	\arrow["{x_2}", from=1-2, to=1-3]
	\arrow["{x_1}", curve={height=-6pt}, from=1-3, to=1-4]
	\arrow["{x_3}", curve={height=6pt}, from=1-3, to=1-4]
	\arrow["{x_4}", curve={height=30pt}, from=1-1, to=1-3]
	\arrow["{x_4}", shift left=3, curve={height=-24pt}, from=1-2, to=1-4]
\end{tikzcd}\]
However, this HPA is not directable, since either $x_2x_3$ cannot be changed to $x_3x_2$ (if $x_3<x_2$), or $x_3x_2$ can not be changed to $x_2x_3$ (if $x_2<x_3$). As can be easily observed, this HPA is not quadratic due to the relation $x_1x_2x_3=x_3x_2x_1$.
\end{example}

\begin{corollary}\label{cor: brkoszul}
If $A$ is a directable, then $A$ is Koszul.
\end{corollary}
\begin{proof}

Without loss of generality, one can assume the total ordering on the variables to be $x_1<\dots <x_n$.  This induces an edge labeling on the intervals \cite{Bj83}, Definition 2.1.
% \jh{
% We verify that the edge labeling gives a recursive coatom ordering.
% Suppose $x_{i_1}, \dots x_{i_k}$ are coatoms of $e_v$ in with $x_{i_1}<\dots < x_{i_k}$. 

% Okay this does not work. An increasing path can't be on its own, otherwise it will likely fail the first condition. Need stronger assumption than unique increasing path. 
% }
% \df{We don't need to check it's a recursive coatom ordering.  Just check Definition 2.1 of \cite{Bj83} directly -- it looks completely obvious to me that it satisfies both conditions.  So, our commented sections can just be removed and the proof is fine..}

Then, since EL shellable $\Rightarrow$ CL shellable $\Rightarrow$ shellable \cite{Bj83}, Proposition 2.3, each interval is shellable (unless it has length 2, in which case it is a finite set of points). By Proposition \ref{prop: shellable-koszul}, $A_{\Phi}$ is Koszul.
\end{proof}

We provide a few more toric examples that are not necessarily of Bondal-Ruan type.  Nevertheless, the acyclic matching above agrees with the minimal resolution.

\begin{example}
Let $A$ be the toric FSEC HPA of the quiver of line bundles for the Hirzebruch surface $\mathbb{F}_3$.
\[\begin{tikzcd}
	\bullet & \bullet & \bullet & \bullet
	\arrow["{x_1}", curve={height=-6pt}, from=1-1, to=1-2]
	\arrow["{x_3}", curve={height=6pt}, from=1-1, to=1-2]
	\arrow["{x_1^2}", curve={height=-12pt}, from=1-2, to=1-3]
	\arrow["{x_3^2}", curve={height=12pt}, from=1-2, to=1-3]
	\arrow["x_1x_3", from=1-2, to=1-3]
	\arrow["{x_1}", curve={height=-6pt}, from=1-3, to=1-4]
	\arrow["{x_3}", curve={height=6pt}, from=1-3, to=1-4]
	\arrow["{x_4}", curve={height=30pt}, from=1-1, to=1-3]
	\arrow["{x_4}", shift left=3, curve={height=-24pt}, from=1-2, to=1-4]
\end{tikzcd}\]
Below is an example of maximal internal acyclic matching on $X_A$.

\begin{enumerate}
    \item $M_1^{bottom}$: 1) $[1<x_1^3]$ 2) $[1<x_1^2x_3]$ 3) $[1<x_1x_3^2]$ 4) $[1<x_3^3]$ 5) $[1<x_1x_4]$ 6) $[1<x_3x_4]$ 7) $[1<x_1^4]$ 8) $[1<x_1^3x_3]$ 9) $[1<x_1^2x_3^2]$ 10) $[1<x_1x_3^3]$ 11) $[1<x_4]$ 12) $[x_1<x_1^4]$ 13) $[x_1<x_1^3x_3]$ 14) $[x_1<x_1^2x_3^2]$ 15) $[x_1<x_1x_3^3]$ 16) $[x_3<x_1^3x_3]$ 17) $[x_3<x_1^2x_3^2]$ 18) $[x_3<x_1x_3^3]$ 19) $[x_3<x_3^4]$.\\
    $M_2^{top}$: 1) $[1<x_1<x_1^3]$ 2) $[1<x_1<x_1^2x_3]$ 3) $[1<x_1<x_1x_3^2]$ 4) $[1<x_3<x_3^3]$ 5) $[1<x_1<x_1x_4]$ 6) $1<x_3<x_3x_4$ 7) $[1<x_1<x_1^4]$ 8) $[1<x_1<x_1^3x_3]$ 9) $[1<x_1<x_1^2x_3^2]$ 10) $[1<x_1<x_1x_3^2]$ 11) $[1<x_3<x_3^4]$ 12) $[x_1<x_1^3<x_1^4]$ 13) $[x_1<x_1^3<x_1^3x_3]$ 14) $[x_1<x_1^2x_3<x_1^2x_3^2]$ 15) $[x_1<x_1x_3^2<x_1x_3^3]$ 16) $[x_3<x_1^2x_3<x_1^3x_3]$ 17) $[x_3< x_1^2x_3 <x_1^2x_3^2]$ 18) $[x_3<x_1x_3^2<x_1x_3^3]$ 19) $[x_3<x_3^3<x_3^4]$.
    \item $M_2^{bottom}$: a) $[1<x_3<x_1^3x_3]$ b) $[1<x_3<x_1^2x_3^2]$ c) $[1<x_3<x_1x_3^3]$ d) $[1<x_1^3<x_1^4]$ e) $[1<x_1^3< x_1^3x_3]$ f) $[1<x_1^2x_3<x_1^3x_3]$ g) $[1<x_1^2x_3<x_1^2x_3^2]$ h) $[1<x_1x_3^2<x_1^2x_3^2]$ i) $[1<x_1x_3^2<x_1x_3^3]$ j) $[1<x_3^3<x_1x_3^3]$ k) $[1<x_3^3<x_3^4]$\\
    $M_3^{top}$: a) $[1<x_3<x_1^2x_3<x_1^3x_3]$ b) $[1<x_3<x_1^2x_3<x_1^2x_3^2]$ c) $[1<x_3<x_1x_3^2<x_1x_3^3]$ d) $[1<x_1<x_1^3<x_1^4]$ e) $[1<x_1<x_1^3<x_1^3x_3]$ f) $[1<x_1< x_1^2x_3< x_1^3x_3]$ g) $[1<x_1<x_1^2x_3<x_1^2x_3^2]$ h) $[1<x_1<x_1x_3^2<x_1^2x_3^2]$ i) $[1<x_1<x_1x_3^2<x_1x_3^3]$ j)$[1<x_3<x_3^3<x_1x_3^3]$ k) $[1<x_3<x_3^3<x_3^4]$
\end{enumerate}
The following critical cells of $X_A$ remains after the matching:
\begin{enumerate}
    \item All 4 0-cells and all 9 consecutive 1-cells.
    \item Six 2-cells: $[1<x_3<x_1^2x_3]$, $[1<x_3<x_1x_3^2]$, $[1<x_4<x_1x_4]$, $[1<x_4<x_3x_4]$, $[x_1<x_1^2x_3<x_1^3x_3]\sim [x_3<x_1x_3^2<x_1^2x_3^2]$, $[x_1<x_1x_3^2<x_1^2x_3^2]\sim [x_3<x_3^3<x_1x_3^3]$.
    \item One 3-cell: $[1<x_3<x_1x_3^2<x_1^2x_3^2]$.
\end{enumerate}
We get the 3 dimensional minimal cellular resolution $X_A^{min}$: 
\[
    \includegraphics[scale=1.0]{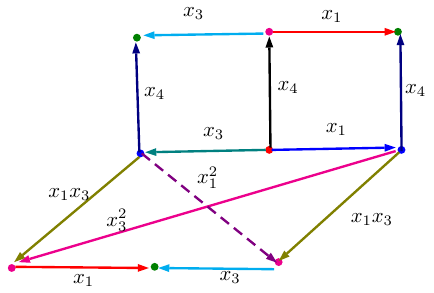}
\]
This is a homotopy torus, where the lower half of the fundamental domain was thickened to a tetrahedron.
\end{example}

This example shows that the minimal bimodule resolution of a toric FSEC HPA, if cellular, can only possibly be a homotopy torus. This is expected from $\mathbb{F}_3$, since the Hochschild dimension(=3) of this FSEC HPA is greater than its dimension(=2). Topologically, it is possible to obtain a CW complex homeomorphic $\mathbb{T}^2$ by matching more. However, such a matching will not be an internal matching, and will no longer provide a bimodule resolution.

\begin{example}
We compare $\mathbb{F}_3$ to the weighted projective stack $\mathbb{P}(1:1:3)$, with the usual FSEC of line bundles:
\[\begin{tikzcd}
	\bullet & \bullet & \bullet & \bullet & \bullet
	\arrow["y"', shift right=1, from=1-1, to=1-2]
	\arrow["x", shift left=1, from=1-1, to=1-2]
	\arrow["y"', shift right=1, from=1-2, to=1-3]
	\arrow["x", shift left=1, from=1-2, to=1-3]
	\arrow["y"', shift right=1, from=1-3, to=1-4]
	\arrow["y"', shift right=1, from=1-4, to=1-5]
	\arrow["x", shift left=1, from=1-3, to=1-4]
	\arrow["x", shift left=1, from=1-4, to=1-5]
	\arrow["z", curve={height=30pt}, from=1-2, to=1-5]
	\arrow["z"', shift left=2, curve={height=-24pt}, from=1-1, to=1-4]
\end{tikzcd}\]

The minimal resolution turns out to be homeomorphic to $\mathbb{T}^2$:
\[
    \includegraphics[scale=1.0]{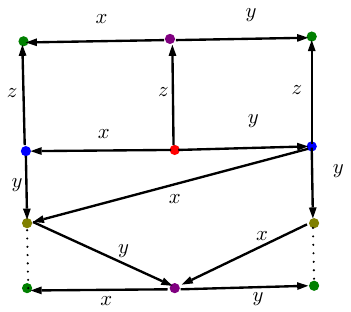}
\]

\end{example}
Note that if one thinks of $\mathbb{P}(1:1:3)$ as obtained from quotient by $\mathbb{C}^*$, the HPA is of Bondal-Ruan type. But if instead it is considered as a VGIT of $\mathbb{F}_3$, thus a quotient by $(\mathbb{C}^*)^2$, the HPA is not of Bondal-Ruan type. There are 6 lattice points in $\im(\Phi)$, but the $\mathbb{P}(1:1:3)$ quiver only has 5 vertices and $\mathbb{F}_3$ quiver has 4.

\begin{remark}
    In fact, one can check by hand that the minimal resolutions for the $\mathbb{P}(1:1:n)$ HPAs are all homeomorphic to $\mathbb{T}^2$.
\end{remark}

\begin{proposition}
Suppose $A$ is a HPA corresponding to an FSEC of line bundles on a smooth variety $X$.  Then $X_A$ has Euler characteristic zero.
\end{proposition}
\begin{proof}
Consider the cellular resolution $\mathsf C_\bullet$ of $A$.  The equivalence $D(A \otimes_{\mathsf k} A^{\op{op}}) \cong D(X \times X)$ sends $\mathsf C_\bullet$ to a locally-free resolution of the diagonal $\mathcal C_\bullet \cong \Delta_X$.  Notice that 
\begin{align*}
\op{rk}(\mathcal C_i) & = \op{rk}(\mathsf C_i) & \text{ since the exceptional collection consists of line bundles}\\
& = \text{ \# of i-cells in }X_A & \text{ by definition of }\mathsf C_\bullet.
\end{align*}
Hence $\chi(\Delta_X) = \chi(X_A)$.  On the other hand, by the HKR theorem $\chi(\Delta_X) =0$.
\end{proof}

We end this paper with the following conjectures.
\begin{conj}
For a toric FSEC HPA $A_{\mathcal{X}}$  of a proper DM toric stack $\mathcal{X}$, $X_{A_\mathcal{X}}$ is homotopic to a real torus $\mathbb{T}$.
\end{conj}

\begin{conj}
Suppose there exists a full strong exceptional collection on a proper DM stack $\mathcal{X}$ of dimension $n$ whose endomorphism algebra $A$ is an HPA. Then there exists a subcomplex $Y \subseteq X_A$ which is homotopic to $\mathbb T^n$.  Furthermore, $\mathcal X$ is rational with local systems on $Y$ corresponding to points on $\mathcal X$.   Finally, $X_{A}$ is homotopic to a torus if and only if $\mathcal X$ is toric. 
\end{conj}

We end with a non-toric example which motivated the conjecture above.

\begin{example}[Homological Berglund-H\"ubsch-Krawitz Mirror Symmetry] \label{ex: BHK}
Let $Z$ by a Calabi-Yau hypersurface $Z$ in a weighted projective space $\mathbb P(a_0:...:a_n)$ defined by an invertible polynomial $w$ and $G_w$ (resp. $\Gamma_w$) be its (resp. extended) maximal symmetry group. Assume that the mirror is Gorenstein (i.e. the dual polynomial $w^T$ has weights $r_i$ such that $r_i$ divides the degree $d^T$).  

Set $k_i = \frac{d^T}{r_i} - 1$ in  Examples~\ref{A_k} and~\ref{A_k CW}. Recall that $A_{k_i}$ denotes the path algebra of corresponding Dynkin quiver. 
 The tensor product was denoted by  
\[
A := A_{k_0} \otimes ... \otimes A_{k_n}
\]
and is the HPA associated to the CW stratification of the cube
\[
X_A^{min} := [0,k_0] \times ... \times [0, k_n].
\]
Fix a homeomorphism transforming the cube $X_A^{min}$ into a smooth $n+1$-disk  with a boundary $n$-sphere 
\[
\Psi: X_A^{min} \to (\mathbb D,\partial \mathbb D).
\]
This induces a stratification $S^{sm}$ of $\mathbb D$ by $\Psi(e_i)$ where $e_i$ is a stratum of $X_A^{min}$.  In addition, observe that $T^*(\mathbb D,\partial \mathbb D)$ is a Liouville sector by \cite{GPS20}, \S 2.5.  We can equip this Loiuville sector with a Lagrangian skeleton by taking the union over the singular support of the constant sheaves extended by zero for all strata:
\[
\Lambda_{S^{sm}} := \bigcup_{\Psi(e_i) \in S^{sm}} ss(j_{i!}\mathsf k_{\Psi(e_i)})
\]
where $j_i: \Psi(e_i)\hookrightarrow \mathbb D$ is the inclusion.
Now we have the following equivalences:
\begin{align*}
D([Z/G_{max}]) & \cong D[\mathbb A^{n+1}, \Gamma_{max} ,W] & \text{ by \cite{BFK}, Theorem 4.2.1, since $Z$ is Calabi-Yau}\\
& \cong D(A) & \text{ by \cite{FKK20}, Theorem 1.1 }\\
& \cong D(Sh_{S^{tr}}(X_A)) & \text{ by Theorem~\ref{thm: main equivalence}} \\
& \cong D(Sh_{S}(X_A^{min})) & \text{ by Corollary~\ref{cor: tree strata} and Example~\ref{A_k} } \\
& \cong D(Sh_{S^{sm}}(\mathbb D, \partial \mathbb D) & \text{ since the smoothing is a homeomorphism}\\
& \cong D\mathcal{W} (T^*(\mathbb D, \partial \mathbb D), \Lambda_{S^{sm}}) & \text{ by  \cite{GPS18}, Theorem 1.1}.
\end{align*}
where $D\mathcal{W}$ refers to the homotopy category of twisted modules over the (partially) wrapped Fukaya category. 
\end{example}


\begin{thebibliography}{HHHH}
\bibitem [1] {Bj83} \emph{A. Bj\"{o}rner and M. Wachs}, On lexicographically shellable posets. Trans. Amer. Math. Soc., \textbf{277} (1983), 323-341.
\bibitem [2] {BDM} \emph{M. Ballard, A. Duncan and P. McFaddin}, Generation and the toric Frobenius, in preparation.
\bibitem[3] {BFK} \emph{M. Ballard, D. Favero and L. Katzarkov} Variation of geometric invariant theory quotients and derived categories. J. Reine Angew. Math.,\textbf{746} (2019), 235-303.
\bibitem [4] {BH05} \emph{E. Babson and P. Hersh} Discrete Morse functions from lexicographic orders. Trans. Amer. Math. Soc., \textbf{357} (2005), 509-534.
\bibitem [5] {BK99} \emph{M.C.R. Butler and  A.D. King} Minimal resolutions of algebras. J. Algebra, \textbf{212} (1999), 323-362.
\bibitem [6] {Bon06} \emph{A. Bondal} Derived categories of toric varieties. Oberwolfach Rep. \textbf{3} (2006), 284-286.
\bibitem [7] {BS98} \emph{D. Bayer and B. Sturmfels} Cellular resolutions of monomial modules. J. Reine Angew. Math., \textbf{502} (1998), 123-140.
\bibitem [8] {BW97} \emph{A. Bj\"{o}rner and M.L. Wachs} Shellable nonpure complexes and posets. II. Trans. Amer. Math. Soc., (1997), 3945-3975.
\bibitem [9] {Cur14} \emph{J.M. Curry} Sheaves, cosheaves and applications. University of Pennsylvania PhD thesis (2014).
\bibitem [10] {CJS95} \emph{R.L. Cohen, J.D. Jones and G.B. Segal}, Morse theory and classifying spaces. (1995)
\bibitem [11] {CP16} \emph{J. Curry and A. Patel} Classification of constructible cosheaves. Theory Appl. Categ. \textbf{35} (2020), 1012–1047.
\bibitem [12] {CQV12} \emph{A. Craw, and A. Quintero V\'elez},  Cellular resolutions of noncommutative toric algebras from superpotentials. Adv. Math. \textbf{229} (2012), 1516-1554.
\bibitem [13] {Del08} \emph{E. Delucchi}, Shelling-type orderings of regular CW-complexes and acyclic matchings of the Salvetti complex. IMRN. \textbf{9} (2008), rnm167.
\bibitem [14] {Ei56} \emph{S. Eilenberg}, Homological dimension and syzygies. Ann. Math. (1956), 328-336.
\bibitem [15] {FKK20} \emph{D. Favero, D. Kaplan, and T. Kelly}, Exceptional collections for mirrors of invertible polynomials. (2020) arXiv preprint arXiv:2001.06500.
\bibitem [16] {FH23} \emph{D. Favero, and J. Huang} Rouquier dimension is Krull dimension for normal toric varieties. European Journal of Mathematics, 9(4), p.91. (2023).
\bibitem [17] {FLTZ11} \emph{B. Fang, C.C.M. Liu, D. Treumann, and E. Zaslow}, A categorification of Morelli’s theorem. Inventiones mathematicae. \textbf{186} (2011), 79-114.
\bibitem [18] {FLTZ12} \emph{B. Fang, C.C.M. Liu, D. Treumann, and E. Zaslow}, T-duality and homological mirror symmetry for toric varieties. Adv. Math. \textbf{229} (2012), 1873-1911.
\bibitem [19] {FLTZ14} \emph{B. Fang, C.C.M. Liu, D. Treumann, and E. Zaslow}, The coherent–constructible correspondence for toric Deligne–Mumford stacks. International Mathematics Research Notices. \textbf{4} (2014), 914-954.
\bibitem [20] {For98} \emph{R. Forman}, Morse theory for cell complexes. Adv Math. \textbf{134} (1998), pp.90-145.
\bibitem [21] {GPS18} \emph{S. Ganatra, J. Pardon, and V. Shende}. Microlocal Morse theory of wrapped Fukaya categories. (2018) arXiv preprint arXiv:1809.08807.
\bibitem [22] {GPS20} \emph{S. Ganatra, , J. Pardon, and V. Shende}, Covariantly functorial wrapped Floer theory on Liouville sectors. Publ. Math. Inst. Hautes Etudes Sci. \textbf{131} (2020), 73-200.
\bibitem [23] {Hat02} \emph{A. Hatcher}, Algebraic topology. Cambridge University Press. (2002)
\bibitem [24] {HHL23} \emph{A. Hanlon, J. Hicks and O. Lazarev}, Resolutions of toric subvarieties by line bundles and applications. Forum of Mathematics, Pi \textbf{12}. (2024)
\bibitem [25] {Iver86} \emph{B. Iversen},  Cohomology of Sheaves. Springer, Berlin, Heidelberg. (1986)
\bibitem [26] {JW09} \emph{M. J\"ollenbeck and V. Welker}, Minimal resolutions via algebraic discrete Morse theory. American Mathematical Soc. (2009)
\bibitem [27] {Kuw20} \emph{T. Kuwagaki}, The nonequivariant coherent-constructible correspondence for toric stacks. Duke Math. J. \textbf{169} (2020), 2125-2197.
\bibitem [28] {Lur09} \emph{J. Lurie},
Derived Algebraic Geometry VI: $E[k]$-Algebras. (2009), Available at \href{http://people.math.harvard.edu/~lurie/papers/DAG-VI.pdf.}{http://people.math.harvard.edu/~lurie/papers/DAG-VI.pdf.}
\bibitem [29]{Mil57} \emph{J. Milnor}, The geometric realization of a semi-simplicial complex. Ann. Math. (1957) 357-362.
\bibitem [30] {Nan19} \emph{V. Nanda}, Discrete Morse theory and localization. J Pure Appl Algebra. \textbf{223} (2019), 459-488.
\bibitem [31] {Nad16} \emph{D. Nadler}, Wrapped microlocal sheaves on pairs of pants. arXiv preprint arXiv:1604.00114.
% \bibitem [N09] {N09} Nadler, D., 2009. Microlocal branes are constructible sheaves. Selecta Mathematica, 15(4), pp.563-619.
% \bibitem [NZ09] {NZ09} Nadler, D. and Zaslow, E., 2009. Constructible sheaves and the Fukaya category. Journal of the American Mathematical Society, 22(1), pp.233-286.
\bibitem [32] {OU13}
\emph{R. Ohkawa and H. Uehara}, Frobenius morphisms and derived categories on two dimensional toric Deligne–Mumford stacks. Adv Math. \textbf{244} (2013), 241-267.
\bibitem [33] {Qui73} \emph{D. Quillen}, Higher algebraic K-theory: I. Higher K-theories. Springer, Berlin, Heidelberg. (1973), 85-147 
\bibitem [34] {Tr09} \emph{D. Treumann}, Exit paths and constructible stacks. Compos. Math. \textbf{145} (2009), 1504-1532.
\bibitem [35] {Ueh14}
\emph{H. Uehara}, Exceptional collections on toric Fano threefolds and birational geometry. Int. J. Math. \textbf{25} (2014), 1450072.
\bibitem [36] {WW82} \emph{K. Watanabe and M. Watanabe}. The classification of Fano 3-folds with torus embeddings. Tokyo J. Math. \textbf{5} (1982), 37-48.
\end{thebibliography}
\end{document}